\documentclass[pdflatex,sn-mathphys-ay]{sn-jnl}

\usepackage{graphicx}%
\usepackage{multirow}%
\usepackage{amsmath,amssymb,amsfonts,}%
\usepackage{amsthm}%
\usepackage{mathrsfs}%
\usepackage[title]{appendix}%
\usepackage{xcolor}%
\usepackage{textcomp}%
\usepackage{manyfoot}%
\usepackage{booktabs}%
\usepackage{algorithm}%
\usepackage{algorithmicx}%
\usepackage{algpseudocode}%
\usepackage{listings}%
\usepackage{bbm,dsfont}
\usepackage{ulem}

\newcommand{\E}{\mathds{E}}
\newcommand{\I}{\mathcal{I}}
\newcommand{\J}{\mathcal{J}}
\renewcommand{\P}{\mathds{P}}
\newcommand{\R}{\mathds{R}}
\newcommand{\C}{\mathds{C}}
\newcommand{\N}{\mathds{N}}
\newcommand{\FO}{L}
\newcommand{\rr}{\mathbf{r}}
\newcommand{\qprob}{\mathbf{p}}
\newcommand{\ind}{\mathbbm{1}}

\numberwithin{equation}{section}
\theoremstyle{plain}
\newtheorem{thm}{Theorem}
\newtheorem{prop}{Proposition}
\newtheorem{corollary}{Corollary}
\newtheorem{lemma}{Lemma}
\newtheorem{ass}{Assumption}
\newtheorem{setting}{Setting}

\theoremstyle{definition}
\newtheorem{algorithmm}{Algorithm}
\newtheorem{definition}{Definition}

\theoremstyle{remark}

\begin{document}

\title[Quantitative limit theorems, bootstrap approximations for empirical spectral projectors]{Quantitative limit theorems and bootstrap approximations for empirical spectral projectors}


\author*[1]{\fnm{Moritz} \sur{Jirak}}
\email{moritz.jirak@univie.ac.at}

\author[2]{\fnm{Martin} \sur{Wahl}}
\email{martin.wahl@math.uni-bielefeld.de}

\affil*[1]{\orgdiv{Department of Statistics and Operations Research}, \orgname{University of Vienna}, \orgaddress{\street{Oskar-Morgenstern-Platz 1}, \city{Vienna}, \postcode{100190}, \country{Austria}}}

\affil[2]{\orgdiv{Fakultät für Mathematik}, \orgname{Universität Bielefeld}, \orgaddress{\street{Postfach 100131}, \city{Bielefeld}, \postcode{33501}, \country{Germany}}}




\abstract{Given finite i.i.d.~samples in a Hilbert space with zero mean and trace-class covariance operator $\Sigma$, the problem of recovering the spectral projectors of $\Sigma$ naturally arises in many applications. In this paper, we consider the problem of finding distributional approximations of the spectral projectors of the empirical covariance operator $\hat \Sigma$, and offer a dimension-free framework where the complexity is characterized by the so-called relative rank of $\Sigma$. In this setting, novel quantitative limit theorems and bootstrap approximations are presented subject to mild conditions in terms of moments and spectral decay. In many cases, these even improve upon existing results in a Gaussian setting.}




\maketitle

\section{Introduction}\label{sec1}

Let $X$ be a random variable in a separable Hilbert space $\mathcal{H}$ with expectation zero and covariance operator $\Sigma = \E X \otimes X$. A fundamental problem in high-dimensional statistics and statistical learning is dimensionality reduction, that is one seeks to reduce the dimension of $X$, while keeping as much information as possible. Letting $(\lambda_j)$ be the sequence of positive eigenvalues of $\Sigma$ (in non-increasing order) and $(u_j)$ be a corresponding orthonormal system of eigenvectors, solutions to this problem are given by the projections $P_{\J}X$ with spectral projector $P_{\J}=\sum_{j\in\J }u_j\otimes u_j$ and index set $\J$.

In statistical applications, the distribution of $X$ and thus its covariance structure are unknown. Instead, one often observes a sample $X_1,\dots,X_n$ of $n$ independent copies of $X$, and the problem now is to find an estimator of $P_{\J}$. The idea of PCA is to solve this problem by first estimating $\Sigma$ by the empirical covariance operator $\hat{\Sigma} = n^{-1}\sum_{i = 1}^n X_i \otimes X_i$, and then constructing the empirical counterpart $\hat P_{\J}$ of $P_{\J}$ based on $\hat \Sigma$ (see Section \ref{sec_further_notation} for a precise definition). Hence, a key problem is to control and quantify the distance between $\hat{P}_{\J}$ and $P_{\J}$.

Over the past decades, an extensive body of literature has evolved around this problem, see e.g. \cite{fan:survey:2021}, \cite{johnstone:paul:2018}, \cite{horvath:kokoszka:book:2012}, \cite{scholkopf}, \cite{jolliffe:jmva:2022} for some overviews. A traditional approach for studying the distance between $\hat{P}_{\J}$ and $P_{\J}$ is to control a norm measuring the distance between the empirical covariance operator $\hat{\Sigma}$ and the population covariance operator $\Sigma$. Once this has been established, one may then deduce bounds for $\hat{P}_{\J} - P_{\J}$ by inequalities such as the Davis-Kahan inequality, see for instance \cite{hsing:eubank:book:2015}, \cite{samworth:pca:2015}, and \cite{cai:wedin:aos:2018}, \cite{jirak:wahl:pams:2020} for some recent results and extensions. However, for a more precise statistical analysis, fluctuation results like limit theorems or bootstrap approximations are much more desirable.

The more recent works by \cite{koltchinskii:lounici:AHIP:2016}, \cite{koltchinskii:lounici:BJ:2017}, \cite{koltchinskii2017} (and related) are of particular interest here. Among other things, they provide the leading order of the expected squared Hilbert-Schmidt distance $\E \|\hat{P}_{\J} - P_{\J}\|_2^2$ and a precise, non-asymptotic analysis of distributional approximations of $\|\hat{P}_{\J} - P_{\J}\|_2^2$ in terms of Berry-Esseen type bounds, given a Gaussian setup. Some extensions and related questions are discussed in \cite{loeffler:2019}, \cite{koltchinskii:jems:2021}, \cite{koltchinski:nickl:loeffler:2020}. However, as noted in \cite{naumov_spokoiny_ptrf2019}, these results have some limitations, and bootstrap approximations may be more desirable and flexible. Again, in a purely Gaussian setup, \cite{naumov_spokoiny_ptrf2019} succeeds in presenting a bootstrap procedure with accompanying bounds that alleviates some of the problems attached to limit distribution for inferential purposes. Let us point out though that, from a mathematical point of view, the results of \cite{koltchinskii2017} and \cite{naumov_spokoiny_ptrf2019} are somewhat complementary. More precisely, there are scenarios where the bound of the bootstrap approximation of Theorem 2.1 in \cite{naumov_spokoiny_ptrf2019} fails (meaning that it only yields a triviality), whereas the bound in Theorem 6 of \cite{koltchinskii2017} does not, and vice versa, see Section \ref{sec:applications:main} for some examples and further discussions. The topic of limit theorems and bootstrap approximations has also been broadly investigated for eigenvalues and related quantities, see for instance \cite{cai:pan:2020}, \cite{miles:2022+}, \cite{miles:blandiono:aue:2019}, \cite{bai:BJ:2021}, \cite{bai:LSS:arxiv}.

The aim of this work is to provide quantitative bounds for both distributional (e.g.~CLTs) and bootstrap approximations, subject to comparatively mild conditions in terms of moments and spectral decay. Concerning the latter, our results display a certain invariance, being largely unaffected by polynomial, exponential (or even faster) decay.

To be more specific, let us briefly discuss the two main previous approaches used in \cite{koltchinskii2017} and  \cite{naumov_spokoiny_ptrf2019}. The Berry-Essen type bounds of \cite{koltchinskii2017} rely on classical perturbation (Neumann) series, an application of the isoperimetric inequality for (Hilbert space-valued) Gaussian random variables, and the classical Berry-Esseen bound for independent real-valued random variables. As a consequence, Theorem 6 in \cite{koltchinskii2017} yields a Berry-Essen type bound for distributional closeness of $\|\hat{P}_{\J} - P_{\J}\|_2^2$ (appropriately normalized) and a standard Gaussian random variable. In contrast, \cite{naumov_spokoiny_ptrf2019} derive novel comparison and anti-concentration results for certain quadratic Gaussian forms, allowing them to directly compare the random variables in question, bypassing any argument involving a limit distribution. They do, however, rely on similar perturbation and concentration results as in \cite{koltchinskii2017}. Finally, let us mention that both only consider the case of single spectral projectors (meaning that $\J$ is the index set of one eigenvalue), and heavily rely on the assumption of Gaussianity for $X$. 

Removing (sub)-Gaussianity also means that the isoperimetric inequality is no longer available, which, however, is the key ingredient of the above approach. Therefore, we opted for a different route, which is more inspired by the relative rank approach developed in \cite{jirak_eigen_expansions_2016}, \cite{JW18}, \cite{jirak:wahl:pams:2020}, \cite{reiss:wahl:aos:2020}, \cite{Wahl}, and consists of the following three key ingredients:
\begin{itemize}
    \item{\textit{Relative perturbation bounds}:} Employing the aforementioned concept of the relative rank, we establish general perturbation results for $\|\hat{P}_{\J} - P_{\J}\|_2$. These results may be of independent interest.

    \item{\textit{Fuk-Nagaev type inequalities}:} We formulate general Fuk-Nagaev type inequalities for independent sums of random operators that allow us to replace the assumption of (sub-)Gaussianity by weaker moment assumptions. Together with the relative perturbation bounds, this leads to approximations for $\|\hat{P}_{\J} - P_{\J}\|_2$ (valid with high probability) in regimes that appeared to be out of reach by previous approaches based on absolute perturbation theory.

    \item{\textit{Bounds for Wasserstein and uniform distance for Hilbert space-valued random variables}:}  The last key point are general bounds for differences of probability measures on Hilbert spaces. Together with the above expansions and concentration inequalities, these permit us to obtain novel quantitative limit theorems and bootstrap approximations. Moment conditions are expressed in terms of the Karhunen–Lo\`{e}ve coefficients of $X$.
\end{itemize}

To briefly demonstrate the flavour of our results, let us present two examples where the eigenvalues $\lambda_j$ of $\Sigma$ have a rather opposing behaviour. The first result deals with polynomially decaying eigenvalues, a common situation in statistics and machine learning (cf.~\cite{bartlett:benign:2020}, \cite{steinwart:jml:2020}, \cite{hall2007}).
For single projectors ($\J = \{J\}$), subject to mild moment conditions on the Karhunen–Lo\`{e}ve coefficients of $X$, we show that
\begin{align*}
\sup_{x \in\R}  \Big\vert \P \Big(n\big\|\hat{P}_{\J} - P_{\J}\big\|_2^2 \leq x\Big) - \P\Big(\big\|L_{\J}Z\big\|_{2}^2 \leq x\Big) \Big\vert \lesssim \frac{J}{n^{1/2}} (\log J \log n)^{3/2}
\end{align*}
for all $J\geq 1$, where $L_{\J}Z$ is an appropriate Gaussian random variable. Key points here are the mild moment conditions and the explicit dependence on $J$, which is optimal up to $\log$-factors. This result, among others, is novel even for the Gaussian case.
For a precise statement, see Section \ref{sec:app:model:I:poly}.

The second case deals with a pervasive factor like structure, a popular model in econometrics and finance (cf.~\cite{fan:survey:2021}, \cite{bai:2003:econometrica}, \cite{stock:watson:2002}). 
In this case, again subject to mild moment conditions for the Karhunen–Lo\`{e}ve coefficients of $X$, we have
\begin{align*}
\sup_{x \in \R}\Big\vert\P\Big(n\big\|\hat{P}_{\J} - P_{\J}\big\|_2^2 \leq x \Big) - \P\Big( \big\|L_{\J}Z\big\|_{2}^2 \leq x\Big)\Big\vert\lesssim \frac{J^{3}}{n^{1/2}} \Big((\log n)^{3/2} +  J^{5/2}\Big)
\end{align*}
for all $J\geq 1$, where $L_{\J}Z$ is an appropriate Gaussian random variable. As before, we are not aware of a similar result even in a Gaussian setup. For a precise statement, see Section \ref{sec:app:model:II:pervasive}.

This work is structured as follows. We first introduce some notation and establish a number of preliminary results in Sections \ref{sec:notation}--\ref{sec:tools:probability}. Bounds for quantitative limit theorems are then presented in Section \ref{sec:quantitative:limit:thms} for both the uniform and the Wasserstein distance, whereas Section \ref{sec:main:boot} contains accompanying results for a suitable bootstrap approximation. Finally, we discuss some key models from the literature and provide a comparison to previous results in Section \ref{sec:applications:main}. Proofs are given in the remaining Section~\ref{sec:proofs}.


\section{Preliminaries}

\subsection{Basic notation}\label{sec:notation}
We write $\lesssim$, $\gtrsim$ and $\asymp$ to denote (two-sided) inequalities involving a multiplicative constant. If the involved constant depends on some parameters, say $p$, then we write $\lesssim_p$, $\gtrsim_p$ and $\asymp_p$. For $a,b\in \R$, we write $a \wedge b = \min(a,b)$ and $a \vee b = \max(a,b)$. Given a subset $\J$ of the index set (in most cases $\N$), we denote with $\J^c$ its complement in the index set. Given a (real) Hilbert space $\mathcal{H}$, we always write $\|\cdot\| = \|\cdot\|_{\mathcal{H}}$ for the corresponding norm. For a bounded linear operator $A$ and $q\in[1,\infty]$, we denote with $\|A\|_q$ the Schatten $q$-norm. In particular, $\|A\|_2$ denotes the Hilbert-Schmidt norm, $\|A\|_1$ the nuclear norm and $\|A\|_{\infty}$ the operator norm. For a self-adjoint compact operator $A$ mapping $\mathcal{H}$ into itself, we write $|A|$ for the absolute value of $A$ and $|A|^{1/2}$ for the positive self-adjoint square-root of $|A|$. For a random variable $X$, we write $\overline{X} = X - \E X$.

\subsection{Tools from perturbation theory}\label{sec_preliminaries_perturbation}
In this section, we present our underlying perturbation bounds. Instead of applying standard perturbation theory, we make use of recent improvements from \cite{jirak:wahl:pams:2020}, \cite{Wahl} adapted for our purposes. Proofs are presented in Section~\ref{sec_preliminaries_perturbation}.

\subsubsection{Further notation}\label{sec_further_notation}
Throughout this section, $\Sigma$ denotes a positive self-adjoint compact operator on the separable Hilbert space $\mathcal{H}$ (in applications, it will be the covariance operator of the random vector $X$ with values in $\mathcal{H}$). By the spectral theorem there exists a sequence $\lambda_1\ge \lambda_2\ge\dots>0$ of  positive eigenvalues (which is either finite or converges to zero) together with an orthonormal system of eigenvectors $u_1,u_2,\dots$ such that $\Sigma$ has
spectral representation
\begin{equation*}
\Sigma=\sum_{j\ge 1}\lambda_j P_j,
\end{equation*}
with rank-one projectors $P_j=u_j\otimes u_j$, where $(u\otimes v)x=\langle v,x\rangle u$, $x\in \mathcal{H}$. Without loss of generality we shall assume that the eigenvectors $u_1,u_2,\dots$ form an orthonormal basis of $\mathcal{H}$ such that $\sum_{j\ge 1}P_j=I$. For $1\leq j_1\leq j_2\leq \infty$, we consider an interval of the form ${\J}=\{j_1,\dots,j_2\}$ (with $\J=\{j_1,j_1+1,\dots\}$ if $j_2=\infty$). We write
\[
P_{\J}=\sum_{j\in {\J}}P_j,\qquad P_{{\J}^c}=\sum_{k\in {\J}^c}P_k
\]
for the orthogonal projection on the direct sum of the eigenspaces
of $\Sigma$ corresponding to the eigenvalues $\lambda_{j_1},\dots,\lambda_{j_2}$, and onto its orthogonal complement. Moreover, let
\begin{align*}
   R_{{\J}^c}=\sum_{k< j_1}\frac{P_k}{\lambda_k-\lambda_{j_1}}+\sum_{k> j_2}\frac{P_k}{\lambda_k-\lambda_{j_2}}
\end{align*}
be the reduced `outer' resolvent of ${\J}$. Finally, let $g_\J$ be the gap between the eigenvalues with indices in ${\J}$ and the eigenvalues with indices not in ${\J}$ defined by
\begin{align*}
    g_{\J}=\min(\lambda_{j_1-1}-\lambda_{j_1},\lambda_{j_2}-\lambda_{j_2+1})
\end{align*}
for $j_1>1$ and $j_2<\operatorname{dim} \mathcal{H}$ and with corresponding changes otherwise. If $\J= \{j\}$ is a singleton, then we also write $g_j$ instead of $ g_{\J}$.

Let $\hat\Sigma$ be another positive self-adjoint compact operator on $\mathcal{H}$ (in applications, it will be the empirical covariance operator). We consider $\hat\Sigma$ as a perturbed version of $\Sigma$ and write $E = \hat\Sigma-\Sigma$ for the (additive) perturbation. Again, there exists a sequence $\hat{\lambda}_1\ge \hat{\lambda}_2\ge\dots\ge 0$ of eigenvalues together with an orthonormal system of eigenvectors $\hat{u}_1,\hat{u}_2,\dots$ such that we can write
\begin{equation*}
\hat{\Sigma}=\sum_{j\ge 1}\hat{\lambda}_j \hat{P}_j
\end{equation*}
with $\hat{P}_j=\hat{u}_j\otimes\hat{u}_j$. We shall also assume that the eigenvectors $\hat u_1,\hat u_2,\dots$ form an orthonormal basis of $\mathcal{H}$ such that $\sum_{j\ge 1}\hat P_j=I$.
We write
\[
\hat P_{\J}=\sum_{j\in {\J}}\hat P_j,\qquad \hat P_{{\J}^c}=\sum_{k\in {\J}^c} \hat P_k
\]
for the orthogonal projection on the direct sum of the eigenspaces
of $\hat \Sigma$ corresponding to the eigenvalues $\hat\lambda_{j_1},\dots,\hat\lambda_{j_2}$, and onto its orthogonal complement.

\subsubsection{Main perturbation bounds}
Let us now present our main perturbation bound. The following quantity  will play a crucial role
\begin{align}\label{defn:delta_J}
\delta_\J=\delta_\J(E):=\left\|\left(\vert R_{\J^c}\vert ^{1/2}+g_\J^{-1/2}P_\J\right)E\left(\vert R_{\J^c}\vert ^{1/2}+g_\J^{-1/2}P_\J\right)\right\|_\infty,
\end{align}
where
\begin{align*}
    \vert R_{\J^c}\vert ^{1/2}=\sum_{k<j_1}\frac{P_k}{\vert \lambda_k-\lambda_{j_1}\vert ^{1/2}}+\sum_{k>j_2}\frac{P_k}{\vert \lambda_k-\lambda_{j_2}\vert ^{1/2}}
\end{align*}
is the positive self-adjoint square-root of $\vert R_{\J^c}\vert $. In the special case that ${\J}=\{j\}$ is a singleton, it has been introduced in \cite{Wahl}.
Moreover, for a Hilbert-Schmidt operator $A$ on $\mathcal{H}$ we define
\begin{align}\label{eq_def_linear_term}
L_{\J} A = \sum_{j \in \J}\sum_{k \in \J^c} \frac{1}{\lambda_j - \lambda_k}(P_k AP_j+P_j A P_k).
\end{align}
The following result presents a linear perturbation expansion with remainder expressed in terms of $\delta_{\J}(E)$.

\begin{prop}\label{PropExp}
We have
\begin{align}\label{Eq0Exp}
\|P_{\J}-\hat P_\J\|_2\leq 4\sqrt{2}\min(\vert \J\vert ,\vert \J^c\vert )^{1/2}\delta_{\J}
\end{align}
and
\begin{align}\label{Eq2Exp}
\|\hat{P}_{\J} - P_{\J} - L_{\J}E\|_2  \leq  20\sqrt{2}\min(\vert \J\vert ,\vert \J^c\vert ) \delta_{\J}^2.
\end{align}
\end{prop}
Using the identities $P_{\J}-\hat P_\J=\hat P_{\J^c}-P_{\J^c}$ and $L_{\J}E=-L_{\J^c}E$, it is possible to replace $\delta_{\J}$ by $\min(\delta_{\J},\delta_{\J^c})$ in Proposition \ref{PropExp}, at least for $j_1=1$. In the latter case, we implicitly assume that $\J$ is chosen such that $\delta_{\J}$ is equal to $\min(\delta_{\J},\delta_{\J^c})$ in what follows.

Clearly, we can bound $\delta_{\J}\leq \|E\|_\infty/g_{\J}$, in which case the inequality \eqref{Eq0Exp} turns into a standard perturbation bound, cf.~Lemma 2 in \cite{jirak:wahl:pams:2020} or Equation (5.17) in \cite{reiss:wahl:aos:2020}.

\begin{corollary}\label{cor:exp:quadratic:term}
We have
\begin{align*}
  &\big\vert\|P_{\J}-\hat P_\J\|_2^2-\|L_{\J}E\|_2^2\big\vert\lesssim \min(\vert \J\vert ,\vert \J^c\vert )^{3/2}\delta_\J^3+\min(\vert \J\vert ,\vert \J^c\vert )^2\delta_\J^4.
\end{align*}
\end{corollary}

For our next consequence (used in the bootstrap approximations), let $\tilde \Sigma$ be a third positive self-adjoint compact operator on $\mathcal{H}$. Similarly as above we will write $\tilde E=\tilde \Sigma-\Sigma$ and $\tilde P_{\J}$ for the orthogonal projection on the direct sum of the eigenspaces of $\tilde \Sigma$ corresponding to the eigenvalues with indices in $\J$. 
\begin{corollary}\label{lem:boot:approx}
We have
\begin{align*}
&\big\vert\|\tilde{P}_{\J} - \hat{P}_{\J} \|_2^2 - \|\FO_{\J}(\tilde{E} - E)\|_2^2\big\vert\\
&\lesssim \min(\vert \J\vert ,\vert \J^c\vert )^{3/2}(\delta_\J^3(E)+\delta_\J^3(\tilde E))+\min(\vert \J\vert ,\vert \J^c\vert )^2(\delta_\J^4(E)+\delta_\J^4(\tilde E)).
\end{align*}
\end{corollary}

The dependence on $\min(\vert {\J}\vert ,\vert {\J}^c\vert )$ in \eqref{Eq2Exp} and Corollaries \ref{cor:exp:quadratic:term} and \ref{lem:boot:approx} can be further improved, by using the first inequality in Lemma \ref{LemFirstOrderExp} below (instead of the second one). Since this leads to additional quantities that have to be controlled, such improvements are not pursued here.

\subsection{Tools from probability}\label{sec:tools:probability}
\subsubsection{Metrics for convergence of laws}\label{sec:metrics}

In this section, we provide several upper bounds for the uniform metric and the $1$-Wasserstein metric. The propositions are proved in Section \ref{sec:metrics:proofs}. For two real-valued random variables $X,Y$, the uniform (or Kolmogorov) metric is defined by
\begin{align*}
\mathbf{U}\big(X,Y\big) = \sup_{x \in \R}\big\vert\P\big(X \leq x\big) - \P\big(Y \leq x\big)\big\vert.
\end{align*}
Moreover, for two (induced) probability measures $\P_X,\P_Y$, let $\mathcal{L}(\P_X,\P_Y)$ be the set of probability with marginals $\P_X, \P_Y$. Then the $1$-Wasserstein metric is defined as the minimal coupling in $L^1$-distance, that is,

\begin{align*}
\mathbf{W}\big(X,Y\big) = \inf\Big\{ \int_{\R} \vert x-y\vert\, \P(dx, dy): \, \P \in \mathcal{L}(\P_X, \P_Y)\Big\}.
\end{align*}

\begin{setting}\label{ass:clts:setup}
Let $T$ be a real-valued random variable. Moreover, let $Y_1,\dots,Y_n$ be independent random variables taking values in a separable Hilbert space $\mathcal{H}$ satisfying $\E Y_i=0$ and $\E \|Y_i\|^q < \infty$ for some $q > 2$. Set
\begin{align*}
&S = n^{-1}\|\sum_{i = 1}^n Y_i\|^2 \quad\text{and}\quad \Psi = \frac{1}{n}\sum_{i = 1}^n \E Y_i \otimes Y_i.
\end{align*}
Let $\lambda_1(\Psi)\geq \lambda_2(\Psi)\geq \dots >0$ be the positive eigenvalues of $\Psi$ (each repeated a number of times equal to its multiplicity). Moreover, let $Z$ be a centered Gaussian random variable in $\mathcal{H}$ with covariance operator $\Psi$, that is a Gaussian random variable with $\E Z = 0$ and $\E Z \otimes Z=n^{-1}\sum_{i = 1}^n \E Y_i \otimes Y_i$.

\end{setting}

In applications $T$ and $S$ will correspond to (scaled or bootstrap versions of) $\|\hat{P}_{\J} - P_{\J}\|_2^2$ and its approximation $\|L_{\J}E\|_2^2$, respectively.

Our first result deals with the $1$-Wasserstein distance between $T$ and $\|Z\|^2$.

\begin{prop}\label{prop:clt:II}
Consider Setting \ref{ass:clts:setup} with $q \in (2,3]$. Assume $\vert T\vert\leq C_T$ almost surely, and $\sum_{i = 1}^n \E \|Y_i\|^r \leq n^{r/2}$ for $r \in \{2,q\}$. Then, for all $s > 0$ and $u>0$,
\begin{align*}
\mathbf{W}\big(T, \|Z\|^2\big) \lesssim_{q,s}\Big(n^{-q/2}\sum_{i = 1}^n\E \|Y_i\|^q\Big)^{\frac{q-2}{2(2q-3)}} + C_T\P(\vert T -S\vert  > u)+ u  + \frac{1}{C_T^{s}}.
\end{align*}
\end{prop}
In connection to the uniform metric, the following quantities and relations will be crucial:
\begin{align}\label{Psi:clt:identities}
A := & \E \|Z\|^2 = \E S=\|\Psi\|_1 , \nonumber\\
B := &\E^{1/2} \big(\|Z\|^2 - \|\Psi\|_1\big)^2=\sqrt{2} \|\Psi\|_2, \nonumber \\
C := & \E^{1/3} \big(\|Z\|^2 - \|\Psi\|_1\big)^3 =2\|\Psi\|_3.
\end{align}
where we also used that $\E(G^2-1)^2=2$ and $\E (G^2-1)^3=8$ for $G \sim\mathcal{N}(0,1)$. By monotonicity of the Schatten norms, we have $C\lesssim B\lesssim A$.

\begin{prop}\label{prop:clt:V}
Consider Setting \ref{ass:clts:setup} with $q =3$. Then, for all $u > 0$,
\begin{align*}
\mathbf{U}\big(T, \|Z^{}\|^2 \big)  &\lesssim_p  n^{-3/5} \Big(\sum_{i = 1}^n \frac{\E \|Y_i\|^3}{\lambda_{1,2}^{3/4}(\Psi)} \Big)^{2/5} + \P\big(\vert T-S \vert> u\big) + \frac{u}{\sqrt{\lambda_{1,2}(\Psi)}},
\end{align*}
as well as
\begin{align*}
\mathbf{U}\big(T, \|Z^{}\|^2 \big)  &\lesssim_p
\Big( n^{-1/2}\lambda_6^{-3}(\Psi)  \frac{(A C)^3}{B^3}  \Big)^{1 + 1/10}+n^{-1/2}\lambda_{1,6}^{-1}(\Psi) \frac{(A^2 C)^3}{B^3}\\
&+ \P\big(\vert T-S \vert> u\big) + \frac{u}{\sqrt{\lambda_{1,2}(\Psi)}},
\end{align*}
with
\begin{align}\label{defn:lambda:product}
\lambda_{1,j}(\Psi)=\prod_{i = 1}^j\lambda_i(\Psi) \quad \text{for $j\geq 1$.}
\end{align}
\end{prop}

The appearance of both $\lambda_6(\Psi)$ and $\lambda_{1,6}(\Psi)$ is, in fact, unavoidable, we refer to \cite{SAZONOV1989304} for more details.

Let $G \sim\mathcal{N}(0,1)$, Then, using invariance properties of the uniform metric, Lemma \ref{lem:BE:gauss} below and \eqref{Psi:clt:identities}, we get
\begin{align*}
\mathbf{U}\Big(\frac{T - A}{B}, G\Big)&\lesssim_p \Big(\frac{C}{B}\Big)^3 + \text{bound of Proposition \ref{prop:clt:V}}.
\end{align*}
However, in this case one may dispose of any conditions on the eigenvalues of $\Psi$ altogether, as is demonstrated by our next result below.

\begin{prop}\label{prop:clt:III}
Consider Setting \ref{ass:clts:setup} with $q \in (2,3]$. Then, for all $u > 0$,
\begin{align*}
\mathbf{U}\Big(\frac{T - A}{B }, G\Big) \lesssim_q \ n^{-q/8}\Big(\frac{A^3}{B^3}\sum_{i = 1}^n \frac{\E \|Y_i\|^q}{A^{q/2} }\Big)^{1/4} +  \Big(\frac{C}{B}\Big)^3 + \P\big(\vert T-S \vert> u\big) + \frac{u}{B}.
\end{align*}
\end{prop}

Propositions \ref{prop:clt:II}, \ref{prop:clt:V} and \ref{prop:clt:III} are based on various normal approximation bounds in the literature, we refer to the proofs for references. Finding (or confirming) the optimal dependence on the dimension and the underlying covariance structure in such bounds is still an open problem in general.

\subsubsection{Concentration inequalities}\label{sec:concentration:inequalities}
In this section, we recall several useful concentration inequalities that we will need in the proofs below. Additional results are given in Section~\ref{sec:proof:concentration:inequalities}.

We will make frequent use of the classical Fuk-Nagaev inequality for real-valued random variables (cf.~\cite{nagaev}).

\begin{lemma}
Let $Z_1,\dots, Z_n$ be independent real-valued random variables. Suppose that $\E Z_i=0$ and $\E \vert Z_i\vert^p<\infty$ for some $p > 2$ and all $i=1,\dots,n$. Then, it holds that
\begin{align*}
\P\Big(\big\vert\sum_{i = 1}^n Z_i \big\vert \geq t \Big) \leq \Big(2 + \frac{2}{p}\Big)^p \frac{\mu_{n,p}}{t^p} + 2 \exp\Big(-\frac{a_p t^2}{\mu_{n,2}} \Big)
\end{align*}
for all $t>0$, with $a_p = 2 e^{-p}(p+2)^{-2}$ and $\mu_{n,p} = \sum_{i = 1}^n \E \vert Z_i\vert^p$.
\end{lemma}
In particular, if $Z_1,\dots,Z_n$ are independent copies of a real-valued random variable $Z$ with $\E Z=0$ and  $\E \vert Z \vert^p<\infty$, then there is a constant $C>0$ such that
\begin{align}\label{lem:fn}
\P\Big(\big\vert\sum_{i = 1}^n Z_i \big\vert \geq C(\E \vert Z\vert^p)^{1/p}\sqrt{n\log n} \Big) \leq n^{1-p/2}.
\end{align}

We will also apply Banach space versions of the Fuk-Nagaev inequality. The following Lemma collects concentration inequalities in Hilbert-Schmidt norm, in nuclear norm, and in operator norm.

\begin{lemma}\label{lem:fuknagaev:HS:nuclear:operator}
Let $Y=\sum_{j\geq 1}\vartheta_j^{1/2}\zeta_ju_j$ be a random variable taking values in a separable Hilbert space $\mathcal{H}$, where $\vartheta = (\vartheta_1,\vartheta_2,\dots)$ is a sequence of positive numbers with $\|\vartheta\|_1:=\sum_{j\geq 1}\vartheta_j<\infty$, $u_1,u_2,\dots$ is an orthonormal system in $\mathcal{H}$, and $\zeta_1,\zeta_2,\dots$ is a sequence of centered random variables satisfying $\sup_{j\geq 1}\E\vert \zeta_{j}\vert ^{2p} \leq C_1$ for some $p>2$ and $C_1 > 0$. Let $Y_1,\dots,Y_n$ be $n$ independent copies of $Y$. Then there exists a constant $C_2 > 0$ depending only on $C_1$ and $p$ such that the following holds.
\begin{description}
    \item[(i)] For all $t \geq 1$, it holds that
    \begin{align*}
    \P\Big(\big\|\frac{1}{n}\sum_{i=1}^nY_i\otimes Y_i-\E Y\otimes Y\big\|_2 > \frac{C_2t\|\vartheta\|_1}{\sqrt{n}} \Big) \leq \frac{n^{1-p/2}}{t^{p}} + e^{-t^2}.
    \end{align*}
    \item[(ii)] Let $\varpi_n^2 \geq \sup_{\|S\|_{\infty} \leq  1} \E \operatorname{tr}^2(S (Y \otimes Y-\Sigma))$. Then, for all $t\geq \|(\E (Y\otimes Y-\Sigma)^2 )^{1/2}\|_1$, it holds that
    \begin{align*}
    \P\Big(\big\|\frac{1}{n}\sum_{i=1}^nY_i\otimes Y_i-\E Y\otimes Y\big\|_1 \geq \frac{C_2t}{\sqrt{n}}\Big) \leq \frac{n^{1-p/2}}{t^p} \|\vartheta\|_1^p    + e^{-t^2/\varpi_n^2}.
    \end{align*}
    \item[(iii)] Assume additionally that  $\|\E (Y \otimes Y-\E Y\otimes Y)^2\|_{\infty} =  1$. Then, for all $t \geq 1$, it holds that
    \begin{align*}
   \P\Big( \big\|\frac{1}{n}\sum_{i=1}^nY_i\otimes Y_i-\E Y\otimes Y\big\|_\infty\geq  \frac{C_2t}{\sqrt{n}}\Big)\leq C_2n^{1-p/2} t^{p}  \|\vartheta\|_1^p +  \|\vartheta\|_1^2 e^{-t^2}.
    \end{align*}
\end{description}
\end{lemma}

Lemma \ref{lem:fuknagaev:HS:nuclear:operator}(iii) is not optimal in terms of $t$ but sufficient for the choice $t\asymp \sqrt{\log n}$.

An important special case is given if the expansion for $Y$ in Lemma~\ref{lem:fuknagaev:HS:nuclear:operator} coincides with the Karhunen-Loève expansion. For this, suppose that $Y$ is centered and strongly square-integrable, meaning that $\mathbb{E} Y =0$ and $\mathbb{E}\|Y\|^2<\infty$. Let $\Sigma =\mathbb{E} Y\otimes Y$ be the covariance operator of $Y$, which is a positive, self-adjoint trace class operator, see e.g. \cite[Theorem 7.2.5]{hsing:eubank:book:2015}. Let $Y_1,\dots,Y_n$ be independent copies of $Y$ and let $\hat{\Sigma}=n^{-1}\sum_{i=1}^nY_i\otimes Y_i$ be the empirical covariance operator. Let  $\lambda_1,\lambda_2,\dots$ and $u_1,u_2,\dots$ be the eigenvalues and eigenvectors of $\Sigma$ as introduced in Section \ref{sec:notation}. Then we can write $Y=\sum_{j\geq 1}\lambda_j^{1/2}\eta_ju_j$ almost surely, where $\eta_j=\lambda_j^{-1/2}\langle u_j, Y\rangle$ are the Karhunen-Lo\`{e}ve coefficients of $Y$. By construction, the $\eta_j$ are centered, uncorrelated and satisfy $\E\eta_j^2=1$. Now, if the Karhunen-Loève coefficients satisfy $\sup_{j\geq 1}\E\vert \eta_{j}\vert ^{2p} \leq  C_1$ for some $p>2$ and $C_1 > 0$, then Lemma~\ref{lem:fuknagaev:HS:nuclear:operator}(i) implies
\begin{align*}
    \P\big(\|\hat\Sigma-\Sigma\|_2 > C_2\operatorname{tr}(\Sigma)t/\sqrt{n}\big) \leq n^{1-p/2}t^{-p} +  e^{-t^2}
\end{align*}
for all $t\geq 1$, while Lemma \ref{lem:fuknagaev:HS:nuclear:operator}(ii) implies
\begin{align}\label{cor:fn:nuclear}
\P(\|\hat\Sigma-\Sigma\|_1 \geq C_2 t/\sqrt{n}) \leq n^{1-p/2}t^{-p} \operatorname{tr}^p(\Sigma)     + e^{-t^2/\varpi_n^2}
\end{align}
for all $t\geq \|(\E (Y\otimes Y-\Sigma)^2 )^{1/2}\|_1$, where $\varpi_n^2 \geq \sup_{\|S\|_{\infty} \leq  1} \E \operatorname{tr}^2(S (Y \otimes Y-\Sigma))$.
If additionally $\|\E (Y \otimes Y-\Sigma)^2\|_{\infty}=  1$ holds, then Lemma \ref{lem:fuknagaev:HS:nuclear:operator}(iii) implies that, for all $t\geq 1$,
\begin{align}\label{cor:fn}
    \P\big(\|\hat\Sigma-\Sigma\|_{\infty} \geq C_2t/\sqrt{n}  \big)  \leq n^{1-p/2} t^{p} \operatorname{tr}^{p}(\Sigma) +  \operatorname{tr}^2(\Sigma) e^{-t^2}.
\end{align}

In order to apply the above concentration inequalites, we will make frequent use of the following moment computations based on the Karhunen-Loève coefficients.

\begin{lemma}\label{lem:moment:computation:KL}
Let $Y=\sum_{j\geq 1}\vartheta_j^{1/2}\zeta_ju_j$ be an $\mathcal{H}$-valued random variable, where $\vartheta = (\vartheta_1,\vartheta_2,\dots)$ is a sequence of positive numbers with $\|\vartheta\|_1<\infty$, $u_1,u_2,\dots$ is an orthonormal system in $\mathcal{H}$, and $\zeta_1,\zeta_2,\dots$ is a sequence of centered random variables satisfying $\sup_{j\geq 1}\E\vert \zeta_{j}\vert ^{2p} \leq  C$ for some $p\geq 2$ and $C > 0$. Then
\begin{description}
\item[(i)] {$\E\|Y\|^{2r}\leq C^{r/p} \|\vartheta\|_1^r$ for all $1\leq r\leq p$}.
\item[(ii)] $\operatorname{tr}(\mathbb{E}(Y\otimes Y-\E Y\otimes Y)^2)\leq C^{2/p}\|\vartheta\|_1^2$.
\end{description}
Suppose additionally that the $\zeta_j$ are uncorrelated and satisfy $\E \zeta_j \zeta_{k}^2 \zeta_{s} = 0$ for all $j,k,s$ such that $j\neq s$. Then
\begin{description}
\item[(iii)] $\|\E(Y\otimes Y-\E Y\otimes Y)^2\|_\infty\leq C^{2/p} \|\vartheta\|_\infty\|\vartheta\|_1$.
\item[(iv)] $\|(\E (Y\otimes Y-\E Y\otimes Y)^2 )^{1/2}\|_1\leq C^{1/p} (\sum_{j\geq 1}\vartheta_j^{1/2})\|\vartheta\|_1^{1/2}$.
\end{description}
In (i), it suffices to assume that $p\geq 1$.
\end{lemma}


\section{Quantitative limit theorems}\label{sec:quantitative:limit:thms}

\subsection{Assumptions and main quantities}

Throughout this section, let $X$ be a random variable taking values in $\mathcal{H}$. We suppose that $X$ is centered and strongly square-integrable, meaning that $\mathbb{E} X =0$ and $\mathbb{E}\|X\|^2<\infty$. Let $\Sigma =\mathbb{E} X\otimes X$ be the covariance operator of $X$, which is a positive self-adjoint trace class operator. Let $X_1,\dots,X_n$ be independent copies of $X$ and let
\begin{equation*}
\hat{\Sigma}=\frac{1}{n}\sum_{i=1}^nX_i\otimes X_i
\end{equation*}
be the empirical covariance operator.

Our main assumptions are expressed in terms of the Karhunen-Loève coefficients of $X$, defined by $\eta_{j} = \lambda_j^{-1/2}\langle X, u_j \rangle$ for $j\geq 1$. These lead to the Karhunen-Loève expansion $X=\sum_{j\geq 1}\sqrt{\lambda_j}\eta_j u_j$ almost surely. If $X$ is Gaussian, then the $\eta_j$ are independent and standard normal. Non-Gaussian examples are given by elliptical models and factor models, see for instance \cite{Lopes} and \cite{JW18} for concrete computations or \cite{siegi:2015} and \cite{PANARETOS20132779} for the context of functional data analysis. 

\begin{ass}\label{ass_moments}
Suppose that for some $p > 2$
\begin{align*}
\sup_{j\geq 1}\E\vert \eta_{j}\vert ^{2p} \leq C_{\eta}.
\end{align*}
\end{ass}

The actual condition on the number of moments $p$ will be mild, and depends on the desired rate of convergence, we refer to our results for more details.
Apart from a relative rank condition, Assumption \ref{ass_moments} is essentially all we need. In order to simplify the bounds (and examples), we also demand a non degeneracy condition.

\begin{ass}\label{ass:lower:bound}
There is a constant $c_\eta>0$, such that for every $j\neq k$,
\begin{align*}
\E \eta_j^2 \eta_k^2 \geq c_{\eta}.
\end{align*}
\end{ass}

In the special case of bootstrap approximations, the situation is more complicated, and thus our next condition is more restrictive. It allows us to explicitly compute moments of certain, rather complicated random variables in connection with some of our bounds, see Lemma \ref{lem:Lambda:Sigma:relation} and Section \ref{sec:applications:main} below for more details.

\begin{ass}[$m$-th cumulant uncorrelatedness]\label{ass:indep:4:moments}
If any of the indices $i_1, \ldots i_m \in \N$ is unequal to all the others, then
\begin{align*}
\E \eta_{i_1} \eta_{i_2} \ldots \eta_{i_m} = 0.
\end{align*}
\end{ass}

Let us point out that both Assumptions \ref{ass:lower:bound} and \ref{ass:indep:4:moments} are trivially true if the sequence $(\eta_k)$ is independent. A more general example is, for instance, provided by martingale type differences $\eta_k = \epsilon_k v_k$ with independent sequences $(\epsilon_k)$ and $(v_k)$, where the $(\epsilon_k)$ are mutually independent. In this case, Assumption \ref{ass:indep:4:moments} holds given the existence of $m$ moments. Apart
from results concerning the bootstrap, Assumption \ref{ass:indep:4:moments} is quite convenient for discussing our examples in Section \ref{sec:applications:main}, allowing us in particular to relate the variance-type term $\sigma_{\J}^2$ (cf.~Definition \ref{defn:sigma} below) to higher order cumulants. This leads to very simple and explicit bounds, mirroring previous Gaussian results, and is in line with our recent findings in \cite{JW18}, where it is shown that moment conditions alone are not enough to get Gaussian type behaviour in general in this context, and an entirely different behaviour is possible. Also note that related assumptions, like the $L_4 - L_2$ norm equivalence (see for instance \cite{mendelson:nikita:aos:2020}, \cite{minsker:nips:2017}, \cite{Z}), have already been used in the literature for such a purpose.

\begin{definition}\label{defn:sigma}
For $\J=\{1,\dots,j_2\}$, we define
\begin{align*}
\sigma_{\J}=\sigma_{\J}(\Sigma) = \Big\|\E \Big((\vert R_{\J}\vert ^{1/2}+g_{\J^c}^{-1/2}P_{\J^c}) \overline{X \otimes X} (\vert R_{\J}\vert ^{1/2}+g_{\J^c}^{-1/2}P_{\J^c})\Big)^2\Big\|_{\infty}^{1/2}.
\end{align*}
On the other hand, for ${\J}=\{j_1,\dots,j_2\}$ with $j_1>1$, we define
\begin{align*}
\sigma_{\J}=\sigma_{\J}(\Sigma) = \Big\|\E \Big((\vert R_{\J^c}\vert ^{1/2}+g_{\J}^{-1/2}P_{\J}) \overline{X \otimes X} (\vert R_{\J^c}\vert ^{1/2}+g_{\J}^{-1/2}P_{\J})\Big)^2\Big\|_{\infty}^{1/2}.
\end{align*}
\end{definition}

The quantity $\sigma_{\J}$ plays an important role in the analysis of $\delta_{\J}$ (resp.~$\delta_{\J^c}$) defined in \eqref{defn:delta_J}. The size of $\sigma_{\J}$ can be characterized by the so-called relative ranks defined as follows.
\begin{definition}
For $\J=\{1,\dots,j_2\}$, we define
\begin{align*}
\mathbf{r}_{\J}=\mathbf{r}_{\J}(\Sigma) = \sum_{k \leq j_2} \frac{\lambda_k}{\lambda_k - \lambda_{j_2+1}}  + \frac{1}{\lambda_{j_2}-\lambda_{j_2+1}}\sum_{k >j_2} \lambda_k.
\end{align*}
On the other hand, for ${\J}=\{j_1,\dots,j_2\}$ with $j_1>1$, we define
\begin{align*}
\mathbf{r}_{\J}=\mathbf{r}_{\J}(\Sigma) = \sum_{k < j_1} \frac{\lambda_k}{\lambda_k - \lambda_{j_1}} + \sum_{k > j_2}\frac{\lambda_k}{\lambda_{j_2} - \lambda_k} + \frac{1}{g_{\J}}\sum_{k \in \J} \lambda_k.
\end{align*}
\end{definition}

Note that in the special case of $\J=\{j\}$, our definition coincides with the notion of the relative rank given in \cite{JW18}. In the general case, we use a slightly weaker version of \cite{jirak:wahl:pams:2020}.

The quantities $\rr_\J$ and $\sigma_\J$ are related as follows.

\begin{lemma}\label{lem:explicit:bounds:4:indep}
Suppose that Assumption \ref{ass_moments} holds. Then
\begin{align*}
\sigma_{\J} \lesssim  \mathbf{r}_{\J}.
\end{align*}
If additionally Assumptions \ref{ass:lower:bound} and \ref{ass:indep:4:moments} hold with $m = 4$, then
\begin{align*}
\sigma_{\J}^2 \lesssim  \frac{\lambda_{j_2}}{\lambda_{j_2}-\lambda_{j_2+1}}\mathbf{r}_{\J}=\frac{\lambda_{j_2}}{g_{\J}}\mathbf{r}_{\J}
\end{align*}
for $\J=\{1,\dots,j_2\}$ and
\begin{align*}
\sigma_{\J}^2 {\lesssim} \max\Big(\frac{\lambda_{j_1-1}}{\lambda_{j_1-1}-\lambda_{j_1}},\frac{\lambda_{j_1}}{g_\J}\Big)\mathbf{r}_{\J}
\end{align*}
for $\J=\{j_1,\dots,j_2\}$ with $j_1>1$. Moreover, we have
\begin{align*}
\sigma_{\J}^2 \gtrsim  \frac{\lambda_{j_2}}{g_{\J}}\mathbf{r}_{\J}-\Big(\frac{\lambda_{j_2}}{g_{\J}}\Big)^2
\end{align*}
for $\J=\{1,\dots,j_2\}$ and
\begin{align*}
\sigma_{\J}^2 \gtrsim \max\Big(\frac{\lambda_{j_1-1}}{\lambda_{j_1-1}-\lambda_{j_1}},\frac{\lambda_{j_1}}{g_\J}\Big)\mathbf{r}_{\J}-\max\Big(\frac{\lambda_{j_1-1}}{\lambda_{j_1-1}-\lambda_{j_1}},\frac{\lambda_{j_1}}{g_\J}\Big)^2
\end{align*}
for $\J=\{j_1,\dots,j_2\}$ with $j_1>1$.
\end{lemma}

\begin{proof}[Proof of Lemma \ref{lem:explicit:bounds:4:indep}]
If $\J=\{j_1,\dots,j_2\}$ with $j_1>1$, then
Lemma \ref{lem:explicit:bounds:4:indep} follows from Lemma \ref{lem:moment:computation:KL} applied to the transformed data $X'=(\vert R_{\J^c}\vert ^{1/2}+g_{\J}^{-1/2}P_{\J})X$. The main observation is that $X'$ has the same Karhunen-Loève coefficients as $X$ (in a possibly different order) and thus satisfies Assumption \ref{ass_moments} with the same constant $C_\eta$. Moreover, the eigenvalues of the covariance operator $\Sigma'=\E\, X'\otimes X'$ are transformed in such a way that the eigenvalues are  $\lambda_k/(\lambda_{k} - \lambda_{j_1})$ for $k < j_1$, $\lambda_k/g_\J$ for $k\in \J$, and $\lambda_k/(\lambda_{j_2} - \lambda_k)$ for $k > j_2$. Hence, the first two claims follow from Lemma \ref{lem:moment:computation:KL} (ii) and (iii). The last claim follows similarly, by using \eqref{eq:upper:lower:sigma:2} below and Assumption \ref{ass:lower:bound} instead. If $\J=\{1,\dots,j_1\}$, then the claim follows similarly by changing the role of $\J$ and $\J^c$.
\end{proof}

The relative rank $\mathbf{r}_{\J}$ and $\sigma_{\J}^2$ allow to characterize the behavior of $\delta_\J$, given in \eqref{defn:delta_J}.

\begin{lemma}\label{lem:bound:events:via:fn:minsker}
If Assumption \ref{ass_moments} holds, then
\begin{align*}
\P\Big(\delta_{\J}(E) \sqrt{n} >  C\sqrt{\sigma_{\J}^2 \log n}\Big)\lesssim_p  \qprob^{}_{\J,n,p}
\end{align*}
with
\begin{align}\label{eq:qprob}
    \qprob^{}_{\J,n,p}:=n^{1-p/2}(\log n)^{p/2} \Big(\frac{\mathbf{r}_{\J}}{\sigma_\J}\Big)^{p}.
\end{align}
\end{lemma}

\begin{proof}[Proof of Lemma \ref{lem:bound:events:via:fn:minsker}]
If $\J=\{j_1,\dots,j_2\}$ with $j_1>1$, then
Lemma \ref{lem:bound:events:via:fn:minsker} follows from \eqref{cor:fn} applied to the transformed and scaled data $X'=\sigma_{\J}^{-1/2}(\vert R_{\J^c}\vert ^{1/2}+g_{\J}^{-1/2}P_{\J})X$, using again that $X'$ has the same Karhunen-Loève coefficients as $X$ and thus satisfies Assumption \ref{ass_moments} with the same constant $C_\eta$. Moreover, the eigenvalues of the covariance operator $\Sigma'=\E\, X'\otimes X'$ are transformed in such a way that $\operatorname{tr}(\Sigma')=\mathbf{r}_\J/\sigma_\J$. Hence, an application of \eqref{cor:fn} with $t = C\sqrt{\log n}$ yields
\begin{align*}
&\P\Big(\delta_{\J} \sqrt{n} >  C\sqrt{\sigma_{\J}^2\log n}\Big) \lesssim_p n^{1-p/2} (\log n)^{p/2} \Big(\frac{\mathbf{r}_{\J}}{\sigma_\J}\Big)^{p},
\end{align*}
where we also used the first claim in Lemma \ref{lem:explicit:bounds:4:indep} and $p>2$. If $\J=\{1,\dots,j_1\}$, then the claim follows similarly by changing the role of $\J$ and $\J^c$.
\end{proof}

We continue by recalling the following asymptotic result from multivariate analysis. By  \cite[Proposition 5]{Dauxois_1982}, we have
\begin{equation}\label{EqCLT}
\sqrt{n}(\hat\Sigma-\Sigma)\xrightarrow{d}Z,
\end{equation}
where $Z$ is a Gaussian random variable in the Hilbert space of all Hilbert-Schmidt operators (endowed with trace-inner product) with mean zero and covariance operator $\operatorname{Cov}(Z)=\operatorname{Cov}(X\otimes X)$. More concretely, we have
\begin{align}\label{eq_limit_empirical_covariance_operator}
    Z=\sum_{j\geq 1}\sum_{k\geq  1}\sqrt{\lambda_j\lambda_k}\xi_{jk} (u_j\otimes u_k),
\end{align}
where the upper triangular coefficients $\xi_{jk}$, $k\ge j$ are Gaussian random variables with
\begin{align*}
    \E \xi_{jk}=0, \,\qquad \E \xi_{jk}\xi_{rs}=\E\overline{\eta_j\eta_k}\, \overline{\eta_r\eta_s}=\E(\eta_j\eta_k-\delta_{jk})(\eta_r\eta_s-\delta_{rs}),
\end{align*}
and the lower triangular coefficients $\xi_{jk}$, $k< j$, are determined by $\xi_{jk}=\xi_{kj}$. In the distributional approximation of $\|\hat{P}_{\J} - P_{\J}\|_{2}$, the random variable $L_\J Z$ defined in \eqref{eq_def_linear_term} plays a central role. The latter is a centered Gaussian random variable taking values in the separable Hilbert space of all (self-adjoint) Hilbert-Schmidt operators on $\mathcal{H}$. Let
\begin{align*}
    \Psi_\J=\Psi_\J(\Sigma)=\E (L_\J Z\otimes L_\J Z)
\end{align*}
be the covariance operator of $L_\J Z$. We will make frequent use of the following quantities and relations (cf.~\eqref{Psi:clt:identities}).
\begin{align}\label{eq:ABC}
    A_{\J} &= A_{\J}(\Sigma)=\E\|L_{\J}Z\|_{2}^2=\|\Psi_\J\|_1, \nonumber \\
     B_{\J} &= B_{\J}(\Sigma)  =\E^{1/2}\big(\|L_{\J}Z\|_{2}^2-A_{\J}(\Sigma)\big)^2=\sqrt{2}\|\Psi_\J\|_2, \nonumber \\
     C_{\J} &= C_{\J}(\Sigma)=\E^{1/3}\big(\|L_{\J}Z\|_{2}^2-A_{\J}(\Sigma)\big)^3=2\|\Psi_\J\|_3.
\end{align}
The following lemma provides the connection of the quantities $A_\J$, $B_\J$ and $C_\J$ with the eigenvalues of $\Sigma$.

\begin{lemma}\label{lem:Lambda:Sigma:relation}
Suppose that Assumptions \ref{ass_moments} and \ref{ass:lower:bound} are satisfied. Then
\begin{align*}
 A_{\J}&\asymp \sum_{j\in \J}\sum_{k \notin \J} \frac{\lambda_j \lambda_k}{(\lambda_k - \lambda_j)^2}.
 \end{align*}
If additionally Assumption \ref{ass:indep:4:moments} holds with $m = 4$, then the eigenvalues of $\Psi_{\J}$ are given by
\begin{align*}
    2\alpha_{jk}^2\frac{\lambda_j\lambda_k}{(\lambda_k-\lambda_j)^2},\qquad j\in \J, k\notin \J,
\end{align*}
with $\alpha_{jk}=\E\eta_{j}^2\eta_k^2$. In particular, we have
 \begin{align*}
 B_{\J}^2&\asymp \sum_{j\in \J}\sum_{k \notin \J} \Big(\frac{\lambda_j \lambda_k}{(\lambda_k - \lambda_j)^2}\Big)^2,\quad C_{\J}^3\asymp \sum_{j\in \J}\sum_{k \notin \J} \Big(\frac{\lambda_j \lambda_k}{(\lambda_k - \lambda_j)^2}\Big)^3.
\end{align*}
\end{lemma}

\begin{proof}[Proof of Lemma \ref{lem:Lambda:Sigma:relation}] Using the representation in \eqref{eq_limit_empirical_covariance_operator}, we get
\begin{align*}
    \|L_{\J}Z\|_{2}^2=2\sum_{j\in \J}\sum_{k\notin \J}\frac{\lambda_j\lambda_k}{(\lambda_k-\lambda_j)^2}\xi^2_{jk}
\end{align*}
Hence, the lower bound follows from inserting Assumption \ref{ass:lower:bound}, while the upper bounds follows from Assumption \ref{ass_moments} and H\"older's inequality. Moreover, subject to Assumption \ref{ass:indep:4:moments} with $m = 4$, the random variable $L_\J Z$ has the Karhunen-Loève expansion
\begin{align*}
    L_\J Z=\sum_{j\in \J}\sum_{k\notin \J}\sqrt{2}\alpha_{jk}\frac{\sqrt{\lambda_j\lambda_k}}{\lambda_j-\lambda_k}\frac{\xi_{jk}}{\alpha_{jk}} (u_j\otimes u_k+u_k\otimes u_j)/\sqrt{2},
\end{align*}
and the last two claims follow from the last relations given in \eqref{eq:ABC}.
\end{proof}

\subsection{Main results}

This section is devoted to quantitative limit theorems for $\|\hat{P}_{\J} - P_{\J}\|_{2}$. We state several results subject to different conditions, all having a different range of application, and refer to Section \ref{sec:applications:main} for examples and illustrations. Proofs are given in Section \ref{sec:proof:quantitative:limit:thms}. We assume throughout this section that $n \geq 2$. Our first result concerns estimates based on the $1$-Wasserstein distance.

\begin{thm}\label{thm:clt:I}
Suppose that Assumptions \ref{ass_moments} ($p > 4$) and \ref{ass:lower:bound} hold and that
\begin{align}\label{ass:relativerank}
\sigma_{\J}^2 \frac{\log n}{n}\vert \J\vert \leq c
\end{align}
for some sufficiently small constant $c > 0$. Then, for any $s> 0$,
\begin{align*}
\mathbf{W}\Big(\frac{n}{A_{\J}}\big\|\hat{P}_{\J} - P_{\J}\big\|_{2}^2, \frac{1}{A_{\J}}\big\|L_{\J}Z\big\|_{2}^2  \Big) &\lesssim_{p,s}  n^{-1/12}  + \frac{(\log n)^{3/2}}{n^{1/2}} \frac{\sigma_{\J}^3\vert \J\vert ^{3/2}}{A_{\J}}\\
&+ \frac{n\vert \J\vert }{A_{\J}}\qprob^{}_{\J,n,p} + \Big(\frac{A_{\J}}{n\vert \J\vert }\Big)^s,
\end{align*}
where $\qprob_{\J,n,p}$ is given in \eqref{eq:qprob}.
\end{thm}
We require Condition \eqref{ass:relativerank} for all of our results. It is, however, redundant in most cases (see for instance Section \ref{sec:applications:main}) in the following sense: Upper bounds are nontrivial only if \eqref{ass:relativerank} holds.

The generality of Theorem \ref{thm:clt:I} has a price: Note that $(A_{\J}/n)^2$ may scale as the variance of $\|\hat{P}_{\J} - P_{\J}\|_{2}^2$ or not. In particular, if the expectation $\E \|\hat{P}_{\J} - P_{\J}\|_{2}^2$ significantly dominates the square root of the variance, finer approximation results can be obtained by studying the appropriately centred and scaled version. As a first result in this direction, we present the following.

\begin{thm}\label{thm:clt:II}
Suppose that Assumptions \ref{ass_moments} and \ref{ass:lower:bound} and \eqref{ass:relativerank} hold with $p\geq 3$. Then we have
\begin{align*}
\mathbf{U}\Big(n\big\|\hat{P}_{\J} - P_{\J}\big\|_2^2, \big\|L_{\J}Z\big\|_{2}^2\Big)&\lesssim_p  n^{-1/5}\Big(\frac{A_{\J}}{\sqrt{\lambda_{1,2}(\Psi_\J)}}\Big)^{3/5}   \\&+ \frac{(\log n)^{3/2}}{n^{1/2}}\frac{\sigma_{\J}^3\vert \J\vert ^{3/2}}{\sqrt{\lambda_{1,2}(\Psi_\J)}}+ \qprob_{\J,n,p}
\end{align*}
as well as
\begin{align*}
\mathbf{U}\Big(n\big\|\hat{P}_{\J} - P_{\J}\big\|_2^2, \big\|L_{\J}Z\big\|_{2}^2\Big)&\lesssim_p
\Big(n^{-1/2} \frac{A_{\J}^3}{\lambda_6^{3}(\Psi_\J)} \frac{C_{\J}^3}{B_{\J}^{3}}  \Big)^{1 + 1/10} \\&+  n^{-1/2} \frac{A_\J^6}{\lambda_{1,6}(\Psi_\J)} \frac{C_{\J}^3}{B_{\J}^{3}} \\&+ \frac{(\log n)^{3/2}}{n^{1/2}}\frac{\sigma_{\J}^3\vert \J\vert ^{3/2}}{\sqrt{\lambda_{1,2}(\Psi_\J)}}+
\qprob_{\J,n,p}.
\end{align*}
Recall the definition of $\lambda_{1,j}(\Psi_\J)$, given in \eqref{defn:lambda:product}.
\end{thm}

A simpler (distributional) approximation is given by the following corollary.

\begin{corollary}\label{cor:thm:clt:II}
Suppose that the assumptions of Theorem \ref{thm:clt:II} hold and that $G \sim\mathcal{N}(0,1)$. Then we have
\begin{align*}
\mathbf{U}\Big(\frac{n\big\|\hat{P}_{\J} - P_{\J}\big\|_2^2 - A_{{\J}}}{B_{{\J}}}, G\Big)&\lesssim_p \Big(\frac{C_{\J}}{B_{\J}}\Big)^3  + \text{bound of Theorem \ref{thm:clt:II}}.
\end{align*}
\end{corollary}

However, in this case one may dispose of any conditions on the eigenvalues $\lambda_k(\Psi)$ altogether, as is demonstrated by our next result below. The cost is a slower rate of convergence.

\begin{thm}\label{thm:clt:III}
Suppose that Assumptions \ref{ass_moments} and \ref{ass:lower:bound} and \eqref{ass:relativerank} hold and let $G \sim\mathcal{N}(0,1)$. Then we have
\begin{align*}
\mathbf{U}\Big(\frac{n\big\|\hat{P}_{\J} - P_{\J}\big\|_2^2 - A_{\J}}{B_{\J}}, G\Big)&\lesssim_p n^{1/4-(p\wedge 3)/8}\frac{A^{3/4}_{\J}}{B^{3/4}_{\J}}+   \Big(\frac{C_{\J}}{B_{\J}}\Big)^3 \\
&+ \frac{(\log n)^{3/2}}{n^{1/2}} \frac{\sigma_{\J}^3\vert \J\vert ^{3/2}}{B_{{\J}}}+\qprob_{\J,n,p}.
\end{align*}
\end{thm}

An important feature of both Corollary \ref{cor:thm:clt:II} and Theorem \ref{thm:clt:III} is the error term  $(C_{\J}/B_{\J})^3$. As can be deduced from \cite{hall_lower}, it is, in general, necessary for a central limit theorem to have
\begin{align}\label{clt:lower:condi}
    \Big(\frac{C_{\J}}{B_{\J}}\Big)^3 \to 0 \quad \text{as}\quad n\rightarrow\infty,
\end{align}
otherwise it cannot hold.

So far, we did not require any condition in terms of $n$, $A_{\J}$, $B_{\J}$ and $C_{\J}$. Questions of optimality related to these quantities remain unknown in general, see also the comments after Proposition \ref{prop:clt:III}. In Section \ref{sec:applications:main}, we discuss examples to further illustrate this and also present optimal results up to $\log$-factors in case of polynomially decaying eigenvalues.


\section{Bootstrap approximations}\label{sec:main:boot}

Bootstrap methods are nowadays one of the most popular
ways for measuring the significance of a test or building confidence sets. Among others, their superiority compared to (conventional) limit theorems stems from the fact that they offer finite sample approximations. Some bootstrap methods even outperform typical Berry-Esseen bounds attached to corresponding limit distributions, see for instance \cite{hall:bootstrap:book}. As is apparent from our results in Section \ref{sec:quantitative:limit:thms}, our (unknown) normalising sequences are quite complicated. Hence, as is discussed intensively in \cite{naumov_spokoiny_ptrf2019}, bootstrap methods are a very attractive alternative in the present context.

\subsection{Assumptions and main quantities}

We first require some additional notation. We denote with
$\mathcal{X} = \sigma\big(X_i,\, i \in \N \big)$ the $\sigma$-algebra generated by the whole sample. We further denote with $(X_i')$ an independent copy of $(X_i)$. Next, we introduce the conditional expectations and probabilities
\begin{align*}
\tilde{\E}[\cdot] = \E[ \cdot \vert  \mathcal{X}], \quad \tilde{\P}(\cdot) = \tilde{\E} \mathbf{1}_{\{\cdot\}}.
\end{align*}
The (conditional) measure $\tilde{\P}$ will act as our probability measure in the bootstrap-world (the latter refers to the fact that the sample is fixed and randomness comes from the bootstrap procedure).
Finally, recall the uniform metric $\mathbf{U}(A,B)$ for real-valued random variables $A$, $B$, and denote with
\begin{align}\label{defn:tilde:U}
\tilde{\mathbf{U}}_{}\big(A, B\big) = \sup_{x \in \R}\big\vert\tilde{\P}_{}\big(A \leq x \big) - \tilde{\P}_{}\big(B \leq x \big)  \big\vert
\end{align}
the corresponding conditional version. As before, we assume throughout that $n \geq 2$.

There are many ways to design bootstrap approximations. A popular and powerful method are multiplier methods, which we also employ here. To this end, let $(w_i)$ be an i.i.d. sequence with the following properties.
\begin{align}\label{ass:boot:sequence:w}
\E w_i^2 = 1, \quad \E w_i^{2p}<\infty,
\end{align}
where $p$ corresponds to the same value as in Assumption \ref{ass_moments}. Moreover, denote with $\sigma_w^2 = \E (w_i^2-1)^2$. One may additionally demand $\E w_i = 0$, but this is not necessary. Throughout this section (and the corresponding proofs), we always assume the validity of \eqref{ass:boot:sequence:w} without mentioning it any further.


\begin{algorithmm}[Bootstrap]\label{alg:boot}
Given $(X_i)$ and $(w_i)$, construct the sequence $(\tilde{X}_i) = (w_i X_i)$. Treat $(\tilde{X}_i)$ as new sample, and compute the corresponding empirical covariance matrix $\tilde{\Sigma}$ and projection $\tilde{P}_{\J}$ accordingly.
\end{algorithmm}

Note that our multiplier bootstrap is slightly different from the one of \cite{naumov_spokoiny_ptrf2019}, but a bit more convenient to analyze. On the other hand, by passing to the complex domain $\C$ as underlying field of our Hilbert space, it is not hard to show that Theorem \ref{thm:boot:I}, Theorem \ref{thm:boot:II} and the attached corollaries below are equally valid for the multiplier bootstrap employed in \cite{naumov_spokoiny_ptrf2019}.

\subsection{Main results}

Recall that $(X_i')$ is an independent copy of $(X_i)$, and thus $\hat{P}_{\J}'$, the empirical projection based on $(X_i')$ , is independent of $\mathcal{X}$. Throughout this section, we entirely focus on the uniform metric $\tilde{\mathbf{U}}$.

\begin{thm}\label{thm:boot:I}
Suppose that Assumptions \ref{ass_moments}, \ref{ass:lower:bound} hold and that \eqref{ass:relativerank} is satisfied. Moreover, suppose that
\begin{align}\label{ass:boot:thirdnorm}
\frac{A_{\J}}{C_{\J}} \sqrt{\frac{\log n}{n}} \leq   c
\end{align}
is satisfied for some sufficiently small constant $c > 0$. Let $q \in (2,3]$ and $s \in (0,1)$ and assume that $p>2q$. Then, with probability at least
$1 - C_p\qprob^{1-s}_{\J,n,p}-C_p n^{1-p/(2q)}$, $C_p > 0$, we have
\begin{align*}
\tilde{\mathbf{U}}_{}\Big(\sigma_w^{-2}\big\|\tilde{P}_{\J} - \hat{P}_{\J} \big\|_2^2, \big\|\hat{P}_{\J}' - {P}_{\J} \big\|_2^2 \Big)&\lesssim_p
{\bf A}_{\J,n,p,s},
\end{align*}
where
\begin{align}\label{eq:bfA} \nonumber
{\bf A}_{\J,n,p,s} &= n^{1/4-q/8}\Big(\frac{A_{{\J}}}{B_{{\J}}}\Big)^{3/4}+  \Big(\frac{C_{{\J}}}{B_{{\J}}}\Big)^3 + \frac{\sqrt{\log n}}{\sqrt{n}} \frac{A_{\J}}{B_{\J}}\\
&+ n^{-1/2}\log^{3/2} n \frac{\sigma_{\J}^3\vert \J\vert ^{3/2}}{B_{{\J}}}+\qprob_{\J,n,p}^s.
\end{align}
\end{thm}

The bound ${\bf A}_{\J,n,p,s}$ above is also based on Theorem \ref{thm:clt:III}, which explains the origin of $(C_{{\J}}/B_{{\J}})^3$. Loosely speaking, this means that we approximate with a standard Gaussian distribution to show closeness in Theorem \ref{thm:boot:I}.

The uniform metric is strong enough to deduce quantitative statements for empirical quantiles, which is useful in statistical applications. To exemplify this further, we state the following result.

\begin{corollary}\label{cor:boot:I:quant}
Given the assumptions and conditions of Theorem \ref{thm:boot:I} and
\begin{align}\label{eq:quantile}
\hat{q}_{\alpha} = \inf\Big\{x:\, \tilde{\P}_{}\Big(\frac{n}{\sigma_w^2} \|\tilde{P}_{\J} - \hat{P}_{\J}\|_2^2  \leq x\Big) \geq 1 - \alpha \Big\},
\end{align}
we have
\begin{align*}
\Big\vert\alpha - \P\Big(n \big\|\hat{P}_{\J} - P_{\J} \big\|_2^2 > \hat{q}_{\alpha} \Big) \Big\vert \lesssim_p \qprob^{1-s}_{\J,n,p} + n^{1-p/(2q)}  + \mathbf{A}_{\J,n,p,s},
\end{align*}
where $\mathbf{A}_{\J,n,p,s}$ is defined in \eqref{eq:bfA}.
\end{corollary}

Next, we present our second bound main result in this section.

\begin{thm}\label{thm:boot:II}
Suppose that Assumptions \ref{ass_moments} and \ref{ass:lower:bound} hold with $p > 6$, Assumption \ref{ass:indep:4:moments} holds with $m\in\{4,8\}$, and that \eqref{ass:relativerank} and \eqref{ass:boot:thirdnorm} are satisfied. Finally, suppose that
\begin{align}\label{ass:boot:truncationset}
\sqrt{\frac{\log n}{n}}\frac{A_{\J}}{\lambda_2(\Psi_\J)}\lesssim 1.
\end{align}
Let $s \in (0,1)$. Then, with probability at least
$1 - C_p\qprob^{1-s}_{\J,n,p} - C_pn^{1-p/6}$, $C_p > 0$, we have
\begin{align*}
\tilde{\mathbf{U}}_{}\Big(\sigma_w^{-2}\big\|\tilde{P}_{\J} - \hat{P}_{\J} \big\|_2^2, \big\|\hat{P}_{\J}' - {P}_{\J} \big\|_2^2 \Big) &\lesssim_p \mathbf{B}_{\J,n,p,s},
\end{align*}
where
\begin{align}\label{eq:bfB} \nonumber
\mathbf{B}_{\J,n,p,s} &= n^{-1/5}\Big(\frac{A_{\J}}{\sqrt{\lambda_{1,2}(\Psi_\J)}}\Big)^{3/5}   +\frac{\log^{3/2} n}{n^{1/2}}\frac{\sigma_{\J}^3\vert \J\vert ^{3/2}}{\sqrt{\lambda_{1,2}(\Psi_\J)}}\\&+\frac{\log^{1/2} n}{n^{1/2}}\frac{\sqrt{A_{\J}} \operatorname{tr}(\Psi_{\J,\I}^{1/2})}{\sqrt{\lambda_{1,2}(\Psi_\J)}} +\frac{A_{\J,\I^c}\log n }{\sqrt{{\lambda}_{1,2}(\Psi_\J)}}+ \big(\qprob^{}_{\J,n,p} + n^{1- p/6}\big)^{s},
\end{align}
and $\qprob_{\J,n,p}$ is given in \eqref{eq:qprob}. In all the results above, we require $\J=\{j_1,\dots,j_2\}$ and $\I=\{1,\dots,i_2\}$ with $i_2>j_2+2$.
\end{thm}

As in Theorem \ref{thm:clt:II}, it is possible to establish a $\sqrt{n}$ rate at the cost of additional factors involving various eigenvalues. Since the proof is very similar, we omit any further details here.

The quantities $A_{\J,\I^c}$ and $\Psi_{\J,\I}$ that appear in the above bound are not defined yet. In essence, the set $\mathcal{I}$ is used to truncate some lower degree indices. The exact definition requires some preparation, and is given in Section \ref{sec:proof:bootstrap:II:notation}. In contrast to Theorem \ref{thm:boot:I}, Theorem \ref{thm:boot:II} above is built around Theorem \ref{thm:clt:II}, and thus avoids the error term $(C_{{\J}}/B_{{\J}})^3$. However, apart from the eigenvalues, this comes at (other) additional costs, in particular the expressions $\operatorname{tr}(\Psi_{\J,\I}^{1/2})$ and $A_{\J,\I^c}$ require a careful selection of the set $\I$. As before, we have the following corollary.

\begin{corollary}\label{cor:boot:II:quant}
Given the assumptions and conditions of Theorem \ref{thm:boot:II} and $\hat{q}_{\alpha}$ defined as in \eqref{eq:quantile}, we have for $s \in (0,1)$
\begin{align*}
\Big\vert\alpha - \P\Big(n \big\|\hat{P}_{\J} - P_{\J} \big\|_2^2 > \hat{q}_{\alpha} \Big) \Big\vert \lesssim_p \qprob^{1-s}_{\J,n,p}   + \mathbf{B}_{\J,n,p,s},
\end{align*}
where $\mathbf{B}_{\J,n,p,s}$ is given in \eqref{eq:bfB}.
\end{corollary}
%


\section{Applications: Specific models and computations}\label{sec:applications:main}

In this Section, we discuss specific models to illustrate our results with explicit bounds and compare them to previous results.

To this end, we discuss two basic, fundamental models omnipresent in the literature. In our first model, we consider the case where the eigenvalues $\lambda_j$ decay at a certain rate (polynomial or exponential). This behaviour is typically encountered in functional data analysis or in a machine learning context involving kernels, see for instance \cite{hall2007}, \cite{bartlett:benign:2020}, \cite{steinwart:jml:2020}. Our second model is the classical spiked covariance model (factor model), which is more popular in high dimensional statistics, econometrics and probability theory, see \cite{johnstone:spiked:2001}, \cite{fan:survey:2021}, \cite{bai:2003:econometrica}.

\subsection{Model I}\label{sec:app:model:I}

Throughout this section, we assume that Assumptions \ref{ass_moments}, \ref{ass:lower:bound} and \ref{ass:indep:4:moments} ($m = 4$) hold with $p>3$.

For $J\geq 1$, we consider the set $\J = \{1,\ldots,J\}$ and the singleton $\J'=\{J\}$. We assume that $n\geq 2$, so $\log n$ is always positive. Throughout this Section, all constants depend on the parameter $\mathfrak{a}$. To simplify notation, we will not indicate this explicitly.

\subsubsection{Exponential decay}\label{sec:app:model:I:exp}
We first consider the case of exponential decay, that is, we
suppose that for some $\mathfrak{a} > 0$, we have $\lambda_j= e^{-\mathfrak{a} j}$ for all $j\geq 1$.


In this setup, the relative rank $\mathbf{r}_{\J}$,  $A_{\J}$ and related quantities have already been computed in the literature, see for instance \cite{jirak_eigen_expansions_2016}, \cite{JW18}, \cite{reiss:wahl:aos:2020}. By similar (straightforward) computations together with Lemmas \ref{lem:explicit:bounds:4:indep} and \ref{lem:Lambda:Sigma:relation}, we have

\begin{align}\label{eq:model:1:exp:bound:ABC} \nonumber
&\mathbf{r}_{\J} \asymp \sigma_{\J}^2 \asymp J, \quad \qprob^{}_{\J,n,p} \asymp n^{1-p/2}J^{p/2}\log^{p/2} n,\\
&A_{\J} \asymp B_{\J} \asymp C_{\J} \asymp 1,
\end{align}
as well as
\begin{align*} \nonumber
&\mathbf{r}_{\J'} \asymp \sigma_{\J'}^2 \asymp J, \quad \qprob^{}_{\J',n,p} \asymp n^{1-p/2}J^{p/2}\log^{p/2} n,\\
&A_{\J'} \asymp B_{\J'} \asymp C_{\J'} \asymp 1.
\end{align*}
Moreover, using again Lemma \ref{lem:Lambda:Sigma:relation}, we have
\begin{align}\nonumber
&\lambda_{j}(\Psi_\J) \asymp \lambda_j(\Psi_{\J'})\asymp 1, \quad j=1,\dots,6.
\end{align}
It is now easy to apply the results. For example, Theorem \ref{thm:clt:II} yields
\begin{align}\label{eq:appl:model1:exp:1}
\mathbf{U}\Big(n\big\|\hat{P}_{\J} - P_{\J}\big\|_2^2, &\big\|L_{\J}Z\big\|_{2}^2\Big)\lesssim_p\Big(\frac{J^6 \log^3 n}{n}\Big)^{1/2}
\end{align}
for cumulated projectors, and
\begin{align}\label{eq:appl:model1:exp:2.2}
\mathbf{U}\Big(n\big\|\hat{P}_{\J'} - P_{\J'}\big\|_2^2, &\big\|L_{\J'}(Z)\big\|_{2}^2\Big)\lesssim_p \Big(\frac{J^3 \log^3 n}{n}\Big)^{1/2}
\end{align}
for the special case of single projectors. Here, we used that $\qprob^{}_{\J,n,p}$ is negligible due to the fact that $p> 3$, and that condition \eqref{ass:relativerank} can be dropped due to the fact that the uniform distance is bounded by $1$, while the right-hand side of these bounds exceeds $1$ whenever \eqref{ass:relativerank} does not hold, that is, when $n^{-1/2}J (\log n)^{1/2}\gtrsim 1$.

\textit{Comparison and discussion:}
The results of \cite{koltchinskii2017} are not applicable in this setup and only provide the trivial bound $\leq 1$. More precisely, Theorem 6 in \cite{koltchinskii2017} only yields a non trivial result if
\begin{align*}
n^2 \mathbf{Var} \|\hat{P}_{\J'} - P_{\J'}\|_2^2 \to \infty.
\end{align*}
However, this is not the case here. In fact, a normal approximation by a standard Gaussian random variable $G$ is not possible at all in this setup. This follows from \eqref{clt:lower:condi} and the fact that $B_{\J} \asymp C_{\J}$.

In stark contrast, our bounds above in \eqref{eq:appl:model1:exp:1}
and \eqref{eq:appl:model1:exp:2.2} provide non-trivial results even for moderately large $J$. This is a consequence of the relative approach that we employ here. In addition, our probabilistic assumptions are much weaker compared to \cite{koltchinskii2017}. Hence, even for the Gaussian case, our results are new.

Next, regarding bootstrap approximations, if $p \geq 9$, then an application of Corollary \ref{cor:boot:II:quant} (with $s = 1/2$, $\I = \N$) yields
\begin{align}\label{eq:appl:model1:exp:2:boot}
    \Big\vert\alpha - \P\Big(n \big\|\hat{P}_{\J} - P_{\J} \big\|_2^2 > \hat{q}_{\alpha} \Big) \Big\vert &\lesssim_p  \frac{1}{n^{1/5}} + \Big(\frac{J^6 \log^3 n}{n}\Big)^{1/2}.
\end{align}


\textit{Comparison and discussion:} In contrast to \cite{koltchinskii2017}, the bootstrap approximation of \cite{naumov_spokoiny_ptrf2019}, Theorem 2.1, is applicable if we additionally assume Gaussianity. For finite $J$, it appears that the bound provided by Theorem 2.1 is better than ours for single projectors. However, this drastically changes if $J$ increases: Simple computations show that their rate is lower bounded by
\begin{align}
\Big(\frac{e^{6\mathfrak{a}J} \log^3 n}{n}\Big)^{1/2},
\end{align}
which is useless already for $J \geq (\log n)/6\mathfrak{a}$. On the other hand, our bound \eqref{eq:appl:model1:exp:2:boot} yields useful results even for $J \leq n^{1/6 - \delta}$, $\delta > 0$. Thus, despite having much weaker assumptions, our results extend and improve upon those of \cite{naumov_spokoiny_ptrf2019} even in the Gaussian case.

\subsubsection{Polynomial decay}\label{sec:app:model:I:poly}

We next consider the case of polynomial decay, that is, we suppose that there is a constant $\mathfrak{a} > 1$ such that $\lambda_j= j^{-\mathfrak{a}}$ for all $j\geq 1$. For simplicity, we assume that $J\geq 2$.

As in the previous Section \ref{sec:app:model:I:exp}, straightforward computations, together with Lemmas \ref{lem:explicit:bounds:4:indep} and \ref{lem:Lambda:Sigma:relation}, yield

\begin{align}\label{eq:model:1:bound:ABC} \nonumber
&\mathbf{r}_{\J} \asymp  J \log J, \quad \qprob^{}_{\J,n,p} \asymp n^{1-p/2} \big(\log J \log n\big)^{p/2},\\
&\sigma_{\J}^2 \asymp A_{\J} \asymp J^2 \log J, \quad B_{\J} \asymp C_{\J} \asymp J^{2},
\end{align}
as well as
\begin{align}\nonumber
&\mathbf{r}_{\J'} \asymp  J \log J, \quad \qprob^{}_{\J',n,p} \asymp n^{1-p/2} \big(\log J \log n\big)^{p/2},\\
&\sigma_{\J'}^2\asymp J^2 \log J,\quad A_{\J'} \asymp B_{\J'} \asymp C_{\J'} \asymp J^{2}.
\end{align}
Moreover, using again Lemma \ref{lem:Lambda:Sigma:relation}, it follows that
\begin{align}
&\lambda_{j}(\Psi_\J) \asymp \lambda_j(\Psi_{\J'})\asymp J^2, \quad j=1,\dots,6.
\end{align}
It is now again easy to apply the results. For example, Theorem \ref{thm:clt:II} yields
\begin{align*}
\mathbf{U}\Big(n\big\|\hat{P}_{\J} - P_{\J}\big\|_2^2, &\big\|L_{\J}Z\big\|_{2}^2\Big)\lesssim_p n^{-1/2} (\log J)^6 + n^{-1/2}J^{5/2}(\log J \log n)^{3/2}
\end{align*}
for general projectors $P_{\J}$, and
\begin{align*}
\mathbf{U}\Big(n\big\|\hat{P}_{\J'} - P_{\J'}\big\|_2^2, &\big\|L_{\J'}(Z)\big\|_{2}^2\Big)\lesssim_p n^{-1/2}J(\log J \log n)^{3/2}
\end{align*}
for the special case of single projectors. Here, we again exploited that condition \eqref{ass:relativerank} can be dropped due to the fact that the uniform distance is bounded by $1$, while the right-hand side of these bounds exceeds $1$ whenever \eqref{ass:relativerank} does not hold, that is, when $n^{-1/2}J^{3/2}(\log J\log n)^{1/2}\gtrsim 1$ and $n^{-1/2}J(\log J\log n)^{1/2}\gtrsim 1$, respectively. Observe that in case of single projectors, the dependence on $J$ is optimal up to $\log$-factors, we refer to \cite{JW18}, Example 2 for a more detailed discussion.

\textit{Comparison and discussion:}
As in the exponential case, the results of \cite{koltchinskii2017} are not applicable here. On the other hand, again due to the relative nature of our approach, our bounds in \eqref{eq:appl:model1:exp:1} are quite general, simple, and new even in for the Gaussian case. We emphasize in particular that our bounds are independent of $\mathfrak{a}$ up to constants.

Next, we want to apply Corollary \ref{cor:boot:II:quant}. To this end, we need to strengthen Assumption \ref{ass:indep:4:moments} to $m = 8$. First, if $\mathfrak{a}>2$, we choose $\I=\N$, in which case we have $\Psi_{\J,\I} = \Psi_{\J}$ (using Lemma \ref{lem:Lambda:Sigma:relation}) and thus
\begin{align*}
\operatorname{tr}(\sqrt{\Psi_{\J,\I}}) = \operatorname{tr}(\sqrt{\Psi_{ \J}}) &\asymp  J^2.
\end{align*}
Hence, if $p \geq 9$, an application of Corollary \ref{cor:boot:II:quant} (with $s=1/2$) yields
\begin{align*}
    \Big\vert\alpha - \P\Big(n \big\|\hat{P}_{\J} - P_{\J} \big\|_2^2 > \hat{q}_{\alpha} \Big) \Big\vert &\lesssim_p  \frac{1}{n^{1/5}}
    (\log J)^{3/5} +\frac{J^{5/2}}{n^{1/2}} (\log n \log J)^{3/2}.
\end{align*}

Second, if $1<\mathfrak{a}\leq 2$, let $\I = \{I,I+1,\ldots\}$ with $I\geq 2J$, in which case
\begin{align*}
\operatorname{tr}(\sqrt{\Psi_{\J,\I}}) &\asymp J^{1+\mathfrak{a}/2}I^{1-\mathfrak{a}/2},\\
A_{\J,\I^c} &\asymp J^{1+\mathfrak{a}}I^{1-\mathfrak{a}}.
\end{align*}
Thus, if $p \geq 9$, an application of Corollary \ref{cor:boot:II:quant} (with $s=1/2$) yields
\begin{align*}
    \Big\vert\alpha - \P\Big(n \big\|\hat{P}_{\J} - P_{\J} \big\|_2^2 > \hat{q}_{\alpha} \Big) \Big\vert &\lesssim_p \frac{1}{n^{1/5}}(\log J)^{3/5} +\frac{J^{5/2}}{n^{1/2}} (\log n \log J)^{3/2}\\
    &+\frac{J^{\mathfrak{a}/2}I^{1-\mathfrak{a}/2}}{n^{1/2}}(\log n\log J)^{1/2}+J^{\mathfrak{a}-1}I^{1-\mathfrak{a}}\log n.
\end{align*}
Balancing with respect to $I$ leads to
\begin{align}\label{eq:model1:boot}\nonumber
    \Big\vert\alpha - \P\Big(n \big\|\hat{P}_{\J} - P_{\J} \big\|_2^2 > \hat{q}_{\alpha} \Big) \Big\vert &\lesssim_p \frac{1}{n^{1/5}}(\log J)^{3/5} +\frac{J^{5/2}}{n^{1/2}} (\log n \log J)^{3/2}\\& +
    \Big(\frac{J^2 \log J}{n}\Big)^{1-1/\mathfrak{a}}
    (\log n)^{1/\mathfrak{a}}.
\end{align}

\textit{Comparison and discussion:}
The situation is similarly as before in the exponential case: For finite $J$, Theorem 2.1 in \cite{naumov_spokoiny_ptrf2019} gives superior results for single projectors, if the underlying distribution is Gaussian. However, their rate is lower bounded by
\begin{align}
\Big(\frac{J^{6 \mathfrak{a} + 2} \log^3 n}{n}\Big)^{1/2},
\end{align}
and thus leads to a much smaller range for $J$, particularly for larger $\mathfrak{a}$. In contrast, the range for $J$  for our results in \eqref{eq:model1:boot} is invariant in $\mathfrak{a}$, and the overall bound even slightly improves as $\mathfrak{a}$ gets larger.

\subsection{Model II}\label{sec:app:model:II}

We only consider the case $\J = \{1,\ldots, J\}$ with $J > 6$ in this section. We note, however, that single projectors can readily be handled in a similar manner.
Throughout this section, we assume the validity of Assumptions \ref{ass_moments}, \ref{ass:lower:bound} and \ref{ass:indep:4:moments} with $p \geq 3$ and $m = 4$.

\subsubsection{Factor models - pervasive case}\label{sec:app:model:II:pervasive}

The literature on (approximate, pervasive) factor models typically assumes that the first $J$ eigenvalues diverge at rate $\asymp d$ (with $d = \operatorname{dim}\mathcal{H}$), whereas all the remaining eigenvalues are bounded and do not, in total, have significantly more mass than any of the first $J$ eigenvalues. This assumption can be expressed in terms of \textit{pervasive} factors, see for instance \cite{bai:2003:econometrica}, \cite{stock:watson:2002}, which are particularly relevant in econometrics. In the language of statistics, this means the remaining $d-J$ components do not explain significantly more (variance) than any of the first $J$.

Here, we generalize the above conditions as follows. We assume that there exist constants $0 < c \leq C < \infty$, such that
\begin{align}\label{eq:app:model:II:pervasive:def}
    \lambda_1\leq C\lambda_J,\qquad \lambda_J-\lambda_{J+1}\geq c\lambda_J, \qquad \frac{\operatorname{tr}_{\J^c}(\Sigma)}{\lambda_1} \leq C.
\end{align}

Observe that this implies
\begin{align*}
\frac{\operatorname{tr}(\Sigma)}{\lambda_1} \asymp J,
\end{align*}
which is the desired feature of pervasive factor models. It is convenient to introduce a notion of a subset trace by
\begin{align}\label{defn:subset:trace}
\operatorname{tr}_{\I}(\Sigma) = \sum_{i \in \I} \lambda_i, \quad \I \subseteq \N.
\end{align}

We then have the following relations.
\begin{align}\label{eq:model:2:bound:ABC} \nonumber
&\rr_{\J} \asymp \sigma_{\J}^2 \asymp  J+\frac{\operatorname{tr}_{\J^c}(\Sigma)}{\lambda_J}, \quad \qprob^{}_{\J,n,p} \asymp n \Big( \frac{\log n \rr_{\J} }{n}\Big)^{p/2},\\
&A_{\J} \asymp J\frac{\operatorname{tr}_{\J^c}(\Sigma)}{\lambda_J}, \quad B_{\J}^2 \asymp J\frac{\operatorname{tr}_{\J^c}(\Sigma^2)}{\lambda_J^2}, \quad C_{\J}^3 \asymp J\frac{\operatorname{tr}_{\J^c}(\Sigma^3)}{\lambda_J^3}.
\end{align}

It turns out that pervasive factor models lead to particularly simple results. Indeed, an application of
Theorem \ref{thm:clt:II} yields

\begin{align}\label{eq:app:model:II:pervasive:thm2}
\mathbf{U}\Big(n\big\|\hat{P}_{\J} - P_{\J}\big\|_2^2, \big\|L_{\J}Z\big\|_{2}^2\Big)&\lesssim_p \Big(\frac{J^{6}}{n}\Big)^{1/2} \Big((\log n)^{3/2} +  J^{5/2}\Big),
\end{align}
where we used that $(C_{\J}/B_{\J})^3 \lesssim 1/\sqrt{J}$.

We note that there is no restriction on the dimension $d$ of the underlying Hilbert space in this case. 

\textit{Comparison and discussion:} Our bound in \eqref{eq:app:model:II:pervasive:thm2} is easy to use and fairly general in terms of underlying assumptions. The literature does not appear to have a comparable result, even in the Gaussian case. The single projector results of Theorem 6 in \cite{koltchinskii2017} only yield the trivial bound $\leq 1$ in this case. We mention however that condition \eqref{clt:lower:condi} amounts to $(C_{\J}/B_{\J})^3 \lesssim 1/\sqrt{J}$. Hence, if we consider the general projectors $P_{\J}$, $\J = \{1,\ldots,J\}$ with $J \to \infty$, one may also use a standard Gaussian approximation as in Theorem \ref{thm:clt:III}. 

Next, we turn to the bootstrap. Using Corollary \ref{cor:boot:II:quant} ($p > 6$, $s = 1/2$, $\I = \{1,\ldots,d\}$, $m = 8$ in Assumption \ref{ass:indep:4:moments}), we arrive at
\begin{align}\label{eq:app:model:II:pervasive:cor5} \nonumber
\Big\vert\alpha - \P\Big(n \big\|\hat{P}_{\J} - P_{\J} \big\|_2^2 > \hat{q}_{\alpha} \Big) \Big\vert &\lesssim_p  \Big(\frac{J^{3}}{n}\Big)^{1/5} + \Big(\frac{ J^6 \log^{3} n}{n}\Big)^{1/2} \\&+\Big(\frac{J d  \log n}{n}\Big)^{1/2} + n^{(6- p)/12},
\end{align}
where we also used $\operatorname{tr}(\sqrt{\Psi_{\J}}) \lesssim \sum_{1 \leq k \leq d} \sqrt{\lambda_k/\lambda_1} \lesssim \sqrt{d}$.

\textit{Comparison and discussion:} The situation is related to Model I. If the dimension $d$ is small and the setup is purely Gaussian, the results of \cite{naumov_spokoiny_ptrf2019} are superior. However, for larger $d$, this changes, as can be seen as follows. In the present context, their rate is lower bounded by
\begin{align}
\frac{d \sqrt{\log n}}{\sqrt{n}}.
\end{align}
This implies in particular $d = o\big( \sqrt{n/\log n}\big)$ for a non trivial result. In contrast, our bound \eqref{eq:app:model:II:pervasive:cor5} is valid also for $d$ as large as $d = o\big(n/\log n\big)$. For a sake of better comparison, we assumed here that $J$ is fixed, since the results of \cite{naumov_spokoiny_ptrf2019} only apply to single projectors. Among others, a key reason for our improvement compared to \cite{naumov_spokoiny_ptrf2019} is the usage of our concentration inequality in Lemma \ref{lem:fuknagaev:HS:nuclear:operator}(ii).


\subsubsection{Spiked covariance}
We next consider a simple spiked covariance model of the form
 \begin{align*}
  \lambda_{J+1}=\dots=\lambda_{d}=\sigma^2  \quad\text{and}\quad  \sigma^2+g_J=\lambda_J\leq \dots\leq \lambda_1\leq \sigma^2+Cg_J,
\end{align*}
where  $\sigma^2>0$ is the level of noise, $g_J=\lambda_J-\lambda_{J+1}>0$ is the relevant spectral gap, and $C>0$ is a constant. For simplicity, we assume that $\sigma^2=1$ such that $J$ and $g_J$ are the only remaining parameters. Moreover, we assume that $d\geq 6$, $J\leq d-J$ and $g_J\in (0,1]$.  In particular, all eigenvalues have the same magnitude up to multiplicative constants and we have the following relations
\begin{align}\label{eq:model:2:bound:ABC:2} \nonumber
&\rr_{\J} \asymp \frac{d}{g_J},\quad \sigma_{\J}^2 \asymp  \frac{d}{g_J^2}, \quad \qprob^{}_{\J,n,p} \asymp n \Big( \frac{d\log n}{n}\Big)^{p/2},\\
&A_{\J} \asymp \frac{dJ}{g_J^2} \quad B_{\J}^2 \asymp \frac{dJ}{g_J^4}, \quad C_{\J}^3 \asymp \frac{dJ}{g_J^6}.
\end{align}
Moreover, using again Lemma \ref{lem:Lambda:Sigma:relation} and the fact that $J>6$, it follows that $\lambda_{j}(\Psi_\J) \asymp 1/g_J^2, \quad j=1,\dots,6$.
It is now again easy to apply the results. For example, Theorem~\ref{thm:clt:I} yields
\begin{align}\label{sec:app:spiked:wasser} \nonumber
\mathbf{W}\Big(\frac{n}{A_{\J}}\big\|\hat{P}_{\J} - P_{\J}\big\|_{2}^2, \frac{1}{A_{\J}}\big\|L_{\J}Z\big\|_{2}^2  \Big) &\lesssim_p  n^{-1/12}  + (\log n)^{3/2} \Big(\frac{J d}{ng_J^2}\Big)^{1/2} \\&+ \frac{n^2}{d} \Big(\frac{d \log n}{n}\Big)^{p/2},
\end{align}
provided that $\log(n)dJ/(ng_J^2)\lesssim 1$, which amounts to condition \eqref{ass:relativerank}. Note that if we demand that the bound in \eqref{sec:app:spiked:wasser} is of magnitude $o(1)$, \eqref{ass:relativerank} is trivially true.
Moreover, an application of Theorem \ref{thm:clt:III} yields
\begin{align}\label{sec:app:spiked:unif} \nonumber
\mathbf{U}\Big(\frac{n\big\|\hat{P}_{\J} - P_{\J}\big\|_2^2 - A_{\J}}{B_{\J}}, G\Big)&\lesssim_p \Big(\frac{d J}{n^{1/3}}\Big)^{3/8} + \frac{1}{\sqrt{dJ}} \\&+ \frac{ dJ\log^{3/2} n}{\sqrt{n}g_J} + n \Big( \frac{ d \log n}{n}\Big)^{p/2},
\end{align}
where we again exploited that condition \eqref{ass:relativerank} can be dropped due to the fact that the uniform distance is bounded by $1$, while the right-hand side of these bounds exceeds $1$ whenever \eqref{ass:relativerank} does not hold, that is, when $dJ\log(n)/(ng_J^2)\gtrsim 1$. Note that the bound in \eqref{sec:app:spiked:unif} is only nontrivial if $d \to \infty$ due to the error term $(d J)^{-1/2}$. However, if $d$ is bounded, one may (again) apply Theorem \ref{thm:clt:II}, we omit the details.

\textit{Comparison and discussion:} We first observe that for $(J d)/(ng_J^2) \leq n^{-\delta}$, $\delta > 0$ and $p$ large enough, \eqref{sec:app:spiked:wasser} is bounded by $n^{-\alpha}$, where $\alpha = (1/12) \wedge (\delta \rho/2)$ for $\rho < 1$. It follows, in particular, that
\begin{align}
\E n\big\|\hat{P}_{\J} - P_{\J}\big\|_2^2 = \E \big\|L_{\J}Z\big\|_{2}^2 \big(1 + O(n^{-\alpha})\big).
\end{align}
This should be contrasted with previous results in the purely Gaussian setting, see for instance \cite{cai_aos_2013}, and \cite{birnbaum_aos_2013} for a closely related result. Regarding \eqref{sec:app:spiked:unif}, the situation is a bit more complex, as, even for large enough $p$, three terms may be dominant. Following the discussion in \cite{koltchinskii2017} below Theorem 6, let us fix $J$ (and hence $g_J$), and assume $d = d_n \to \infty$ as $n$ increases, and that $p$ is large enough. Then \eqref{sec:app:spiked:unif} and Theorem 6 in \cite{koltchinskii2017} amount to
\begin{align*}
\Big(\frac{d }{n^{1/3}}\Big)^{3/8} + \frac{1}{\sqrt{d}} \qquad \text{and}\qquad \sqrt{\frac{d}{n} \log\Big(\frac{n}{d} \vee 2\Big)} + \frac{1}{\sqrt{d}},
\end{align*}
respectively. So in this case, our bound  is inferior compared to the Gaussian case treated in \cite{koltchinskii2017}, only allowing a range for $d = o(n^{1/3})$ compared to (almost) $d = o(n)$. Such a behaviour will always be present if the error bound from the pertubation approximation dominates, and is weaker than bounds implied by the isoperimetric inequality. On the other hand, \cite{koltchinskii2017} only treats $\J = \{J\}$, and our distributional assumptions are much weaker. In particular, no sub-Gaussianity or even Gaussianity is necessary for \eqref{sec:app:spiked:unif} to hold.

Next, we consider bootstrap approximations. Assuming $p > 6$, Corollary \ref{cor:boot:I:quant} ($q = 3$, $s = 1/2$), yields

\begin{align}\label{sec:app:spiked:boot} \nonumber
\Big\vert\alpha - \P\Big(n \big\|\hat{P}_{\J} - P_{\J} \big\|_2^2 > \hat{q}_{\alpha} \Big) \Big\vert &\lesssim_p \Big(\frac{d J}{n^{1/3}}\Big)^{3/8} + \frac{1}{\sqrt{d J}} \\&+ \frac{ d J\log^{3/2} n}{\sqrt{n}g_J}+ n^{1/2} \Big( \frac{ d \log n}{n}\Big)^{p/4}.
\end{align}

For a comparison with Theorem 2.1 in \cite{naumov_spokoiny_ptrf2019}, we follow the discussion above and reconsider the case where $J$, $g_J$ are finite, $d = d_n \to \infty$ and $p$ is large enough. Our bound \eqref{sec:app:spiked:boot}, and the one of Theorem 2.1 in \cite{naumov_spokoiny_ptrf2019} (purely Gaussian setup), are of magnitude
\begin{align*}
\Big(\frac{d}{n^{1/3}}\Big)^{3/8} + \frac{1}{\sqrt{d}} \qquad \text{and} \qquad \frac{(d \sqrt{\log n})^3}{\sqrt{n}},
\end{align*}
respectively. We see that for smaller $d$, the bound of  \cite{naumov_spokoiny_ptrf2019} is superior compared to ours. However, for larger $d$, our bound prevails. In particular, our range of applicability $d = o(n^{1/3})$ is larger compared to $d = o(n^{1/6}/(\log n)^{3/3})$, despite not demanding any sub-Gaussianity or even Gaussianity as in \cite{naumov_spokoiny_ptrf2019}.


\section{Proofs}\label{sec:proofs}

\subsection{Proofs for Section \ref{sec_preliminaries_perturbation}}\label{sec:proof:preliminaries:perturbation}
Conceptually, we borrow some ideas from \cite{JW18} and \cite{Wahl}, coupled with some substantial innovations to deal with sums of eigenspaces.

The proofs are based on a series of lemmas. We start with an eigenvalue separation property.

\begin{lemma}\label{LemEVSep}
If $\delta_{\J}\leq 1$, then we have
\begin{align*}
    \hat \lambda_{j_1}-\lambda_{j_1}\leq \delta_{\J}(\lambda_{j_1-1}-\lambda_{j_1}),\qquad\hat \lambda_{j_2}-\lambda_{j_2}\geq -\delta_{\J}(\lambda_{j_2}-\lambda_{j_2+1}).
\end{align*}
\end{lemma}

\begin{proof}
The claim follows from \cite[Proposition 1]{JW18}, using the same line of arguments as in the proof of Lemma 2 in \cite{Wahl}. For completeness, we repeat the proof. Set
\begin{align*}
T_{\geq j_1}&=\sum_{k\geq j_1}\frac{1}{\sqrt{\lambda_{j_1}+\delta_{\J}(\lambda_{j_{1}-1}-\lambda_{j_1})-\lambda_k}}P_k,\\
T_{\leq j_2}&=\sum_{k\leq j_2}\frac{1}{\sqrt{\lambda_k+\delta_{\J}(\lambda_{j_2}-\lambda_{j_2+1})-\lambda_{j_2}}}P_k.
\end{align*}
Then \cite[Proposition 1]{JW18} states that $\hat\lambda_{j_1}-\lambda_{j_1}\leq \delta_{\J}(\lambda_{j_{1}-1}-\lambda_{j_1})$ (resp. $\hat\lambda_{j_2}-\lambda_{j_2}\geq -\delta_{\J}(\lambda_{j_2}-\lambda_{j_2+1})$), provided that $\|T_{\geq j_1}ET_{\geq j_1}\|_\infty\leq 1$ (resp. $\|T_{\leq j_2}ET_{\leq j_2}\|_\infty\leq 1$). Now, by simple properties of the operator norm, using that
\[
\lambda_{j_1}+\delta_{\J}(\lambda_{j_{1}-1}-\lambda_{j_1})-\lambda_k\geq \begin{cases}\delta_{\J} g_{\J},&\quad k=j_1,\dots,j_2,\\
\lambda_{j_2}-\lambda_k,&\quad k>j_2,\end{cases}
\]
we get (recall that $\delta_\J\leq 1$)
\begin{align*}
\|T_{\geq j_1}ET_{\geq j_1}\|_\infty&\leq \|(\vert R_{{\J}^c}\vert ^{1/2}+(\delta_{\J} g_{\J})^{-1/2}P_{\J})E(\vert R_{{\J}^c}\vert ^{1/2}+(\delta_{\J} g_{\J})^{-1/2}P_{\J})\|_\infty\\
&\leq \delta_{\J}^{-1}\|(\vert R_{{\J}^c}\vert ^{1/2}+g_{\J}^{-1/2}P_{\J})E(\vert R_{{\J}^c}\vert ^{1/2}+g_{\J}^{-1/2}P_{\J})\|_\infty\leq 1.
\end{align*}
Similarly, we have $\|T_{\leq j_2}ET_{\leq j_2}\|_\infty\leq 1$, and the claim follows.
\end{proof}

The following lemma follows from simple properties of the operator norm.

\begin{lemma}\label{SimpleLemma}
We have
\begin{align*}
\max\Big(\|\vert R_{{\J}^c}\vert ^{1/2}E\vert R_{{\J}^c}\vert ^{1/2}\|_\infty,g_{\J}^{-1/2}\|\vert R_{{\J}^c}\vert ^{1/2}EP_{\J}\|_\infty,g_{\J}^{-1}\|P_{\J} E P_{\J}\|_\infty\Big)\leq \delta_{\J}.
\end{align*}
\end{lemma}

\begin{lemma}\label{LemER} If $\delta_{\J}<1/2$, then we have
\begin{align*}
    g_{\J}^{-1/2}\|\vert R_{{\J}^c}\vert ^{-1/2}\hat P_{{\J}}\|_2\leq \sqrt{\min(\vert {\J}\vert ,\vert {\J}^c\vert )}\frac{\delta_{\J}}{1-2\delta_{\J}}
\end{align*}
and
\begin{align*}
    \|P_{{\J}^c}\hat P_{{\J}}\|_2\leq \sqrt{\min(\vert {\J}\vert ,\vert {\J}^c\vert )}\frac{\delta_{\J}}{1-2\delta_{\J}},
\end{align*}
where $\vert R_{{\J}^c}\vert ^{-1/2}$ is the inverse of $\vert R_{{\J}^c}\vert ^{1/2}$ on the range of $P_{{\J}^c}$
\end{lemma}

\begin{proof}
By Lemma \ref{LemEVSep}, we have for every $k<j_1$,
\begin{align*}
\lambda_k-\hat\lambda_{j_1}&= \lambda_k-\lambda_{j_1}-(\hat\lambda_{j_1}-\lambda_{j_1})\\
&\geq \lambda_k-\lambda_{j_1}-\delta_{\J}(\lambda_{j_1-1}-\lambda_{j_1})\geq (1-\delta_{\J})(\lambda_k-\lambda_{j_1})
\end{align*}
and thus also
\begin{align}\label{EqEVSep1}
\lambda_k-\hat\lambda_{j}\geq (1-\delta_{\J})(\lambda_k-\lambda_{j_1}),\qquad \forall k<j_1,\forall j\in {\J}.
\end{align}
Similarly, we have
\begin{align}\label{EqEVSep2}
\hat\lambda_{j}-\lambda_k\geq (1-\delta_{\J})(\lambda_{j_2}-\lambda_k),\qquad \forall k>j_2,\forall j\in {\J}.
\end{align}
By the identity $(\hat\lambda_j-\lambda_k)P_k\hat P_j=P_kE\hat P_j$, valid for every $j,k\geq 1$, we get
\begin{align}
    &\|\vert R_{{\J}^c}\vert ^{-1/2}\hat P_{{\J}}\|_2^2\nonumber\\
    &=\sum_{k< j_1}\sum_{j\in {\J}}\frac{\lambda_k-\lambda_{j_1}}{(\lambda_k-\hat\lambda_j)^2}\|P_kE\hat P_j\|_2^2+\sum_{k> j_2}\sum_{j\in {\J}}\frac{\lambda_{j_2}-\lambda_{k}}{(\hat\lambda_j-\lambda_k)^2}\|P_kE\hat P_j\|_2^2.\label{EqBoundER1}
\end{align}
Inserting \eqref{EqEVSep1} and \eqref{EqEVSep2} into \eqref{EqBoundER1}, we arrive at
\begin{align*}
    \|\vert R_{{\J}^c}\vert ^{-1/2}\hat P_{{\J}}\|_2\leq \frac{\|\vert R_{{\J}^c}\vert ^{1/2}E\hat P_{{\J}}\|_2}{1-\delta_{\J}}.
\end{align*}
Moreover, we have
\begin{align*}
    \|\vert R_{{\J}^c}\vert ^{1/2}E\hat P_{{\J}}\|_2&\leq \|\vert R_{{\J}^c}\vert ^{1/2}EP_{\J}\hat P_{{\J}}\|_2+\|\vert R_{{\J}^c}\vert ^{1/2}EP_{{\J}^c}\hat P_{{\J}}\|_2\\
    &=\|\vert R_{{\J}^c}\vert ^{1/2}EP_{\J}\hat P_{{\J}}\|_2+\|\vert R_{{\J}^c}\vert ^{1/2}E\vert R_{{\J}^c}\vert ^{1/2}\vert R_{{\J}^c}\vert ^{-1/2}\hat P_{{\J}}\|_2\\
    &\leq \|\vert R_{{\J}^c}\vert ^{1/2}EP_{\J}\|_2+\|\vert R_{{\J}^c}\vert ^{1/2}E\vert R_{{\J}^c}\vert ^{1/2}\|_\infty\| \vert R_{{\J}^c}\vert^{-1/2}\hat P_{{\J}}\|_2.
\end{align*}
Combining these two estimates with Lemma \ref{SimpleLemma}, we arrive at
\begin{align*}
    \|\vert R_{{\J}^c}\vert ^{-1/2}\hat P_{{\J}}\|_2\leq \frac{\|\vert R_{{\J}^c}\vert ^{1/2}EP_{\J}\|_2}{1-\delta_{\J}}+ \frac{\delta_{\J} \| \vert R_{{\J}^c}\vert^{-1/2}\hat P_{{\J}}\|_2}{1-\delta_{\J}},
\end{align*}
and the first claim follows inserting
\begin{align}
    g_{\J}^{-1/2}\|\vert R_{{\J}^c}\vert ^{1/2}EP_{\J}\|_2&\leq \sqrt{\min(\vert {\J}\vert ,\vert {\J}^c\vert )} g_{\J}^{-1/2}\|\vert R_{{\J}^c}\vert ^{1/2}EP_{\J}\|_\infty\nonumber\\
    &\leq \sqrt{\min(\vert {\J}\vert ,\vert {\J}^c\vert )}\delta_{\J},\label{EqBoundER2}
\end{align}
where we again applied Lemma \ref{SimpleLemma} in the second inequality. To get the second claim, note that
\begin{align*}
    \|P_{{\J}^c}\hat P_{{\J}}\|_2&=\|\vert R_{{\J}^c}\vert ^{1/2}\vert R_{{\J}^c}\vert ^{-1/2}\hat P_{{\J}}\|_2\\
    &\leq \|\vert R_{{\J}^c}\vert ^{1/2}\|_{\operatorname{\infty}}\|\vert R_{{\J}^c}\vert ^{-1/2}\hat P_{{\J}}\|_2\leq g_{\J}^{-1/2}\|\vert R_{{\J}^c}\vert ^{-1/2}\hat P_{{\J}}\|_2
\end{align*}
and the second claim follows from inserting the first claim.
\end{proof}

In what follows, we abbreviate
\begin{align}\label{defn:gamma}
\gamma_{\J}:=\Big\|\sum_{k\in {\J}^c}\sum_{j\in {\J}}\frac{P_kEP_j}{\lambda_k-\lambda_j}\Big\|_\infty.
\end{align}

\begin{lemma}\label{LemFirstOrderExp}
Suppose that $\delta_{\J}<1/2$. Then we have
\begin{align*}
    &\bigg\|P_{{\J}^c}\hat P_{\J}+\sum_{k\in {\J}^c}\sum_{j\in {\J}}\frac{P_kEP_j}{\lambda_k-\lambda_j}\bigg\|_2\\
    &\leq \sqrt{\min(\vert {\J}\vert ,\vert {\J}^c\vert )}\bigg(\frac{2\delta_{\J}^2}{1-2\delta_{\J}}+\frac{\delta_{\J}\gamma_{\J}}{(1-2\delta_{\J})(1-\delta_{\J})}\bigg).
\end{align*}
In particular, we have
\begin{align*}
    &\bigg\|P_{{\J}^c}\hat P_{\J}+\sum_{k\in {\J}^c}\sum_{j\in {\J}}\frac{P_kEP_j}{\lambda_k-\lambda_j}\bigg\|_2\\
    &\leq \min(\vert {\J}\vert ,\vert {\J}^c\vert )\bigg(\frac{2\delta_{\J}^2}{1-2\delta_{\J}}+\frac{\delta_{\J}^2}{(1-2\delta_{\J})(1-\delta_{\J})}\bigg).
\end{align*}
\end{lemma}

\begin{proof}
Proceeding as in (4.10)--(4.11) in \cite{JW18}, we have
\begin{align*}
    P_{{\J}^c}\hat P_{\J}+\sum_{k\in {\J}^c}\sum_{j\in {\J}}\frac{P_kEP_j}{\lambda_k-\lambda_j}
    &=\sum_{k\in {\J}^c}\sum_{j\in {\J}}\frac{P_kEP_j\hat P_{{\J}^c}}{\lambda_k-\lambda_j}-\sum_{k\in {\J}^c}\sum_{j\in {\J}}\frac{P_kEP_{{\J}^c}\hat P_j}{\lambda_k-\hat \lambda_j}\\
    &+\sum_{k\in {\J}^c}\sum_{j\in {\J}}\sum_{l\in {\J}}\frac{P_kE P_jE\hat P_l}{(\lambda_k-\lambda_j)(\lambda_k-\hat \lambda_l)}.
\end{align*}
Hence, by the triangular inequality, we get
\begin{align*}
    &\bigg\|P_{{\J}^c}\hat P_{\J}+\sum_{k\in {\J}^c}\sum_{j\in {\J}}\frac{P_kEP_j}{\lambda_k-\lambda_j}\bigg\|_2\\
    &\leq \bigg\|\sum_{k\in {\J}^c}\sum_{j\in {\J}}\frac{P_kEP_j\hat P_{{\J}^c}}{\lambda_k-\lambda_j}\bigg\|_2+\bigg\|\sum_{k\in {\J}^c}\sum_{j\in {\J}}\frac{P_kEP_{{\J}^c}\hat P_j}{\lambda_k-\hat \lambda_j}\bigg\|_2\\
    &+\bigg\|\sum_{k\in {\J}^c}\sum_{j\in {\J}}\sum_{l\in {\J}}\frac{P_kE P_jE\hat P_l}{(\lambda_k-\lambda_j)(\lambda_k-\hat \lambda_l)}\bigg\|_2=:I_1+I_2+I_3.
\end{align*}
First, we have
\begin{align*}
    I_1=\bigg\|\sum_{k\in {\J}^c}\sum_{j\in {\J}}\frac{P_kEP_j}{\lambda_k-\lambda_j}P_{{\J}}\hat P_{{\J}^c}\bigg\|_2\leq \bigg\|\sum_{k\in {\J}^c}\sum_{j\in {\J}}\frac{P_kEP_j}{\lambda_k-\lambda_j}\bigg\|_\infty\|P_{{\J}}\hat P_{{\J}^c}\|_2
\end{align*}
Inserting $\|P_{{\J}}\hat P_{{\J}^c}\|_2=\|P_{{\J}^c}\hat P_{{\J}}\|_2$ and the second bound in Lemma \ref{LemER}, we get
\begin{align*}
    I_1\leq \sqrt{\min(\vert {\J}\vert ,\vert {\J}^c\vert )}\frac{\delta_{\J}\gamma_{\J}}{1-2\delta_{\J}},
\end{align*}
where $\gamma_{\J}$ is defined in \eqref{defn:gamma}. Second, by Lemma \ref{LemEVSep}, proceeding as in the proof of Lemma \ref{LemER}, we get
\begin{align*}
    I_2&\leq \frac{g_{\J}^{-1/2}}{1-\delta_{\J}}\|\vert R_{{\J}^c}\vert ^{1/2}EP_{{\J}^c}\hat P_{\J}\|_2=\frac{g_{\J}^{-1/2}}{1-\delta_{\J}}\|\vert R_{{\J}^c}\vert ^{1/2}E\vert R_{{\J}^c}\vert ^{1/2}\vert R_{{\J}^c}\vert ^{-1/2}\hat P_{\J}\|_2\\
    &\leq \frac{g_{\J}^{-1/2}}{1-\delta_{\J}}\|\vert R_{{\J}^c}\vert ^{1/2}E\vert R_{{\J}^c}\vert ^{1/2}\|_\infty\|\vert R_{{\J}^c}\vert ^{-1/2}\hat P_{\J}\|_2.
\end{align*}
Inserting Lemmas \ref{SimpleLemma} and \ref{LemER}, we get
\begin{align*}
    I_2\leq \sqrt{\min(\vert {\J}\vert ,\vert {\J}^c\vert )}\frac{\delta_{\J}^2}{(1-2\delta_{\J})(1-\delta_{\J})}.
\end{align*}
Finally,
\begin{align*}
    I_3&\leq \frac{g_{\J}^{-1}}{1-\delta_{\J}}\bigg\|\sum_{k\in {\J}^c}\sum_{j\in {\J}}\frac{P_kE P_jE\hat P_{\J}}{\lambda_k-\lambda_j}\bigg\|_2\\
    &\leq \frac{g_{\J}^{-1}}{1-\delta_{\J}}\bigg\|\sum_{k\in {\J}^c}\sum_{j\in {\J}}\frac{P_kE P_j}{\lambda_k-\lambda_j}P_{\J}E P_{\J}\bigg\|_2\\
    &+\frac{g_{\J}^{-1}}{1-\delta_{\J}}\bigg\|\sum_{k\in {\J}^c}\sum_{j\in {\J}}\frac{P_kE P_j}{\lambda_k-\lambda_j}P_{\J} EP_{{\J}^c}\hat P_{\J}\bigg\|_2\\
    &\leq \frac{g_{\J}^{-1}}{1-\delta_{\J}}\bigg\|\sum_{k\in {\J}^c}\sum_{j\in {\J}}\frac{P_kE P_j}{\lambda_k-\lambda_j}\bigg\|_2\|P_{\J} E P_{\J}\|_\infty\\
    &+\frac{g_{\J}^{-1}}{1-\delta_{\J}}\bigg\|\sum_{k\in {\J}^c}\sum_{j\in {\J}}\frac{P_kE P_j}{\lambda_k-\lambda_j}\bigg\|_\infty\|P_{\J} E\vert R_{{\J}^c}\vert ^{1/2}\|_\infty\|\vert R_{{\J}^c}\vert ^{-1/2}\hat P_{\J}\|_2.
\end{align*}
Applying Lemmas \ref{SimpleLemma} and \ref{LemER} and the inequality
\begin{align}\label{EqBoundLinTerm}
\bigg\|\sum_{k\in {\J}^c}\sum_{j\in {\J}}\frac{P_kEP_j}{\lambda_k-\lambda_j}\bigg\|_2\leq g_{\J}^{-1/2}\|\vert R_{{\J}^c}\vert ^{1/2}EP_{\J}\|_2\leq \sqrt{\min(\vert {\J}\vert ,\vert {\J}^c\vert )}\delta_{\J},
\end{align}
we get
\begin{align*}
    I_3&\leq \sqrt{\min(\vert {\J}\vert ,\vert {\J}^c\vert )}\Big(\frac{\delta_{\J}^2}{1-\delta_{\J}}+\frac{\delta_{\J}^2\gamma_{\J}}{(1-2\delta_{\J})(1-\delta_{\J})}\Big).
\end{align*}
Collecting these bounds for $I_1,I_2,I_3$, we get
\begin{align*}
	&\bigg\|P_{{\J}^c}\hat P_{\J}+\sum_{k\in {\J}^c}\sum_{j\in {\J}}\frac{P_kEP_j}{\lambda_k-\lambda_j}\bigg\|_2\\
	&\leq \sqrt{\min(\vert {\J}\vert ,\vert {\J}^c\vert )}\bigg(\frac{\delta_{\J}^2}{1-\delta_{\J}}+\frac{\delta_{\J}^2+\delta_{\J}^2\gamma_{\J}}{(1-2\delta_{\J})(1-\delta_{\J})}+\frac{\delta_{\J}\gamma_{\J}}{1-2\delta_{\J}}\bigg)\\
	&=\sqrt{\min(\vert {\J}\vert ,\vert {\J}^c\vert )}\bigg(\frac{2\delta_{\J}^2}{1-2\delta_{\J}}+\frac{\delta_{\J}\gamma_{\J}}{(1-2\delta_{\J})(1-\delta_{\J})}\bigg).
\end{align*}
By \eqref{EqBoundLinTerm}, we have
\begin{align*}
\gamma_{\J}\leq \bigg\|\sum_{k\in {\J}^c}\sum_{j\in {\J}}\frac{P_kEP_j}{\lambda_k-\lambda_j}\bigg\|_2\leq \sqrt{\min(\vert {\J}\vert ,\vert {\J}^c\vert )}\delta_{\J},
\end{align*}
and the second claim follows from inserting this into the first one.
\end{proof}

\begin{proof}[Proof of Proposition \ref{PropExp}]
In order to prove \eqref{Eq0Exp}, assume first that $\delta_{\J}<1/4$. Then it follows from $\|P_{{\J}}-\hat P_{\J}\|_2^2=2\|P_{{\J}^c}\hat P_{\J}\|_2^2$ and the second inequality in Lemma \ref{LemER} that \eqref{Eq0Exp} holds. On the other hand, we always have $\|P_{{\J}}-\hat P_{\J}\|_2^2\leq 2\min(\vert {\J}\vert ,\vert {\J}^c\vert )$, implying \eqref{Eq0Exp} also in the case that $\delta_{\J}\geq 1/4$.

It remains to prove \eqref{Eq2Exp}. Applying the identities $I=P_{{\J}}+P_{{\J}^c}= \hat P_{\J}+\hat P_{\J^c}$ and $P_{\J}P_{\J^c}=0$, we have
\begin{align*}
    \hat P_{\J}-P_{\J}&=\hat P_{\J}P_{\J^c}-\hat P_{\J^c}P_{\J}\\
    &=\hat P_{\J}P_{\J^c}-P_{\J^c}\hat P_{\J^c}P_{\J}-P_{\J}\hat P_{\J^c} P_{\J}\\
    &=\hat P_{\J}P_{\J^c}+P_{\J^c}\hat P_{\J} P_{\J}-P_{\J}\hat P_{\J^c} P_{\J}\\
    &=\hat P_{\J}P_{\J^c}+P_{\J^c}\hat P_{\J} -P_{\J^c}\hat P_{\J}P_{\J^c}-P_{\J}\hat P_{\J^c}P_{\J}.
\end{align*}
Combining this with the triangle inequality, we obtain
\begin{align*}
    &\|\hat{P}_{\J} - P_{\J} - L_{\J}E\|_2\\
    &\leq \|\hat P_{\J}P_{\J^c}+P_{\J^c}\hat P_{\J}-L_{\J}E\|_2+\|P_{\J^c}\hat P_{\J}\hat P_{\J}P_{\J^c}\|_2+\|P_{\J}\hat P_{\J^c}\hat P_{\J^c} P_{\J}\|_2.
\end{align*}
Inserting the definition of the linear term in \eqref{eq_def_linear_term}, we get
\begin{align*}
    &\|\hat P_{\J}P_{\J^c}+P_{\J^c}\hat P_{\J} - L_{\J}E\|_2\\
    &\leq \Big\|\hat P_{\J}P_{\J^c}+\sum_{j\in \J}\sum_{k\in \J^c}\frac{P_j E P_k}{\lambda_k-\lambda_j} \Big\|_2+\Big\|P_{\J^c}\hat P_{\J}+\sum_{j\in \J}\sum_{k\in \J^c}\frac{P_k E P_j}{\lambda_k-\lambda_j} \Big\|_2\\
    &=2\Big\|P_{\J^c}\hat P_{\J}+\sum_{j\in \J}\sum_{k\in \J^c}\frac{P_k E P_j}{\lambda_k-\lambda_j} \Big\|_2.
\end{align*}
We again start with the case $\delta_{\J}<1/4$. First, by Lemma \ref{LemER} and the identity $\|\hat P_{\J}P_{\J^c}\|_2=\|\hat P_{\J^c} P_{\J}\|_2$, we get
\begin{align*}
    \|P_{\J^c}\hat P_{\J}\hat P_{\J}P_{\J^c}\|_2\leq \|\hat P_{\J}P_{\J^c}\|_2^2\leq 4\min(\vert \J\vert ,\vert \J^c\vert )\delta_{\J}^2
\end{align*}
and
\begin{align*}
    \|P_{\J}\hat P_{\J^c}\hat P_{\J^c}P_{\J}\|_2\leq \|\hat P_{\J^c}P_{\J}\|_2^2\leq 4\min(\vert \J\vert ,\vert \J^c\vert )\delta_{\J}^2.
\end{align*}
Second, by Lemma \ref{LemFirstOrderExp}, we have
\begin{align*}
    2\Big\|P_{\J^c}\hat P_{\J}+\sum_{j\in \J}\sum_{k\in \J^c}\frac{P_k E P_j}{\lambda_k-\lambda_j}  \Big\|_2\leq 2\Big(4+\frac{8}{3}\Big)\min(\vert \J\vert ,\vert \J^c\vert )\delta_{\J}^2.
\end{align*}
Hence, in the case $\delta_{\J}< 1/4$, we arrive at
\begin{align*}
    \|\hat{P}_{\J} - P_{\J} - L_{\J}E\|_2\leq \Big(8+\frac{16}{3}+8\Big)\min(\vert \J\vert ,\vert \J^c\vert )\delta_{\J}^2.
\end{align*}
Finally, if $\delta_{\J}\geq 1/4$, then we can use $\|P_{\J}-\hat P_{\J}\|_2\leq \sqrt{2}\min(\vert \J\vert ,\vert \J^c\vert )^{1/2}$ and \eqref{EqBoundLinTerm} to obtain
\begin{align*}
    \|\hat{P}_{\J} - P_{\J} - L_{\J}E\|_2&\leq \|\hat{P}_{\J} - P_{\J}\|_2 + \|L_{\J}E\|_2\\
    &\leq \min(\vert \J\vert ,\vert \J^c\vert )^{1/2}(\sqrt{2}+\sqrt{2}\delta_{\J})\\
    &\leq\min(\vert \J\vert ,\vert \J^c\vert )(16\sqrt{2}+4\sqrt{2})\delta_\J^2.
\end{align*}
This completes the proof.
\end{proof}

\begin{proof}[Proof of Corollary \ref{cor:exp:quadratic:term}]
By the Cauchy-Schwarz inequality, we have
\begin{align*}
    &\big\vert\|P_{\J}-\hat P_\J\|_2^2-\|L_{\J}E\|_2^2\big\vert\\
    &=\big\vert\|P_{\J}-\hat P_\J-L_{\J}E+L_{\J}E\|_2^2-\|L_{\J}E\|_2^2\big\vert\\
    &\leq \|P_{\J}-\hat P_\J-L_{\J}E\|_2^2+2\|P_{\J}-\hat P_\J-L_{\J}E\|_2\|L_{\J}E\|_2.
\end{align*}
Hence, the claim follows from inserting the second claim of Proposition \ref{PropExp} and \eqref{EqBoundLinTerm}.
\end{proof}

\begin{proof}[Proof of Corollary \ref{lem:boot:approx}]
By the triangular inequality, the inequality $(a+b)^2 \leq 2a^2 + 2b^2$ and the Cauchy-Schwarz inequality, it follows that
\begin{align*}
&\big\vert\|\tilde{P}_{\J} - \hat{P}_{\J} \|_2^2 - \|\FO_{\J}(E - \tilde{E})\|_2^2\big\vert \\
& \leq \|\tilde{P}_{\J} - P_{\J} - \FO_{\J}\tilde{E} - (\hat{P}_{\J} - P_{\J} - \FO_{\J}E)\|_2^2\\
&+2 \big\vert\langle \tilde{P}_{\J} - P_{\J} - \FO_{\J}\tilde{E} - (\hat{P}_{\J} - P_{\J} - \FO_{\J}E),  \FO_{\J}(E - \tilde{E})\rangle\big\vert\\
&\leq 2\big\|\tilde{P}_{\J} - P_{\J} - \FO_{\J}\tilde{E}\big\|_2^2 + 2\big\|\hat{P}_{\J} - P_{\J} - \FO_{\J}E\big\|_2^2 \\
&+ 2\big(\|\tilde{P}_{\J} - P_{\J} - \FO_{\J}\tilde{E} \|_2 + \|  \hat{P}_{\J} - P_{\J} - \FO_{\J}E \|_2\big) \big(\| \FO_{\J}E \|_2 + \| \FO_{\J}\tilde{E} \|_2\big).
\end{align*}
Hence, the claim follows from inserting the second claim of Proposition \ref{PropExp} and \eqref{EqBoundLinTerm} applied to both $\hat\Sigma$ and $\tilde\Sigma$.
\end{proof}

\subsection{Proofs for Section \ref{sec:metrics}}\label{sec:metrics:proofs}

\begin{lemma}\label{lem:be:1}
Consider Setting \ref{ass:clts:setup}. Then we have
\begin{align*}
\mathbf{U}\Big(n^{-1}\|\sum_{i = 1}^n Y_i \|^2 + V, \|Z\|^2 + V \Big) \lesssim v^{-3}n^{-q/2}\sum_{i = 1}^n \E \|Y_i\|^q.
\end{align*}
Here, $V \sim\mathcal{N}(0,v^2)$, $v \in (0,1]$, is independent of the $Y_i$ and $Z$.
\end{lemma}

\begin{proof}[Proof of Lemma \ref{lem:be:1}]
The proof is mainly relying on methods given in~\cite{bentkus1984asymptotic}.
For a random variable $Y$ in $\mathcal{H}$, let $Y^{\delta} = Y \ind_{\|Y\| \leq \delta \sqrt{n}}$.
Then by properties of the uniform distance (or a simple conditioning argument), we have
\begin{align}\nonumber \label{eq:lem:be:1:1}
&\mathbf{U}\Big(n^{-1}\|\sum_{i = 1}^n Y_i\|^2 + V, n^{-1}\|\sum_{i = 1}^n Y_i^{1}\|^2 + V \Big) \leq \mathbf{U}\Big(n^{-1}\|\sum_{i = 1}^n Y_i\|^2, n^{-1}\|\sum_{i = 1}^n Y_i^{1} \|^2 \Big) \\&\leq \sum_{i = 1}^n \P\big( Y_i^1 \neq Y_i \big) = \sum_{i = 1}^n \P\big(\|Y_i\| > \sqrt{n}\big) \leq n^{-q/2}\sum_{i = 1}^n \E \|Y_i\|^q.
\end{align}
An application of Esseens smoothing inequality (cf.~\cite[Lemma 1, XVI.3]{fellervolume2}) yields
\begin{align}\nonumber
&\mathbf{U}\Big(n^{-1}\|\sum_{i = 1}^n Y_i^{1}\|^2 + V, \|Z\|^2 + V \big) \\&\lesssim \int_{\R}\Big\vert\frac{\E e^{\mathrm{i}\xi n^{-1}\|\sum_{i = 1}^n Y_i^{1}\|^2} - \E e^{\mathrm{i}\xi \|Z\|^2}}{\xi}\Big\vert e^{-v^2 \xi^2/2} d \xi.
\end{align}
Next, proceeding as in~\cite{ulyanov1987asymptotic} (see Equation (86), resp. Theorem 4.6 in ~\cite{bentkus1984asymptotic}), one derives that
\begin{align}\nonumber \label{eq:lem:be:1:2}
&\Big\vert\E e^{\mathrm{i}\xi n^{-1}\|\sum_{i = 1}^n Y_i^{1} \|^2} - \E e^{\mathrm{i}\xi \|Z \|^2}\Big\vert \\&\lesssim \sum_{i = 1}^n (\vert \xi\vert  + \vert \xi\vert ^3) \big(\E n^{-1}\|Y_i\|^2 \ind_{\|Y_i\| \geq \sqrt{n}}  + \E n^{-3/2}\|Y_i\|^3 \ind_{\|Y_i\| \leq \sqrt{n}} \big).
\end{align}
Since $q\in(2,3]$, we have
\begin{align*}
    &\E n^{-1}\|Y_i\|^2 \ind_{\|Y_i\| \geq \sqrt{n}}\leq n^{-q/2}\E \|Y_i\|^q
\end{align*}
and
\begin{align*}
    \E n^{-3/2}\|Y_i\|^3\ind_{\|Y_i\| \leq \sqrt{n}}\leq n^{-q/2}\E \|Y_i\|^q.
\end{align*}
It follows that (recall $v\in(0,1]$)
\begin{align*}
&\int_{\R}\Big\vert\frac{\E e^{\mathrm{i}\xi n^{-1}\|\sum_{i = 1}^n Y_i^{1} \|^2} - \E e^{\mathrm{i}\xi \|Z\|^2}}{\xi} \Big\vert e^{-v^2 \xi^2/2} d \xi \\& \lesssim  v^{-3} \int_{\R} (1  \vee \vert \xi \vert^2) n^{-q/2}\sum_{i = 1}^n \E \|Y_i\|^q e^{-\xi^2/2} d \xi \\&\lesssim  v^{-3} n^{-q/2}\sum_{i = 1}^n \E \|Y_i\|^q.
\end{align*}
This completes the proof.
\end{proof}

\begin{lemma}[\cite{pinelis:1994}, Theorem 4.1]\label{lem:rosenthal}
Let $(Y_i) \in \mathcal{H}$ be an independent sequence with $\E Y_i = 0$ and $\E \|Y_i\|^{q}<\infty$ for all $i$ and $q\geq 1$. Then,
\begin{align*}
\E \big\|\sum_{i = 1}^n Y_i \big\|^{q} &\lesssim_q \Big( \sum_{i = 1}^n \E \|Y_i\|^2 \Big)^{q/2} +  \sum_{i = 1}^n \E \big\|Y_i\big\|^{q}.
\end{align*}
\end{lemma}

\begin{lemma}\label{lem:be:3}
Consider Setting \ref{ass:clts:setup} and assume in addition that $\sum_{i = 1}^n \E \|Y_i\|^r \leq n^{r/2}$ for $r \in \{2,q\}$. Then
\begin{align*}
&\mathbf{W}\Big(n^{-1}\big\|\sum_{i = 1}^n Y_i \big\|^2 + V, \big\|Z\big\|^2 + V \Big)\lesssim_q v^{6/q-3} \Big(n^{-q/2}\sum_{i = 1}^n\E \big\|Y_i\big\|^q \Big)^{1 - 2/q}.
\end{align*}
Here, $V \sim\mathcal{N}(0,v^2)$, $v \in (0, 1]$, is independent of $Y_1,\dots,Y_n$ and $Z$.
\end{lemma}

\begin{proof}[Proof of Lemma \ref{lem:be:3}]
We use the well-known fact that for real-valued random variables $X,Y$, we have the dual representation
\begin{align*}
\mathbf{W}\big(X,Y\big) = \int_{\R} \big\vert\P(X \leq x\big) - \P\big(Y \leq x\big) \big\vert dx.
\end{align*}
We conclude that for any $a > 0$ and $s > 1$, Markov's inequality gives
\begin{align*}
\mathbf{W}\big(X,Y\big) &\leq 2 a \mathbf{U}\big(X,Y\big) + 2 \int_{a}^{\infty} \P\big(\vert X\vert  \geq x\big) + \P\big(\vert Y\vert  \geq x\big)dx\\&
\leq 2 a \mathbf{U}\big(X,Y\big) + 2 s^{-1}a^{1-s}\big( \E\vert X\vert ^s  + \E\vert Y\vert ^s\big).
\end{align*}
We apply this inequality with $s=q/2>1$. By assumption, we have $\E\|Z\|^2\leq 1$ and thus $\E \|Z\|^{q} \lesssim_q 1$ (cf.~Lemma \ref{lem:moment:computation:KL}). Lemma \ref{lem:be:1}, the inequality $(x+y)^{q/2} \leq 2^{q/2-1}\big(x^{q/2} + y^{q/2}\big)$, $x,y \geq 0$, together with Lemma \ref{lem:rosenthal} then yield
\begin{align*}
&\mathbf{W}\Big(n^{-1}\big\| \sum_{i = 1}^n Y_i \big\|^2 + V, \big\|Z \big\|^2 + V \Big)\\&\lesssim_q
a \mathbf{U}\Big(n^{-1}\big\|\sum_{i = 1}^n Y_i\big\|^2 + V, \big\|Z \big\|^2 + V \Big)\\&+  a^{1-q/2}\big( \E\|n^{-1/2}\sum_{i = 1}^n Y_i\|^{q}  + \E\|Z\|^{q} + \E \vert V\vert^{q/2}\big)\\&\lesssim_q a  v^{-3} n^{-q/2}\sum_{i = 1}^n \E \|Y_i\|^q + a^{1-q/2}.
\end{align*}
Selecting $a = v^{6/q} \big(n^{-q/2}\sum_{i = 1}^n\E \|Y_i\|^{q} \big)^{-2/q}$, the claim follows.
\end{proof}

\begin{lemma}\label{lem:W1:approx}
Consider Setting \ref{ass:clts:setup}. Suppose in addition that $\E\|Z\|^2\leq 1$ and $|T|\leq C_T$ almost surely. Then for any $u,s > 0$, we have
\begin{align*}
\mathbf{W}\big(T,\|Z\|^2\big) \lesssim_s  \mathbf{W}\big(S+V,\|Z\|^2+V\big) + v + C_T \P(\vert T-S \vert>u) +u+ C_T^{1-s}.
\end{align*}
Here, $V \sim \mathcal{N}(0,v^2)$, $v \in (0,1]$, is independent of $S$, $T$, $Z$.
\end{lemma}

\begin{proof}[Proof of Lemma \ref{lem:W1:approx}]
Let $\mathcal{E} = \{|T-S|\leq u \}$. By the triangle inequality and $\ind_{\mathcal{E}}\leq 1$, we have
\begin{align*}
\mathbf{W}\big(T,\|Z\|^2\big)& \leq \mathbf{W}\big(T,T\ind_{\mathcal{E}} \big) + \mathbf{W}\big(\|Z\|^2,\|Z\|^2\ind_{\mathcal{E}}\big) + \mathbf{W}\big(T\ind_{\mathcal{E}},\|Z\|^2\ind_{\mathcal{E}}\big)\\& \leq
\E \big\vert T\big\vert\ind_{\mathcal{E}^c} + \E \big\|Z\big\|^2\ind_{\mathcal{E}^c} +  \mathbf{W}\big(S\ind_{\mathcal{E}},\|Z\|^2\ind_{\mathcal{E}}\big) + \mathbf{W}\big(S\ind_{\mathcal{E}},T\ind_{\mathcal{E}} \big)\\&\leq
\E \big\vert T\big\vert\ind_{\mathcal{E}^c} + \E \big\|Z\big\|^2\ind_{\mathcal{E}^c} +  \mathbf{W}\big(S,\|Z\|^2\big) + u.
\end{align*}
Since again by the triangle inequality
\begin{align*}
    \mathbf{W}\big(S,\|Z\|^2\big) &\leq  \mathbf{W}\big(S+V,\|Z\|^2+V\big) + 2 \E \big\vert V\big\vert,
\end{align*}
we get
\begin{align*}
\mathbf{W}\big(T,\|Z\|^2\big) &\leq \E \big\vert T\big\vert\ind_{\mathcal{E}^c} + \E \big\|Z\big\|^2\ind_{\mathcal{E}^c} +  u + \mathbf{W}\big(S+V,\|Z\|^2+V\big) + 2\E \big\vert V\big\vert.
\end{align*}
Since $\E \|Z\|^2 \leq  1$, it follows that $\E \|Z\|^{2s+2} \lesssim_s 1$ for all $s>0$ (cf.~Lemma \ref{lem:moment:computation:KL}) and thus
\begin{align*}
\E \big\|Z\big\|^2 \ind_{\|Z\|^2 \geq C_T} &\leq C_T^{-s}\E \|Z\big\|^{2s+2} \lesssim_s C_T^{-s}.
\end{align*}
Using this bound, we arrive at
\begin{align*}
\E \big\vert T\big\vert\ind_{\mathcal{E}^c} + \E \big\|Z\big\|^2\ind_{\mathcal{E}^c} &\leq 2C_T \P\big(\mathcal{E}^c\big) + \E \big\|Z\big\|^2 \ind_{\|Z\|^2 \geq C_T} \\&\lesssim_s C_T \P\big(\mathcal{E}^c\big) + C_T^{-s},
\end{align*}
and the claim follows from $\E\vert V\vert  \leq v$.
\end{proof}

\begin{proof}[Proof of Proposition \ref{prop:clt:II}]
The claim follows by combining Lemmas \ref{lem:be:3} and \ref{lem:W1:approx} and balancing with respect to $v$. 
\end{proof}

\begin{lemma}\label{lem:BE:gauss}
Let $Y_1,\dots,Y_n$ be independent random variables satisfying $\E Y_i = 0$, $\E Y_i^2 = 1$ and $\E \vert Y_i\vert^3 < \infty$. For $a_i \in \R$, let $A_n = \sum_{i = 1}^n a_i^2$. Then
\begin{align*}
\mathbf{U}\Big(A_n^{-1/2}\sum_{i = 1}^n a_i Y_i , G\Big) \lesssim A_n^{-3/2} \sum_{i = 1}^n \vert a_i\vert ^3 \E \vert Y_i\vert ^3.
\end{align*}
\end{lemma}

\begin{proof}[Proof of Lemma \ref{lem:BE:gauss}]
This immediately follows from a general version of the Berry-Esseen Theorem for non-identically distributed random variables, see e.g.~\cite{petrov_book_1995}.
\end{proof}

\begin{lemma}\label{lem:be:unif:approx:2}
Let $G \sim \mathcal{N}(0,1)$ and $T$, $S$ be real-valued random variables. Then, for any $a,b > 0$, $c \in \R$,
\begin{align*}
\mathbf{U}\big((T-c)/b, G\big) \lesssim\mathbf{U}\big((S-c)/b, G\big) + \P\big(\vert T-S\vert > a\big) + a/b.
\end{align*}
\end{lemma}

\begin{proof}[Proof of Lemma \ref{lem:be:unif:approx:2}]
Recall first that the uniform metric $\mathbf{U}$ is invariant with respect to affine transformations of the underlying random variables.
Next, note that
\begin{align*}
\P\big(T \leq x\big) \leq \P\big(S \leq x + a\big) + \P\big(\vert T-S\vert > a\big).
\end{align*}
A corresponding lower bound is also valid. Therefore, using the above mentioned affine invariance and the fact that the distribution function of $G$ is Lipschitz continuous, the claim follows.
\end{proof}

Next, we recall the following smoothing inequality (cf.~\cite{senatov:book}, Lemma 4.2.1).

\begin{lemma}\label{lem:smoothing:senatov}
Let $X$ be a real-valued random variable, $G \sim \mathcal{N}(0,\sigma^2)$, $\sigma > 0$, and $G_{\epsilon} \sim \mathcal{N}(0,\epsilon^2)$, $\epsilon > 0$, independent of $X,G$. Then there exits an absolute constant $C$, such that
\begin{align*}
\mathbf{U}\big(X, G \big) \leq \mathbf{U}\big(X + G_{\epsilon}, G + G_{\epsilon} \big) + C \frac{\epsilon}{\sigma}.
\end{align*}
\end{lemma}

\begin{proof}[Proof of Proposition \ref{prop:clt:III}]
Due to Lemma \ref{lem:be:unif:approx:2} (with $c = A$, $b = B$), it suffices to derive a bound for
\begin{align*}
\mathbf{U}\big((S - A)/B, G\big).
\end{align*}
For any $\epsilon > 0$, let $G_{\epsilon} \sim \mathcal{N}(0,\epsilon^2)$. By Lemma \ref{lem:smoothing:senatov} and the triangle inequality, we get
\begin{align*}
&\mathbf{U}\big((S - A)/B, G\big) \\&\lesssim \mathbf{U}\big((S - A)/B + G_{\epsilon}, G+G_{\epsilon}\big) + \epsilon \\&\lesssim \mathbf{U}\big((S - A)/B+ G_{\epsilon}, (\|Z\|^2 - A)/B + G_{\epsilon}\big) \\&+ \mathbf{U}\big((\|Z\|^2 - A)/B + G_{\epsilon}, G + G_{\epsilon}\big) + \epsilon.
\end{align*}
By the affine invariance of $\mathbf{U}$, we have
\begin{align*}
&\mathbf{U}\big((S - A)/B + G_{\epsilon}, (\|Z\|^2 - A)/B + G_{\epsilon}\big) \\&=
\mathbf{U}\big(S/A+ G_{\epsilon}B/A,\|Z\|^2/A+ G_{\epsilon}B/A \big).
\end{align*}
Setting $V = G_{\epsilon}B/A$ implies $v^2 = \E V^2 =  (B \epsilon/A)^2 $. Hence an application of Lemma \ref{lem:be:1} yields
\begin{align*}
\mathbf{U}\big(S/A+ G_{\epsilon}B/A,\|Z\|^2/A+ G_{\epsilon}B/A \big) \lesssim \Big(\frac{A }{\epsilon B} \Big)^3 n^{-q/2}\sum_{i = 1}^n \frac{\E \|Y_i\|^q}{A^{q/2}}.
\end{align*}
On the other hand, by the regularity of $\mathbf{U}$ (or explicitly by a simple conditioning argument) and Lemma \ref{lem:BE:gauss} (applied with $n=1$), we obtain the bound
\begin{align*}
\mathbf{U}\big((\|Z\|^2 - A)/B + G_{\epsilon}, G + G_{\epsilon}\big) & \leq \mathbf{U}\big((\|Z\|^2 - A)/B, G\big) \lesssim \frac{C^3}{B^3}.
\end{align*}
Piecing everything together, the claim follows.
\end{proof}

For the next result, recall the definition of $\lambda_{1,j}(\Psi)$, given in \eqref{defn:lambda:product}.

\begin{lemma}[Lemma 3 in \cite{naumov_spokoiny_ptrf2019}]\label{lem:anticoncentration:spok}
Let $Z\in \mathcal{H}$ be a Gaussian random variable with $\E Z = 0$ and covariance operator $\Psi$. Then, for any $a>0$,
\begin{align*}
\sup_{x > 0}\P\big(x < \|Z\|^2 < x + a \big) \leq \frac{a}{\sqrt{\lambda_{1,2}(\Psi)}}.
\end{align*}
In particular, for any $x \geq 0$, we have
\begin{align*}
\limsup_{a \to 0} \Big\vert\frac{\P\big(\|Z\|^2 < x + a \big) - \P\big(\|Z\|^2 < x\big)}{a}\Big\vert \leq \frac{1}{\sqrt{\lambda_{1,2}(\Psi)}},
\end{align*}
that is, we have a uniform upper bound for the density of $\|Z\|^2$.
\end{lemma}

\begin{lemma}[Lemma 2 in \cite{naumov_spokoiny_ptrf2019}]\label{lem:gaussian:comparison:spok}
For $i \in \{1,2\}$, let $Z_i$ be Gaussian with $\E Z_i = 0$ and covariance operator $\Psi_i$. Then
\begin{align*}
\mathbf{U}(\|Z_1\|^2,\|Z_2\|^2)\lesssim \Big(\frac{1}{\sqrt{\lambda_{1,2}(\Psi_1) }} +  \frac{1}{\sqrt{\lambda_{1,2}(\Psi_2)}} \Big)\big\|\Psi_1 - \Psi_2\big\|_1.
\end{align*}
\end{lemma}

\begin{lemma}\label{lem:Be:ulyanov+spok}
Consider Setting \ref{ass:clts:setup} with $q=3$. Then, we have
\begin{align*}
\mathbf{U}\Big(n^{-1}\big\|\sum_{i = 1}^n Y_i\big\|^2, \big\|Z \big\|^2 \Big) \lesssim n^{-3/5}\Big(\sum_{i = 1}^n \frac{\E \|Y_i\|^3}{\lambda_{1,2}^{3/4}(\Psi)}\Big)^{2/5}.
\end{align*}
\end{lemma}

\begin{proof}[Proof of Lemma \ref{lem:Be:ulyanov+spok}]
We argue very similarly as in the proof of Lemma \ref{lem:be:1}, using also the notation therein. By Esseen's smoothing inequality (cf.~\cite[Lemma 1, XVI.3]{fellervolume2}) together with Lemma \ref{lem:anticoncentration:spok}, we have for any $a > 0$
\begin{align}\nonumber
&\mathbf{U}\Big(n^{-1}\big\|\sum_{i = 1}^n (Y_i/a)^{1}\big\|^2, \big\|Z/a \big\|^2 \Big) \\&\nonumber \lesssim \int_{-t}^{t} \Big\vert\frac{\E e^{\mathrm{i}\xi n^{-1}\|\sum_{i = 1}^n (Y_i/a)^{1}\|^2} - \E e^{\mathrm{i}\xi \|Z/a\|^2}}{\xi}\Big\vert d \xi + \frac{a^2}{\sqrt{\lambda_{1,2}(\Psi)}t} \\&\lesssim t^3 n^{-3/2} \sum_{i = 1}^n \frac{\E \|Y_i\|^3}{a^{3}} + \frac{a^2}{\sqrt{\lambda_{1,2}(\Psi)}t}\label{eq:ES:two:eigenvalues}.
\end{align}
Setting $t=1$ and using also \eqref{eq:lem:be:1:1}, the claim follows from the (optimal) choice $a = (n^{-3/2} \sum_{i = 1}^n \E \|Y_i\|^3 \sqrt{\lambda_{1,2}(\Psi)})^{1/5}$.
\end{proof}

Finally, let us mention the following result of \cite{SAZONOV1989304}.

\begin{lemma}\label{lem:ulyanov:optim}
Consider Setting \ref{ass:clts:setup} with $q= 3$. Then
\begin{align*}
\mathbf{U}\Big(n^{-1}\big\|\sum_{i = 1}^n Y_i \big\|^2, \big\|Z\big\|^2\Big) &\lesssim \Big(n^{-1/2} \lambda_6^{-3}(\Psi) \frac{(A \, C)^3}{B^3}\Big)^{1 + 1/10} \\&+  n^{-1/2}\lambda_{1,6}^{-1}(\Psi) \frac{(A^2 \, C)^3}{B^3}.
\end{align*}
\end{lemma}

\begin{proof}[Proof of Proposition \ref{prop:clt:V}]
There exists a constant $C > 0$ such that, for all $x \geq 0$,
\begin{align*}
\P\big(T \leq x\big) &\leq \P\big(S \leq x + u\big) + \P\big(\vert T-S\vert  > u \big) \\&\leq \P\big(\|Z\|^2 \leq x + u\big) + \mathbf{U}\big(S,\|Z\|^2 \big) + \P\big(\vert T-S\vert  > u \big)
\\&\leq \P\big(\|Z\|^2 \leq x\big) + \frac{C u}{\sqrt{\lambda_{1,2}(\Psi)}} + \mathbf{U}\big(S,\|Z\|^2 \big) + \P\big(\vert T-S\vert  > u \big),
\end{align*}
where we used Lemma \ref{lem:anticoncentration:spok} in the last step. In the same way, one derives a corresponding lower bound. Hence, the first claim now follows from Lemma \ref{lem:Be:ulyanov+spok}. Using Lemma \ref{lem:ulyanov:optim} instead of Lemma \ref{lem:Be:ulyanov+spok}, the second claim follows.
\end{proof}

\subsection{Proofs for Section \ref{sec:concentration:inequalities}}\label{sec:proof:concentration:inequalities}

The following general result is from \cite{einmahl:tams:2008}.

\begin{lemma}[\cite{einmahl:tams:2008}]\label{lemma:einmahl}
Let $(B,\|\cdot\|)$ be a real separable Banach space with dual $B^*$ and let $B_1^*$ be the unit ball of $B^*$. Let $Z_1,\ldots,Z_n$ be independent $B$-valued random variables with mean zero such that for some $s > 2$, $\E \|Z_i\|^s < \infty$, $1 \leq i \leq n$.
Then we have for $0 < \nu \leq 1$, $\delta > 0$ and any $t > 0$,
\begin{align*}
\P\Big(\Big\|\sum_{i = 1}^n Z_i \Big\| \geq (1 + &\nu)\E \Big\|\sum_{i = 1}^n Z_i \Big\| + t \Big) \\&\leq C \sum_{i = 1}^n \E \big\| Z_i \big\|^s/t^s + \exp\Big(-\frac{t^2}{(2 + \delta) \varpi_n} \Big),
\end{align*}
where $\varpi_n = \sup\big\{ \sum_{i = 1}^n \E f^2(Z_i) \,\, : \,\, f \in B_1^* \big\}$ and $C$ is a positive constant depending on $\nu$, $\delta$ and $s$.
\end{lemma}

\begin{proof}[Proof of Lemma \ref{lem:fuknagaev:HS:nuclear:operator}]
\textit{(i)} Follows from applying Lemma \ref{lemma:einmahl} to the Hilbert space of all Hilbert-Schmidt operators equipped with the Hilbert-Schmidt norm (see also Lemma~1 in  \cite{jirak:wahl:pams:2020}).

\textit{(ii)} Let $B$ be the Banach space of all nuclear operators on $\mathcal{H}$ equipped with the nuclear norm. It is well-known that the dual space $B^*$ of $B$ is the Banach space of all bounded linear operators equipped with the operator norm $\|\cdot\|_\infty$. Moreover, for a bounded linear operator $S$, the corresponding functional is given by $T\mapsto \operatorname{tr}(TS)$. In order to apply Lemma~\ref{lemma:einmahl} with $Z_i=Y_i\otimes Y_i$ and $s=p>2$, it remains to bound all involved quantities. First, since the map $A \mapsto \operatorname{tr}(A^{1/2})$ is concave on the set of all positive self-adjoint trace-class operators, Jensen's inequality yields
\begin{align*}
\E \big\|\frac{1}{n}\sum_{i=1}^n\overline{Y_i\otimes Y_i}\big\|_1 &\leq \operatorname{tr}\Big(\E^{1/2}\Big(\frac{1}{n}\sum_{i=1}^n\overline{Y_i\otimes Y_i} \Big)^{2}\Big)= \frac{1}{\sqrt{n}} \operatorname{tr}\Big(\E^{1/2} (\overline{Y \otimes Y})^2\Big).
\end{align*}
Next, since $Y \otimes Y$ is rank-one, we have $\| Y \otimes Y \|_1=\| Y \|^2$. Hence, by the triangle inequality and Jensen's inequality, we have
\begin{align*}
\E \big\| \overline{Y \otimes Y} \big\|_1^{p} &\leq 2^{p-1} \E \big\|Y \otimes Y\big\|_1^{p} + 2^{p-1}\big\| \E Y \otimes Y\big\|_1^{p}\\&\leq 2^{p}  \E \big\|Y \otimes Y\big\|_1^{p} \leq 2^{p} \E \big\| Y \big\|^{2p}.
\end{align*}
Combining this with Lemma \ref{lem:moment:computation:KL}, we get
\begin{align}\label{eq:moment:bound:trace:norm}
\E \big\| \overline{Y \otimes Y} \big\|_1^p \leq 2^pC\|\vartheta\|_1^p.
\end{align}
Finally, by the above discussion, we have
\begin{align*}
    \varpi_n^2=  \sup_{\|S\|_{\infty} \leq  1} \E \operatorname{tr}^2(S  \overline{Y \otimes Y}).
\end{align*}
This completes the proof.

\textit{(iii)} It is also possible to apply Lemma \ref{lemma:einmahl} to the operator norm. Yet, since the involved expectation term is difficult to bound in this case, we proceed differently and apply techniques from~\cite{minsker17:aos} instead. Note that under Assumptions~\ref{ass_moments} and~\ref{ass:indep:4:moments} (which jointly imply $L^4-L^2$ norm equivalence) one can alternatively apply techniques from \cite{Z} and \cite{ZA}, leading to slightly better statements in terms of logarithmic terms.

Let $\psi: \R \to \R$ be the truncation function defined by
\begin{align*}
 \psi(x) = \operatorname{sign}(x) \big( |x| \wedge 1 \big).
\end{align*}
The following lemma follows from Theorem 3.2 in~\cite{minsker17:aos} and a standard approximation result.

\begin{lemma}\label{lem:minsker:1}
Let $(Z_i)$ with $Z_i \stackrel{d}{=} Z$ be a sequence of i.i.d.~compact self-adjoint random operators on $\mathcal{H}$, let $\tau_n^2 \geq n \|\E Z^2 \|_{\infty}$ and let $\theta>0$ be a truncation level. Then
\begin{align*}
\P\Big(\Big\|\sum_{i = 1}^n \Big(\frac{1}{\theta} \psi(\theta Z_i) - \E Z_i\Big) \Big\|_{\infty} \geq t \sqrt{n} \Big) \leq 2 \overline{d} \Big(1 + \frac{1}{\theta t \sqrt{n}} \Big) \exp\Big(-\theta t \sqrt{n} + \frac{\theta^2 \tau_n^2}{2}\Big),
\end{align*}
for any $t>0$, where $\overline{d} = \operatorname{tr}(\E Z^2 )/\|\E Z^2 \|_{\infty}$. Here, $\psi(\theta Z_i)$ denotes the usual functional transformation, i.e.~the function $\psi$ acts only on the spectrum of $\theta Z_i$.
\end{lemma}

We now turn to the proof of Lemma \ref{lem:fuknagaev:HS:nuclear:operator}(iii), combining Lemma \ref{lem:minsker:1} with a truncation argument. To this end, set $\theta  = t/\sqrt{n}$, $Z=\overline{Y \otimes Y}$ and $\tau_n^2 = n \|\E (\overline{Y \otimes Y})^2 \|_{\infty} = n$, recalling that $\|\E (\overline{Y \otimes Y})^2\|_\infty=1$. First,
\begin{align*}
\P\Big(\frac{1}{\theta} \psi\big(\theta \overline{Y_i \otimes Y_i}\big) &\neq\overline{Y_i \otimes Y_i} \text{ for at least one $1 \leq i \leq n$} \Big) \\& \leq  \P\Big(\max_{1 \leq i \leq n}\|\overline{Y_i \otimes Y_i} \|_{\infty} \geq \sqrt{n}/t \Big)\\&\leq \sum_{i = 1}^n n^{-p/2}t^p \E \|\overline{Y_i \otimes Y_i}\|_{\infty}^p = n^{1-p/2} t^p \E \| \overline{Y \otimes Y} \|_{\infty}^p.
\end{align*}
Combining this with \eqref{eq:moment:bound:trace:norm}, it follows that
\begin{align*}
&\P\Big(\Big\|\sum_{i = 1}^n \overline{Y_i \otimes Y_i} \Big\|_{\infty} \geq t \sqrt{n} \Big) \leq 2^pCn^{1-q/2} t^q \| \vartheta\|_1^p + \P\Big(\Big\|\sum_{i = 1}^n \frac{1}{\theta} \psi\big(\theta \overline{Y_i \otimes Y_i} \big) \Big\|_{\infty} \geq t \sqrt{n} \Big).
\end{align*}
An application of Lemma \ref{lem:minsker:1} yields
\begin{align*}
\P\Big(\Big\|\sum_{i = 1}^n \frac{1}{\theta} \psi\big(\theta \overline{Y_i \otimes Y_i} \big) \Big\|_{\infty} \geq t \sqrt{n} \Big) \leq 4 \operatorname{tr}\big(\E (\overline{Y\otimes Y})^2 \big) \exp\Big(-\frac{t^2}{2} \Big),
\end{align*}
The claim now follows form Lemma \ref{lem:moment:computation:KL}(ii).
\end{proof}

\begin{proof}[Proof of Lemma \ref{lem:moment:computation:KL}]
For $1\leq r\leq p$, we have
\begin{align*}
    \|Y\|^{2r}=\Big(\sum_{j\geq 1}\vartheta_j\zeta^2_j\Big)^r.
\end{align*}
Hence, by the triangle inequality, the Hölder inequality and the moment assumption, we have
\begin{align*}
    (\E \|Y\|^{2r})^{1/r}\leq \sum_{j\geq 1}\vartheta_j(\E \zeta_j^{2r})^{1/r}\leq \sum_{j\geq 1}\vartheta_j(\E \zeta_j^{2p})^{1/p}\leq C^{1/p}\sum_{j\geq 1}\vartheta_j=C^{1/p}\|\vartheta\|_1.
\end{align*}
This gives claim (i). In order to prove (ii), let us write
\begin{align*}
    (Y\otimes Y-\E Y\otimes Y)^2&=\Big(\sum_{j,k\geq 1}(\vartheta_j\vartheta_k)^{1/2}\overline{\zeta_j \zeta_k}(u_j \otimes u_k)  \Big)^2\\
    &=\sum_{j,k,s \geq 1} (\vartheta_{j} \vartheta_k^2\vartheta_{s})^{1/2}  \overline{\zeta_j \zeta_k}\, \overline{\zeta_k \zeta_s} (u_j \otimes u_s),
\end{align*}
where the second equality follows from the fact that $(u_j \otimes u_k)  (u_r \otimes u_s)$ is equal to $u_j \otimes u_s$ if $k = r$ and equal to $0$ otherwise. Hence,
\begin{align*}
    \operatorname{tr}(Y\otimes Y-\E Y\otimes Y)^2=\sum_{j,k \geq 1} \vartheta_{j}\vartheta_{k}  (\overline{\zeta_j \zeta_k})^2
\end{align*}
and thus
\begin{align*}
    \operatorname{tr}(\mathbb{E}(Y\otimes Y-\E Y\otimes Y)^2)= \mathbb{E}\operatorname{tr}((Y\otimes Y-\E Y\otimes Y)^2)\leq C^{2/p}\sum_{j,k \geq 1} \vartheta_{j}\vartheta_{k} ,
\end{align*}
as can be seen from inserting
\begin{align*}
    \E  (\overline{\zeta_j \zeta_k})^2\leq \E  (\zeta_j \zeta_k)^2\leq (\E \zeta_j^4)^{1/2}(\E \zeta_k^4)^{1/2}\leq C^{2/p}.
\end{align*}
To see the improvements of (iii)--(iv), let us note that under the additional assumptions,  we have $\E\overline{\zeta_j \zeta_k}\, \overline{\zeta_k \zeta_s}=0$ for $j\neq s$. Indeed, if $j\neq s$, then either $j\neq k$ and $k\neq s$ in which case $\E\overline{\zeta_j \zeta_k}\, \overline{\zeta_k \zeta_s}=\E\zeta_j \zeta_k^2 \zeta_s=0$, or $j=k$ and $k\neq s$ in which case $\E\overline{\zeta_j \zeta_k}\, \overline{\zeta_k \zeta_s}=\E \zeta_k^3\zeta_s-(\E\zeta_k^2)(\E \zeta_k \zeta_s)=0$, or $j\neq k$ and $s=k$ in which case $\E\overline{\zeta_j \zeta_k}\, \overline{\zeta_k \zeta_s}=\E \zeta_k^3\zeta_j-(\E\zeta_k^2)(\E \zeta_k \zeta_j)=0$. Hence, under the additional assumptions, we have
\begin{align*}
    \E(Y\otimes Y-\E Y\otimes Y)^2=\sum_{j,k \geq 1} \vartheta_j\vartheta_k  \E(\overline{\zeta_j \zeta_k})^2\, (u_j \otimes u_j),
\end{align*}
leading to
\begin{align}\label{eq:upper:lower:sigma:2}
    \big\|\E(Y\otimes Y-\E Y\otimes Y)^2\big\|_{\infty} &=\max_{j\geq 1}\sum_{k\geq 1}\vartheta_j\vartheta_k\E(\overline{\zeta_j\zeta_k})^2\leq C^{2/p}\|\vartheta \|_{\infty}\|\vartheta\|_1
\end{align}
and
\begin{align*}
    \operatorname{tr}((\E (Y\otimes Y-\E Y\otimes Y)^2 )^{1/2})&=\sum_{j \geq 1}\Big(\sum_{k\geq 1}\vartheta_j\vartheta_k  \E(\overline{\zeta_j \zeta_k})^2\Big)^{1/2} \\
    &\leq  C^{1/p}\Big(\sum_{j\geq 1}\sqrt{\vartheta_j}\Big)\|\vartheta\|_1^{1/2}.
\end{align*}
This gives claims (iii) and (iv).
\end{proof}

\subsection{CLTs: Proofs for Section \ref{sec:quantitative:limit:thms}}\label{sec:proof:quantitative:limit:thms}

\begin{proof}[Proof of Theorem \ref{thm:clt:I}]
We want to apply Proposition \ref{prop:clt:II} with the choices
\begin{align*}
T &= \frac{n}{A_{\J}}\big\|\hat{P}_{\J} - P_{\J}\big\|_2^2 \qquad\text{and}\qquad  S = \frac{n}{A_{\J}}\| L_\J E \|_2^2. \end{align*}
For this, let us write
\begin{align*}
    S &= \frac{1}{n}\Big\|\sum_{i=1}^n Y_i\Big\|_2^2\quad\text{with}\quad Y_i=\frac{1}{\sqrt{A_{\J}}}L_\J (X_i\otimes X_i).
\end{align*}
The random variable $L_\J X\otimes X$ takes values in the separable Hilbert space of all (self-adjoint) Hilbert-Schmidt operators on $\mathcal{H}$ (endowed with trace-inner product) and has the decomposition
\begin{align}\label{eq:KL:expansion:linear:term}
    \FO_{\J} (X \otimes X)=\sum_{j\in \J}\sum_{k\notin \J}\frac{\sqrt{\lambda_j \lambda_k}}{\lambda_j - \lambda_k}\zeta_{jk}(u_j \otimes u_k+u_k\otimes u_j)
\end{align}
with $\zeta_{jk} = \eta_j \eta_k$, as can be seen from inserting the Karhunen-Loève expansion of $X$ and the definition of $L_\J$.
Using Assumption \ref{ass_moments}, the Cauchy-Schwarz inequality and the fact that the summation is over different indices, we have
\begin{align*}
\E\zeta_{jk}=0\quad\text{and}\quad \E \big\vert\zeta_{jk}\big\vert^p \leq C_{\eta}.
\end{align*}
Hence, by Lemma \ref{lem:moment:computation:KL}(i) and Lemma \ref{lem:Lambda:Sigma:relation}, we have for every $1\leq r\leq p$,
\begin{align}\label{eq:norm:linear:term}
\E\|L_\J (X\otimes X)\|_2^r \lesssim_r \Big(\sum_{j \in \J}\sum_{k \notin \J} \frac{\lambda_k \lambda_j}{(\lambda_k- \lambda_j)^2}\Big)^{r/2}\asymp A_\J^{r/2}.
\end{align}
In particular, setting $q=p\wedge 3$, we get
\begin{align*}
    \E\|Y_i\|_2^{r}\lesssim 1
\end{align*}
for all $r\in[2,q]$. By Corollary \ref{cor:exp:quadratic:term}, we have
\begin{align*}
    \vert T-S\vert = \frac{n}{A_{\J}}\big\vert\|\hat{P}_{\J} - P_{\J}\big\|_2^2-\| L_\J E \|_2^2\big\vert\leq  \frac{n}{A_{\J}}\big(\vert \J\vert ^{3/2}\delta_\J^3(E)+\vert \J\vert ^2\delta_\J^4(E)\big).
\end{align*}
Thus, by \eqref{ass:relativerank} and Lemma \ref{lem:bound:events:via:fn:minsker}, we get
\begin{align}\label{eq:conc:T-S}
\P\big(\vert T - S\vert  > u \big) \lesssim_p \qprob_{\J,n,p},
\end{align}
as long as
\begin{align}\label{eq:conc:T-S:cond}
    u =Cn^{-1/2}\log^{3/2} n \frac{\sigma_{\J}^3\vert \J\vert ^{3/2}}{A_{\J}}.
\end{align}
Finally, we have
\begin{align}
\frac{n}{A_{\J}} \big\|\hat{P}_{\J} - P_{\J} \big\|_2^2 \leq \frac{2\vert \J\vert n}{A_{\J}}  =: C_T.
\end{align}
Inserting these choices for $T$, $S$, $u$ and $C_T$ into Proposition \ref{prop:clt:II} (applied with $q=p\wedge 3$), the claim follows.
\end{proof}

\begin{proof}[Proof of Theorem \ref{thm:clt:III}]
We want to apply Proposition \ref{prop:clt:III} with the choices $A=A_{\J}$, $B=B_{\J}$, $C=C_{\J}$ (recall \eqref{eq:ABC}),
\begin{align*}
T &= n\big\|\hat{P}_{\J} - P_{\J}\big\|_2^2,\qquad  S = n\| L_\J E \|_2^2,
\end{align*}
and
\begin{align*}
    u=C n^{-1/2}(\log n)^{3/2} \sigma_{\J}^3\vert \J\vert ^{3/2}.
\end{align*}
For this, let us write
\begin{align*}
    S &= \frac{1}{n}\Big\|\sum_{i=1}^n Y_i\Big\|_2^2\quad\text{with}\quad Y_i =L_\J (X_i\otimes X_i).
\end{align*}
Inserting these choices for $A$, $B$, $C$, $T$, $S$, $u$ into Proposition \ref{prop:clt:III} (applied with $q=\min(3,p)$), the claim follows from inserting \eqref{eq:norm:linear:term} and \eqref{eq:conc:T-S}.
\end{proof}

\begin{proof}[Proof of Corollary \ref{cor:thm:clt:II}]
Using Lemma \ref{lem:BE:gauss}, we may argue as in the proof of Proposition~\ref{prop:clt:III}.
\end{proof}

\begin{proof}[Proof of Theorem \ref{thm:clt:II}]
We proceed exactly as in the proof of Theorem \ref{thm:clt:III}, using Proposition \ref{prop:clt:V} instead of Proposition \ref{prop:clt:III}.
\end{proof}

\subsection{Bootstrap I: Proof of Theorem \ref{thm:boot:I}}\label{sec:proof:thm:boot:I}
We start by applying Proposition \ref{prop:clt:III} with respect to $\tilde{\mathbf{U}}$, and
\begin{align}\label{eq:choices:setting1:bootstrap}
   \tilde T &= \frac{n}{\sigma_w^2}\big\|\tilde{P}_{\J} - \hat{P}_{\J}\big\|_2^2\quad\text{and}\quad \tilde S =\frac{n}{\sigma_w^2}\big\|\FO_{\J} (\tilde{E} - E)\|_2^2.
\end{align}
In doing so, we consider Setting \ref{ass:clts:setup} with
\begin{align*}
   \tilde S &= \frac{1}{n}\Big\|\sum_{i=1}^n \tilde Y_i\Big\|_2^2,\quad\text{where}\quad \tilde Y_i =\frac{w_i^2-1}{\sigma_w}Y_i,\quad Y_i=L_\J (X_i\otimes X_i).
\end{align*}
Note that we use an additional tilde to indicate that we are dealing with the bootstrap quantities. Hence, we have
\begin{align*}
    \tilde{\Psi}_\J =\frac{1}{n}\sum_{i=1}^n\tilde\E \tilde Y_i\otimes \tilde Y_i= \frac{1}{n}\sum_{i = 1}^n Y_i \otimes Y_i,
\end{align*}
as well as
\begin{align*}
    \tilde A_{\J}&=\|\tilde \Psi_\J \|_1=\tilde \E \tilde S,\\
    \tilde B_{\J}&= \sqrt{2}\|\tilde \Psi_\J \|_2,\\
    \tilde C_{\J}&= 2\|\tilde \Psi_\J \|_3.
\end{align*}
The following lemma shows that the quantities $\tilde{A}_{\J}$, $\tilde{B}_{\J}$, $\tilde{C}_{\J}$ are concentrated around their corresponding quantities ${A}_{\J}$, ${B}_{\J}$, ${C}_{\J}$ from \eqref{eq:ABC}.

\begin{lemma}\label{lem:emp:replace}
Suppose that \eqref{ass:boot:thirdnorm} holds. Then with probability at least $1 - 2n^{1- p/4}$, we have
\begin{align*}
&\big\vert\tilde{A}_{\J} - {A}_{\J} \big\vert \lesssim_p {A}_{\J} \frac{\sqrt{ \log n}}{\sqrt{n}},\\
&\big\vert\tilde{B}_{\J} - {B}_{\J} \big\vert \lesssim_p {A}_{\J} \frac{\sqrt{ \log n}}{\sqrt{n}},\\
&\big\vert\tilde{C}_{\J} - {C}_{\J} \big\vert \lesssim_p {A}_{\J} \frac{\sqrt{ \log n}}{\sqrt{n}}.
\end{align*}
In particular, for $n$ large enough, we have with the same probability,
\begin{align*}
&{\bf (i)} \quad A_{\J} \leq 2\tilde{A}_{\J} \leq 4 A_{\J},\\
&{\bf (ii)} \quad B_{\J} \leq 2\tilde{B}_{\J} \leq 4 B_{\J},\\
&{\bf (iii)} \quad C_{\J} \leq 2\tilde{C}_{\J} \leq 4 C_{\J}.
\end{align*}
\end{lemma}
\begin{proof}
First, \textbf(i)--\textbf(iii) are consequences of \eqref{ass:boot:thirdnorm} and the first three claims (note that $C_{\J}\lesssim B_{\J}\lesssim A_{\J}$ by properties of the Schatten norms). Let us start proving the first claim. By independence of $(w_i)$ and $(Y_i)$, we have
\begin{align*}
\tilde{A}_{\J} &=\|\tilde\Psi\|_1=\tilde\E \tilde S=\frac{1}{n}\sum_{i=1}^n\tilde\E\|\tilde Y_i\|^2=\frac{1}{n}\sum_{i=1}^n\| Y_i\|^2.
\end{align*}
By \eqref{eq:norm:linear:term}, we have $(\E (\| Y \|_2^2)^{p/2})^{2/p}\lesssim_p A_\J$. Since $\E \tilde A_{\J} =  A_{\J}$, \eqref{lem:fn} yields that for some constant $C > 0$,
\begin{align*}
\P\Big(\big\vert\tilde{A}_{\J} - {A}_{\J}\big\vert \geq C A_{\J}\frac{\sqrt{\log n}}{\sqrt{n}} \Big) \leq  n^{1- p/4}.
\end{align*}
Using that $\tilde{B}_{\J}=\sqrt{2}\|\tilde \Psi_\J \|_2$ and ${B}_{\J}=\sqrt{2}\| \Psi_\J \|_2$, the triangle inequality and $\vert \| \tilde \Psi_\J \|_2-\|\Psi_\J \|_2 \vert \leq \|\tilde \Psi_\J -\Psi_\J \|_2$, the second claim follows if we can show that, with probability at least $1 - n^{1- p/4}$,
\begin{align}\label{eq:emp:matrix:L2}
\|\tilde{\Psi}_\J  - \Psi_\J \|_2 = \Big\|\frac{1}{n}\sum_{i=1}^nY_i\otimes Y_i-\Psi \Big\|_2\lesssim_p A_{\J} \frac{\sqrt{\log n}}{\sqrt{n}}.
\end{align}
In order to get \eqref{eq:emp:matrix:L2}, we apply Lemma \ref{lem:fuknagaev:HS:nuclear:operator}(i). For this, recall from \eqref{eq:KL:expansion:linear:term} that we have
\begin{align*}
    Y_i\stackrel{d}{=}Y=\FO_{\J} X \otimes X=\sum_{j\in \J}\sum_{k\notin \J}\frac{\sqrt{\lambda_j \lambda_k}}{\lambda_j - \lambda_k}\zeta_{jk}(u_j \otimes u_k+u_k\otimes u_j),
\end{align*}
with $\zeta_{jk} = \eta_j \eta_k$, where the $\zeta_{jk}$ satisfy $\E\zeta_{jk}=0$ and $\E \big\vert\zeta_{jk}\big\vert^p \leq C_{\eta}$ for all $j\in\J$ and $k\notin \J$. Hence, Lemma \ref{lem:fuknagaev:HS:nuclear:operator}(i) and Lemma \ref{lem:Lambda:Sigma:relation} yield that for some $C > 0$,
\begin{align*}
\P\Big(\big\|\tilde{\Psi}_\J - \Psi_\J \big\|_2 \geq C \frac{t}{\sqrt{n}} A_{\J} \Big) \leq \frac{n^{1 - p/4}}{t^{p/2}} + e^{-t^2},
\end{align*}
where $t\geq 1$, and \eqref{eq:emp:matrix:L2} follows by setting $t = C\sqrt{\log n}$. Finally, using that $\tilde{C}_{\J}=2\|\tilde \Psi_\J \|_3$ and  ${C}_{\J}=2\| \Psi_\J \|_3$ and the inequality $\vert \| \tilde\Psi_\J \|_3-\| \Psi_\J \|_3 \vert \leq \|\tilde\Psi_\J - \Psi_\J \|_3\leq \|\tilde\Psi_\J - \Psi_\J \|_2$, the third claim follows from \eqref{eq:emp:matrix:L2}. Applying the union bound completes the proof.
\end{proof}

\begin{corollary}\label{lem:boot:conditional:BE:I}
Let $q \in (2,3]$ with $q\leq p$ and $s \in (0,1)$. Then, with probability at least $1 - C_p \qprob^{1-s}_{\J,n,p}-C_pn^{1-p/(2q)}$, $C_p > 0$, we have
\begin{align*}
\tilde{\mathbf{U}}_{}\Big(\frac{\tilde T - \tilde{A}_{\J}}{\tilde{B}_{\J}}
, G\Big) &\lesssim_p n^{1/4-q/8}\frac{{A}_{\J}^{3/4}}{B_{\J}^{3/4}}  +  \frac{{C}_{\J}^3}{{B}_{\J}^3} + n^{-1/2}(\log n)^{3/2}\frac{\sigma_\J^3\vert \J\vert^{3/2}}{{B}_{\J}}+ \qprob^{s}_{\J,n,p}.
\end{align*}
\end{corollary}

\begin{proof}
Applying Proposition \ref{prop:clt:III} to the choices \eqref{eq:choices:setting1:bootstrap}, we get
\begin{align*}
\tilde{\mathbf{U}}_{}\Big(\frac{\tilde T - \tilde{A}_{\J}}{ \tilde{B}_{\J} }
, G\Big) &\lesssim \frac{1}{n^{q/8}}\Big( \frac{\tilde{A}_{\J}^{3} }{\tilde{B}_{\J}^{3}}\sum_{i = 1}^n \frac{\|Y_i\|_2^q}{\tilde{A}_{\J}^{q/2}}\Big)^{1/4}+ \frac{\tilde{C}_{\J}^3}{\tilde{B}_{\J}^3} + \tilde{\P}_{}\big(\vert \tilde S-\tilde T\vert > u\big) + \frac{u}{\tilde{B}_{\J}}.
\end{align*}
By \eqref{eq:norm:linear:term}, we have $(\E (\| Y \|_2^q)^{p/q})^{q/p}\lesssim_q A_\J^{q/2}$. Hence, \eqref{lem:fn} yields that for some constant $C > 0$,
\begin{align*}
\P\Big(\Big\vert\sum_{i = 1}^n \big(\|Y_i\|_2^q - \E \|Y_i\|_2^q\big) \Big\vert \geq C A_{\J}^{q/2}\sqrt{n\log n}  \Big) \leq n^{1-p/(2q)}.
\end{align*}
Hence, with probability at least $1 - n^{1-p/(2q)}$,
\begin{align}\label{eq:moment:q:bound}
\sum_{i = 1}^n \|Y_i\|_2^q &\lesssim_q  n\E \|Y\|_2^q + \sqrt{n\log n} A_{\J}^{q/2}\lesssim_q  n A_{\J}^{q/2}.
\end{align}
By Corollary \ref{lem:boot:approx}, we have
\begin{align*}
    \vert \tilde T-\tilde S\vert &= \frac{n}{\sigma_w^2}\big\vert\|\tilde{P}_{\J} - \hat{P}_{\J}\big\|_2^2-\|\FO_{\J} (\tilde{E} - E)\|_2^2\big\vert\\
    &\leq \frac{n}{\sigma_w^2}\big(\vert \J\vert ^{3/2}(\delta_\J^3(E)+\delta_\J^3(\tilde E))+\vert \J\vert ^2(\delta_\J^4(E)+\delta_\J^4(\tilde E))\big).
\end{align*}
Similarly as in Lemma \ref{lem:bound:events:via:fn:minsker}, we have
\begin{align}\label{eq:delta:J:tilde}
    \P \Big(\delta_\J(\tilde E)>C\sqrt{\frac{\sigma_\J^2\log n}{n}}\Big)\lesssim_p \qprob_{\J,n,p}.
\end{align}
In fact, since $w$ is independent of $X$ and satisfies $\E w^2=1$, the random variabe $w(\vert R_{\J^c}\vert^{1/2}+g_\J^{-1/2}P_\J)X$ has the same covariance operator as $X'=(\vert R_{\J^c}\vert+g_\J^{-1/2}P_\J)X$ and its Karhunen-Loève coefficients are given by $w\eta_j$, where $\eta_j$ are the Karhunen-Loève coefficients of $X'$ and $X$. Hence, \eqref{eq:delta:J:tilde} follows by the same line of arguments as Lemma same line of arguments as Lemma \ref{lem:bound:events:via:fn:minsker}. Combining \eqref{eq:delta:J:tilde} with \eqref{ass:relativerank}, we thus obtain
\begin{align}\label{eq:P(vT-Sv:bigger:u)}
\P\big(\vert \tilde T - \tilde S\vert  > u \big) \lesssim_p \qprob_{\J,n,p}\quad\text{with}\quad u=Cn^{-1/2}(\log n)^{3/2} \sigma_{\J}^3\vert \J\vert ^{3/2}
\end{align}
for $C > 0$ sufficiently large. By Markov's inequality, we have
\begin{align}\nonumber
\P\big( \tilde{\P}_{}(\vert \tilde S-\tilde T\vert > u) \geq \qprob^{s}_{\J,n,p} \big) &\leq  \qprob^{-s}_{\J,n,p} \P\big(\vert \tilde S-\tilde T\vert > u\big)\lesssim_p \qprob^{1-s}_{\J,n,p}.
\end{align}
Due to Lemma \ref{lem:emp:replace}, we can replace $\tilde{A}_{\J}, \tilde{B}_{\J}, \tilde{C}_{\J}$ by its counterparts ${A}_{\J}, {B}_{\J}, {C}_{\J}$ with probability at least $1 - C n^{1-p/4}$, $C > 0$. Piecing everything together, the claim follows.
\end{proof}

\begin{proof}[Proof of Theorem \ref{thm:boot:I}]
To simplify the notation, we assume w.l.o.g. $\sigma_w^2 = 1$. Then by affine invariance and the triangle inequality
\begin{align*}
&\tilde{\mathbf{U}}_{}\Big(\big\|\tilde{P}_{\J} - \hat{P}_{\J} \big\|_2^2, \big\|\hat{P}_{\J}' - {P}_{\J} \big\|_2^2 \Big) \\&=
\tilde{\mathbf{U}}_{}\Big(\frac{n\big\|\tilde{P}_{\J} - \hat{P}_{\J} \big\|_2^2 - \tilde{A}_{\J}}{ \tilde{B}_{\J} }
, \frac{n\big\|\hat{P}_{\J}' - {P}_{\J} \big\|_2^2 - \tilde{A}_{\J}}{ \tilde{B}_{\J} }\Big)\\& \leq
\tilde{\mathbf{U}}_{}\Big(\frac{n\big\|\tilde{P}_{\J} - \hat{P}_{\J} \big\|_2^2 - \tilde{A}_{\J}}{ \tilde{B}_{\J} }
, G\Big) + \tilde{\mathbf{U}}_{}\Big( G
, \frac{n\big\|\hat{P}_{\J}' - {P}_{\J} \big\|_2^2 - \tilde{A}_{\J}}{ \tilde{B}_{\J} }\Big).
\end{align*}
The first term on the right-hand side is treated in Corollary \ref{lem:boot:conditional:BE:I}. In order to deal with the second term, we can directly apply Theorem \ref{thm:clt:III} as follows. By affine invariance, the triangle inequality and independence
\begin{align*}
\tilde{\mathbf{U}}_{}\Big( G
, \frac{n\big\|\hat{P}_{\J}' - {P}_{\J} \big\|_2^2 - \tilde{A}_{\J}}{ \tilde{B}_{\J} }\Big) &\leq \mathbf{U}_{}\Big( G
, \frac{n\big\|\hat{P}_{\J}' - {P}_{\J} \big\|_2^2 - {A}_{\J}}{{B}_{\J} }\Big) \\&+ \tilde{\mathbf{U}}_{}\Big( G
, \frac{ \tilde{B}_{\J} G + \tilde{A}_{\J} -  {A}_{\J}}{{B}_{\J} }\Big).
\end{align*}
For the first term on the right-hand-side, we can apply Theorem \ref{thm:clt:III}. Note that the corresponding bound is dominated by the one provided by Corollary~\ref{lem:boot:conditional:BE:I}. For the second term, using that the standard Gaussian distribution function $\Phi$ is Lipschitz continuous and satisfies for $a\geq 1$ and $x>0$ that $\Phi(ax)-\Phi(x)\leq (a-1)x\Phi'(x)\lesssim a-1$, we get
\begin{align*}
\tilde{\mathbf{U}}_{}\Big( G
, \frac{ \tilde{B}_{\J} G + \tilde{A}_{\J} -  {A}_{\J}}{{B}_{\J} }\Big) \lesssim \Big\vert\frac{\tilde{A}_{\J} -  {A}_{\J}}{\tilde{B}_{\J}}\Big\vert + \Big\vert\frac{{B}_{\J}}{\tilde{B}_{\J}} - 1\Big\vert+ \Big\vert\frac{\tilde{B}_{\J}}{B_{\J}} - 1\Big\vert.
\end{align*}
Due to Lemma \ref{lem:emp:replace}, we have
\begin{align*}
 \Big\vert\frac{\tilde{A}_{\J} -  {A}_{\J}}{\tilde{B}_{\J}}\Big\vert + \Big\vert\frac{{B}_{\J}}{\tilde{B}_{\J}} - 1\Big\vert+ \Big\vert\frac{\tilde{B}_{\J}}{B_{\J}} - 1\Big\vert \lesssim_p  \frac{\sqrt{\log n}}{\sqrt{n}} \frac{A_{\J}}{{B}_{\J}},
\end{align*}
with probability at least $1 - 2n^{1-p/4}$. Combining all bounds together with the union bound for the involved probabilities, the claim follows.
\end{proof}

\begin{proof}[Proof of Corollary \ref{cor:boot:I:quant}]
Denote with $\hat{\mathcal{E}}$ the event where Theorem \ref{thm:boot:I} applies, hence
\begin{align}\label{eq:thm:boot:I:quant:1}
\P\big(\hat{\mathcal{E}} \big) \geq 1 - C  \qprob^{1-s}_{\J,n,p} \vee  n^{1-p/(2q)}.
\end{align}
We now slightly reverse the argument used to prove Theorem \ref{thm:boot:I}, where we again assume $\sigma_w^2 = 1$ for simplicity. The triangle inequality and the affine invariance gives
\begin{align*}
&\tilde{\mathbf{U}}_{}\Big(n\big\|\tilde{P}_{\J} - \hat{P}_{\J} \big\|_2^2, B_{\J} G + A_{\J} \Big) \\&\leq \tilde{\mathbf{U}}_{}\Big(\big\|\tilde{P}_{\J} - \hat{P}_{\J} \big\|_2^2, \big\|\hat{P}_{\J}' - {P}_{\J} \big\|_2^2 \Big) \\&+ \mathbf{U}_{}\Big(B_{\J} G + A_{\J}, n\big\|\hat{P}_{\J}' - {P}_{\J} \big\|_2^2 \Big),
\end{align*}
where we recall that ${X_i}'$ denotes independent copies of $X_i$. By Theorems \ref{thm:clt:III} and \ref{thm:boot:I} and again the affine invariance, we get that on the event $\hat{\mathcal{E}}$
\begin{align*}
\tilde{\mathbf{U}}_{}\Big(n\sigma_w^{-1}\big\|\tilde{P}_{\J} - \hat{P}_{\J} \big\|_2^2, \sqrt{2} B_{\J} G + A_{\J} \Big)  \lesssim \mathbf{A}_{\J,n,p,s}.
\end{align*}
Let $q_{\alpha}^G$ be the quantile of the distribution function of $ B_{\J} G + A_{\J}$. The above and Lemma \ref{lem_gen_quant} (see below) yields that on the event $\hat{\mathcal{E}}$, there exists a constant $C_1 > 0$ such that for any $\alpha \in [0,1]$ and $\delta > C_1  \mathbf{A}_{\J,n,p,s}$
\begin{align}
\hat{q}_{\alpha+ \delta} \leq q_{\alpha}^G \leq \hat{q}_{\alpha - \delta}.
\end{align}
Consequently, for $\delta > C_1  \mathbf{A}_{\J,n,p,s}$, it follows that for some $C_2 > 0$
\begin{align*}
\P\Big(n\big\|\hat{P}_{\J} - {P}_{\J} \big\|_2^2 \leq \hat{q}_{\alpha} \Big) &\leq \P\Big(n\big\|\hat{P}_{\J} - {P}_{\J} \big\|_2^2 \leq \hat{q}_{\alpha}, \hat{\mathcal{E}} \Big) + \P\big(\hat{\mathcal{E}}^c\big) \\&\leq \P\Big(n\big\|\hat{P}_{\J} - {P}_{\J} \big\|_2^2 \leq q_{\alpha-\delta}^G, \hat{\mathcal{E}} \Big) + \P\big(\hat{\mathcal{E}}^c\big) \\&\leq \P\Big(n\big\|\hat{P}_{\J} - {P}_{\J} \big\|_2^2 \leq q_{\alpha-\delta}^G \Big) + 2\P\big(\hat{\mathcal{E}}^c\big) \\&\leq \P\Big( B_{\J} G + A_{\J} \leq q_{\alpha-\delta}^G \Big) + C_2 \mathbf{A}\big(\J,n,p\big) + 2\P\big(\hat{\mathcal{E}}^c\big)\\&\leq 1 - \alpha + (C_1 + C_2) \mathbf{A}_{\J,n,p,s} + 2\P\big(\hat{\mathcal{E}}^c\big),
\end{align*}
where we also used Theorem \ref{thm:clt:III}. 
In the same manner, one obtains the corresponding lower bound, and the claim follows together with \eqref{eq:thm:boot:I:quant:1}.
\end{proof}

\begin{lemma}\label{lem_gen_quant}
Let $F,G$ be distribution functions such that
\begin{align*}
&\sup_{x \in \R}\bigl\vert F(x) - G(x)\bigr\vert  \leq \delta.
\end{align*}
If $G$ is continuous, then for any $\delta' > 3\delta \geq 0$
\begin{align*}
q_{\alpha + \delta'}(F) \leq q_{\alpha}(G) \leq q_{\alpha - \delta'}(F).
\end{align*}
Here, $q_{\alpha}(H) = \inf\{x: H(x) \geq 1 - \alpha\}$, $H \in \{F,G\}$.
\end{lemma}
%

\begin{proof}[Proof of Lemma \ref{lem_gen_quant}]
Suppose first that $q_{\alpha}(G) > q_{\alpha - \delta'}(F)$, $\delta' > \delta$. Then
\begin{align*}
1 - \alpha = G\big(q_{\alpha}(G)\big) \geq F\big(q_{\alpha}(G)\big) - \delta \geq  F\big(q_{\alpha - \delta'}(F)\big) - \delta
\geq 1 - (\alpha - \delta') - \delta,
\end{align*}
which yields a contradiction. Observe next that due to the continuity of $G$, we have
\begin{align*}
F(x) - F(x_-) \leq 2 \delta.
\end{align*}
Hence, if $q_{\alpha}(G) < q_{\alpha + \delta'}(F)$, $\delta' > 3\delta$, then
\begin{align*}
1 - \alpha = G\big(q_{\alpha}(G)\big) \leq F\big(q_{\alpha}(G)\big) + \delta \leq  F\big(q_{\alpha + \delta'}(F)\big) + \delta \leq 1 - \alpha - \delta' + 3\delta,
\end{align*}
which again yields a contradiction. Hence the claim follows.
\end{proof}

\subsection{Bootstrap II: Proof of Theorem \ref{thm:boot:II}}\label{sec:proof:bootstrap:II}

\subsubsection{Additional notation}\label{sec:proof:bootstrap:II:notation}
Let us first introduce the additional key quantities, also used in the formulation of Theorem \ref{thm:boot:II}. Recall that $\J=\{j_1,\dots,j_2\}$ and $\I=\{1,\dots,i_2\}$ with $i_2>j_2+2$. We use the same notation as in the proof of Theorem \ref{thm:boot:I} in Section \ref{sec:proof:thm:boot:I}. Let
\begin{align*}
   \tilde T &= \frac{n}{\sigma_w^2}\big\|\tilde{P}_{\J} - \hat{P}_{\J}\big\|_2^2\quad\text{and}\quad  \tilde S_\I=\frac{n}{\sigma_w^2}\big\|\FO_{\J,\I} (\tilde{E} - E)\|_2^2,
\end{align*}
where
\begin{align*}
L_{\J,\I}A = \sum_{j \in \J}\sum_{k \in \J^c \cap \I} \frac{1}{\lambda_j - \lambda_k}(P_k AP_j+P_j A P_k)
\end{align*}
for a Hilbert-Schmidt operator $A$ on $\mathcal{H}$. For this, let us write
\begin{align} \label{defn:bootstrapII:proof:T:Y:S:Sigma}
    \tilde{S}_{\I} &= \frac{1}{n}\Big\|\sum_{i=1}^n \tilde Y_{i\I}\Big\|_2^2\quad\text{with}\quad \tilde Y_{i\I} =\frac{w_i^2-1}{\sigma_w}Y_{i\I},\quad Y_{i\I}=L_{\J,\I} (X_i\otimes X_i),
\end{align}
and let
\begin{align*}
    \Psi_{\J,\I}=\E Y_{\I}\otimes Y_{\I}\quad\text{and}\quad \tilde\Psi_{\J,\I}=\frac{1}{n}\sum_{i=1}^n\tilde \E\tilde Y_{i\I}\otimes \tilde Y_{i\I}=\frac{1}{n}\sum_{i=1}^nY_{i\I}\otimes Y_{i\I}
    \end{align*}
with $Y_{\I}=\FO_{\J, \I} X \otimes X$. Moreover, let $Z_{\J,\I}$ be Gaussian with ${\E} Z_{\J,\I} = 0$ and covariance operator ${\Psi}_{\J,\I}$,  independent of $\mathcal{X}$, and likewise $\tilde Z_{\J,\I}$ be Gaussian with $\tilde{\E} \tilde Z_{\J,\I} = 0$ and covariance operator $\tilde{\Psi}_{\J,\I}$. Finally, we set
\begin{align*}
    A_{\J,\I}&=\| \Psi_{\J,\I}\|_1,\qquad\ \ \ \ \tilde A_{\J,\I}=\|\tilde \Psi_{\J,\I}\|_1,\\
    B_{\J,\I}&=\sqrt{2}\|\Psi_{\J,\I}\|_2,\qquad\tilde B_{\J,\I}= \sqrt{2}\|\tilde \Psi_{\J,\I}\|_2,\\
    C_{\J,\I}&=2\|\Psi_{\J,\I}\|_3,\qquad\ \ \ \tilde C_{\J,\I}= 2\|\tilde \Psi_{\J,\I}\|_3.
\end{align*}
A more explicit expression of $A_{\J,\I^c}$ (resp. $A_{\J,\I}$) is given in \eqref{eq:compute:A_{JI^c}} below.

\subsubsection{Proofs}

We first apply Proposition \ref{prop:clt:V} to $\tilde{\mathbf{U}}$ and establish the following result.

\begin{corollary}\label{cor:boot:II:conditional}
Let $s \in (0,1)$. Then, with probability at least $1 - C_p(\qprob^{}_{\J,n,p} + n^{1-p/6})^{1-s}$, $C_p > 0$ we have
\begin{align*}
\tilde{\mathbf{U}}\Big(\tilde T, \|\tilde Z_{\J,\I}\|_2^2\Big) &\lesssim_p
n^{-1/5}\Big(\frac{A_\J}{\sqrt{{\lambda}_{1,2}(\Psi_{\J})}} \Big)^{3/5}+n^{-1/2}\log^{3/2} n\frac{\sigma_{\J}^3\vert \J\vert ^{3/2}}{\sqrt{\lambda_{1,2}(\Psi_{\J})}}
 \\& +\frac{A_{\J,\I^c}\log n }{\sqrt{{\lambda}_{1,2}(\Psi_{\J})}}+ \big(\qprob^{}_{\J,n,p} + n^{1- p/6}\big)^{s}.
\end{align*}
\end{corollary}

The following lemma is needed in the proof of Corollary \ref{cor:boot:II:conditional} and extends some basic inequalities from Section \ref{sec:proof:thm:boot:I}  to the truncated setting.

\begin{lemma}\label{lem:fn:norm:bound:II}
We have, with probability at least $1 - 3n^{1-p/(2q)}$,
\begin{description}
\item[(i)] $\|\Psi_\J - \Psi_{\J,\I} \|_{1}\lesssim_p A_{\J,\I^c}$.
\item[(ii)] $\vert \tilde S- \tilde S_\I\vert  \lesssim_p A_{\J,\I^c}\log n $.
\item[(iii)] $\vert \sum_{i = 1}^n (\| \tilde{Y}_{i\I}\|_2^q - \E \|\tilde{Y}_{i\I}\|_2^q ) \vert  \lesssim_p A_{\J,\I}^{q/2}\sqrt{n \log n}\lesssim_p A_{\J}^{q/2}\sqrt{n \log n}$.
\item[(iv)] $\|\tilde{\Psi}_{\J,\I} - \Psi_{\J,\I}\|_2 \lesssim_p A_{\J,\I} \sqrt{n^{-1/2} \log n}\lesssim_p A_{\J} \sqrt{n^{-1/2} \log n}$.
\end{description}
\end{lemma}

\begin{proof}[Proof of Lemma \ref{lem:fn:norm:bound:II}]
Inserting the  Karhunen-Loève expansion of $X$ into the definition of $Y_\I$, we have
\begin{align}\label{eq:KL:expansion:linear:term:truncated}
    Y_{\I}=\sum_{j\in \J}\sum_{k\in \J^c\cap \I}\frac{\sqrt{2\lambda_k \lambda_j}}{\lambda_j - \lambda_k}\zeta_{jk}(u_j \otimes u_k+u_k\otimes u_j)/\sqrt{2}
\end{align}
with $\zeta_{jk} = \eta_j \eta_k$. Moreover, by Assumptions \ref{ass_moments}, \ref{ass:lower:bound} and \ref{ass:indep:4:moments} with $m=4$, the Cauchy-Schwarz inequality and the fact that summation is over different indices, we have that the $\zeta_{jk}$ are centered, uncorrelated and satisfy $\E \vert \zeta_{kj}\vert ^p \leq C_{\eta}$ and $\alpha_{jk}= \E \zeta_{jk}^2\geq c_\eta$ for all $j\in \J, k\in\J^c\cap\I$. Hence, scaling the $\zeta_{jk}$ appropriately, \eqref{eq:KL:expansion:linear:term:truncated} yields the Karhunen-Loève expansion of $Y_{\I}$ and the eigenpairs of $\Psi_{\J,\I}$ are given by
\begin{align*}
    \alpha_{jk}\frac{2\lambda_k \lambda_j}{(\lambda_k - \lambda_j)^2}\quad\text{and}\quad \frac{1}{\sqrt{2}}(u_j \otimes u_k+u_k\otimes u_j),\qquad j\in \J, k\in \J^c\cap \I.
\end{align*}
In order to obtain the first claim, observe that $\Psi_\J$ has the same eigenvalues and eigenvectors, but with indices $k\in \J^c$ instead of $k\in \J^c\cap \I$. Hence
\begin{align}\label{eq:compute:A_{JI^c}}
\|\Psi_\J - \Psi_{\J,\I} \|_{1}=\sum_{j\in \J}\sum_{k\in \I^c}\alpha_{jk}\frac{2\lambda_k \lambda_j}{(\lambda_k - \lambda_j)^2}= A_{\J,\I^c}.
\end{align}
By properties of the Hilbert-Schmidt norm, we have due to the orthogonality of $P_{\I}$ and $P_{\I^c}$
\begin{align}\nonumber
\tilde S -\tilde S_\I=\tilde S - \frac{n}{\sigma_w^2}\big\|\FO_{\J,\I} (\tilde{E} - E) \big\|_2^2&=  \frac{n}{\sigma_w^2}\big\|\FO_{\J,\I^c} (\tilde{E} - E) \big\|_2^2.
\end{align}
The latter can be written as
\begin{align*}
    \frac{1}{n}\Big\|\sum_{i=1}^n\tilde Y_{i\I^c}\Big\|_2^2\quad\text{with}\quad \tilde Y_{i\I^c} =\frac{w_i^2-1}{\sigma_w}Y_{i\I^c},\quad Y_{i\I^c}=L_{\J,\I^c} X_i\otimes X_i.
\end{align*}
Now, using that $\E\|\tilde Y_{i\I^c}\|_2^p\lesssim_p A_{\J,\I^c}^{p/2}$ by Lemma \ref{lem:moment:computation:KL}, claim (ii) can be deduced from Lemma \ref{lemma:einmahl}, setting $t=CA_{\J,\I^c}^{1/2}\sqrt{\log n}$. Claims (iii) and (iv) follow by the same arguments as in \eqref{eq:emp:matrix:L2} and \eqref{eq:moment:q:bound}.
\end{proof}

\begin{proof}[Proof of Corollary \ref{cor:boot:II:conditional}]
Since $i_2>j_2+2$, we have $\lambda_l(\Psi_{\J,\I})=\lambda_l(\Psi_\J)$ for $l=1,2$ (cf.~the begin of the proof of Lemma \ref{lem:fn:norm:bound:II}). By Weyl's inequality, $\|\cdot\|_{\infty} \leq \|\cdot\|_2$ and Lemma \ref{lem:fn:norm:bound:II}, we get for $l=1,2$,
\begin{align}\label{weyl} \nonumber
\big\vert{\lambda}_{l}(\tilde{\Psi}_{\J,\I}) - \lambda_l(\Psi_\J) \big\vert&=\big\vert{\lambda}_{l}(\tilde{\Psi}_{\J,\I}) - \lambda_l(\Psi_{\J,\I}) \big\vert \\&\leq
\big\|\tilde{\Psi}_{\J,\I} - \Psi_{\J,\I} \big\|_{2}\lesssim_p \frac{\sqrt{\log n}}{\sqrt{n}} A_{\J}
\end{align}
with probability at least $1 - 3n^{1-p/4}$. Invoking  \eqref{ass:boot:truncationset}, we obtain that on this event
\begin{align*}
2\lambda_{l}(\tilde{\Psi}_{\J,\I}) \geq \lambda_{l}(\Psi_\J), \quad l=1,2.
\end{align*}
Next, observe $\E\|\tilde{Y}_{i\I}\|_2^3 \lesssim_q A_{\J,\I}^{3/2}\lesssim_q A_{\J}^{3/2}$ by Lemma \ref{lem:moment:computation:KL}. Combining this with Lemma \ref{lem:fn:norm:bound:II}(iii), we arrive at
\begin{align} \label{eq:cor:boot:II:conditional}
\sum_{i = 1}^{n} \|\tilde{Y}_{i\I}\|_2^3  \lesssim_q n A_{\J,\I}^{3/2} \lesssim_q n A_{\J}^{3/2},
\end{align}
with probability at least $1 - 2n^{1-p/4}$. Proposition \ref{prop:clt:V}, applied to the setting in \eqref{defn:bootstrapII:proof:T:Y:S:Sigma}, then
yields, with probability at least $1 - 2n^{1-p/4}$,
\begin{align*}
\tilde{\mathbf{U}}\Big(\tilde T, \|\tilde{Z}_{\J,\I}\|^2\Big) &\lesssim_p
\frac{1}{n^{1/5}}\Big(\frac{{A}_{\J}}{\sqrt{\lambda_{1,2}(\Psi_\J)}} \Big)^{3/5} + \tilde{\P}\big(\vert \tilde T-\tilde{S}_{\I}\vert > u\big) + \frac{u}{\sqrt{{\lambda}_{1,2}(\Psi_\J)}}.
\end{align*}

By Corollary \ref{lem:boot:approx}, we have
\begin{align*}
    \vert \tilde T-\tilde S_\I\vert &\leq \vert \tilde T-\tilde S\vert +\vert \tilde S-\tilde S_\I\vert \\
    &\leq \frac{n}{\sigma_w^2}\big(\vert \J\vert ^{3/2}(\delta_\J^3(E)+\delta_\J^3(\tilde E))+\vert J\vert ^2(\delta_\J^4(E)+\delta_\J^4(\tilde E))\big)+\vert \tilde S-\tilde S_\I\vert .
\end{align*}
Letting
\begin{align*}
    u = Cn^{-1/2}(\log n)^{3/2} \sigma_{\J}^3\vert \J\vert ^{3/2}+CA_{\J,\I^c}\log n,
\end{align*}
it follows from \eqref{eq:P(vT-Sv:bigger:u)} and Lemma \ref{lem:fn:norm:bound:II} that
\begin{align*}
{\P}\big(\vert \tilde T-\tilde{S}_{\I}\vert > u\big) &\leq {\P}\big(\vert \tilde T-\tilde S\vert > u/2\big) + {\P}\big(\vert \tilde S-\tilde S_\I\vert > u/2\big) \\& \lesssim_p \qprob^{}_{\J,n,p} + n^{1- p/6}.
\end{align*}
By Markov's inequality, we get for $s \in (0,1)$
\begin{align*}
\P\big(\tilde{\P}(\vert \tilde{S}_{\I}-T\vert > u)  \geq (\qprob^{}_{\J,n,p} + n^{1- p/6})^s) \lesssim_p \big(\qprob^{}_{\J,n,p} + n^{1- p/6}\big)^{1-s}.
\end{align*}
Piecing all bounds together, the claim follows from the union bound.
\end{proof}

\begin{lemma}\label{lem:gaussian:comparison:boot:II}
With probability at least $1 - 4n^{1-p/4}$, we have
\begin{align*}
&\tilde{\mathbf{U}}\big(\|L_{\J} Z\|_2, \|\tilde Z_{\J,\I}\|_2 \big) \lesssim_p  \frac{\sqrt{A_{\J}} \operatorname{tr}(\sqrt{\Psi_{\J,\I}})}{\sqrt{\lambda_{1,2}(\Psi)}}\frac{\sqrt{\log n}}{\sqrt{n}}+\frac{A_{\J,\I^c}}{\sqrt{\lambda_{1,2}(\Psi)}}.
\end{align*}
\end{lemma}

\begin{proof}[Proof of Lemma \ref{lem:gaussian:comparison:boot:II}]
By the triangle inequality and independence, we have
\begin{align}\label{eq:uniform:three:Gaussian}
\tilde{\mathbf{U}}\big(\|L_{\J} Z\|_2, \|\tilde Z_{\J,\I}\|_2 \big) \leq \tilde{\mathbf{U}}\big(\|Z_{\J,\I}\|_2, \|\tilde Z_{\J,\I}\|_2 \big) + {\mathbf{U}}\big(\|Z_{\J,\I}\|_2, \|L_{\J} Z\|_2 \big)
\end{align}
We start with the first term on the right-hand side. An application of Lemma~\ref{lem:gaussian:comparison:spok} gives
\begin{align*}
\tilde{\mathbf{U}}\big(\|Z_{\J,\I}\|_2, \|\tilde{Z}_{\J,\I}\|_2 \big) \lesssim \Big(\frac{1}{\sqrt{\lambda_{1,2}(\Psi_{\J,\I})}} + \frac{1}{\sqrt{\lambda_{1,2}(\tilde\Psi_{\J,\I})}} \Big)\big\|\Psi_{\J,\I} - \tilde{\Psi}_{\I}\big\|_1.
\end{align*}
In the proof of Corollary \ref{cor:boot:II:conditional}, we have shown that with probability at least $1-2n^{1-p/4}$, we have $4\lambda_{1,2}(\tilde\Psi_{\J,\I})\geq \lambda_{1,2}(\Psi)$. In order to control $\|\Psi_{\J,\I} - \tilde{\Psi}_{\I}\|_1$, we will apply Lemma \ref{lem:fuknagaev:HS:nuclear:operator}(ii). To this end, by Lemma \ref{lem:moment:computation:KL} (i) and (iv) (note that the $8$-th cumulant uncorrelatedness assumption implies that the $\zeta_{jk}=\eta_j\eta_k$ satisfy the additional conditions from Lemma \ref{lem:moment:computation:KL}, and that the additional weights $(w_i^2-1)/\sigma_w$ do not affect these assumptions), we get
\begin{align*}
(\E \|\tilde{Y}_{i\I}\|_2^p)^{2/p} \lesssim_p A_{\J,\I}
\end{align*}
and
\begin{align*}
\operatorname{tr}\Big(\sqrt{\E (\tilde{Y}_{i\I} \otimes \tilde{Y}_{i\I}-\Psi_\I)^2 } \Big) &\lesssim_p \sqrt{\operatorname{tr}(\Psi_{\J,\I})} \operatorname{tr}(\Psi_{\J,\I}^{1/2})= \operatorname{tr}(\Psi_{\J,\I}^{1/2}) \sqrt{A_{\J,\I}}.
\end{align*}
Similarly, as shown below, for an operator $S$ on $\operatorname{HS}(\mathcal{H})$ with $\|S\|_\infty\leq 1$, we have under the $8$-th cumulant uncorrelatedness assumption that
\begin{align}\label{eq:last:inequality}
 \sqrt{\E \operatorname{tr}^2\big(S (\tilde{Y}_{i\I} \otimes \tilde{Y}_{i\I}-\Psi_\I)\big)}&\lesssim_p  \operatorname{tr}(\Psi_{\J,\I})=A_{\J,\I}.
\end{align}
Hence, by Lemma~\ref{lem:fuknagaev:HS:nuclear:operator}(ii), we have
\begin{align}\label{eq:lem:gaussian:comparison:boot:II:1}
\big\|\Psi_{\J,\I} - \tilde{\Psi}_{\I} \big\|_1 \lesssim_p \sqrt{\log n} \sqrt{A_{\J,\I}} \operatorname{tr}(\Psi_{\J,\I}^{1/2})/\sqrt{n},
\end{align}
with probability at least $1 - 2n^{1-p/4}$. By the union bound and the inequality $A_{\J,\I}\leq A_{\J}$, we conclude that
\begin{align}\label{eq:lem:gaussian:comparison:boot:II:1.5}
\tilde{\mathbf{U}}\big(\|Z_{\J,\I}\|_2, \|\tilde Z_{\J,\I}\|_2 \big) \lesssim_p  \frac{\sqrt{\log n}}{\sqrt{n}} \frac{\sqrt{A_{\J}}\operatorname{tr}(\Psi_{\J,\I}^{1/2})}{\sqrt{\lambda_{1,2}(\Psi)}}
\end{align}
with probability at least $1 - 4n^{1-p/4}$. It remains to consider the second term on the right-hans side of \eqref{eq:uniform:three:Gaussian}. By Lemma \ref{lem:gaussian:comparison:spok}, Lemma \ref{lem:fn:norm:bound:II}(i), and the fact that $\lambda_l(\Psi_{\J,\I})=\lambda_l(\Psi)$, $l=1,2$, we have
\begin{align*}
\mathbf{U}\big(\|L_{\J} Z\|_2, \|Z_{\J,\I}\|_2 \big) &\lesssim \frac{1}{\sqrt{\lambda_{1,2}(\Psi)}} \big\|\Psi - \Psi_{\J,\I} \big\|_{1} \lesssim_p  \frac{A_{\J,\I^c}}{\sqrt{\lambda_{1,2}(\Psi)}}.
\end{align*}
This completes the proof.
\end{proof}

\begin{proof}[Proof of Theorem \ref{thm:boot:II}]
The basic idea is the same as in the proof of Theorem~\ref{thm:boot:I}. By the triangle inequality
\begin{align*}
&\tilde{\mathbf{U}}_{}\Big(\sigma_w^{-2}\big\|\tilde{P}_{\J} - \hat{P}_{\J} \big\|_2^2, \big\|\hat{P}_{\J}' - {P}_{\J} \big\|_2^2 \Big) \\& \leq
\tilde{\mathbf{U}}_{}\Big(n\sigma_w^{-2}\big\|\tilde{P}_{\J} - \hat{P}_{\J} \big\|_2^2, \|\tilde{Z}_{\J,\I}\|^2\Big) + \tilde{\mathbf{U}}_{}\Big(\|\tilde{Z}_{\J,\I}\|_2^2, n\big\|\hat{P}_{\J}' - {P}_{\J} \big\|_2^2\Big).
\end{align*}
For the first term on the left-hand side, we can apply Corollary \ref{cor:boot:II:conditional}. By the triangle inequality and independence
\begin{align} \nonumber \label{eq:thm:boot:II:1}
\tilde{\mathbf{U}}_{}\big(\|\tilde{Z}_{\J,\I}\|_2^2, n\big\|\hat{P}_{\J}' - {P}_{\J} \big\|_2^2\big) &\leq {\mathbf{U}}_{}\big(\|\tilde{Z}_{\J,\I}\|_2^2, \|L_{\J} Z\|_2^2\big) \\&+ {\mathbf{U}}_{}\big(\|L_{\J} Z\|_2^2, n\big\|\hat{P}_{\J}' - {P}_{\J} \big\|_2^2\big).
\end{align}
Lemma \ref{lem:gaussian:comparison:boot:II} deals with ${\mathbf{U}}_{}(\|\tilde{Z}_{\J,\I}\|_2^2, \|L_{\J} Z\|_2^2)$, while Theorem \ref{thm:clt:II} deals with ${\mathbf{U}}_{}(\|L_{\J} Z\|_2^2, n\|\hat{P}_{\J}' - {P}_{\J} \|_2^2)$. Piecing everything together, the claim follows.
\end{proof}

\begin{proof}[Proof of Corollary \ref{cor:boot:II:quant}]
Using Theorems \ref{thm:clt:II} and \ref{thm:boot:II} instead of Theorems \ref{thm:clt:III} and \ref{thm:boot:I}, we may proceed as in the proof of Corollary \ref{cor:boot:I:quant}.
\end{proof}

\begin{proof}[Proof of Equation \eqref{eq:last:inequality}] Using the notation $\mathbf{j}=(j,k)$ for $j\in\J$ and $k\in \J^c\cap \I$ and $\boldsymbol{\J}=\J\times (\J^c\cap\I)$, the Karhunen-Loève expansion of $\tilde Y_{\I}$ is given by
\begin{align*}
    \tilde Y_{\I}=\sum_{\mathbf{j}\in\boldsymbol{\J}}\tilde\vartheta_{\mathbf{j}}^{1/2}\tilde\zeta_{\mathbf{j}}\tilde u_{\mathbf{j}}
\end{align*}
with Karhunen-Loève coefficients
\begin{align*}
\tilde\zeta_{\mathbf{j}} = \frac{w^2-1}{\sigma_w}\frac{\eta_j \eta_k}{\alpha_{jk}^{1/2}}
\end{align*}
and eigenpairs of $\Psi_{\J,\I}=\E \tilde Y_{\I}\otimes \tilde Y_{\I}$ given by
\begin{align*}
    \tilde\vartheta_{\mathbf{j}}=2\alpha_{jk}\frac{\lambda_k \lambda_j}{(\lambda_k - \lambda_j)^2}\quad\text{and}\quad \tilde{u}_{\mathbf{j}}=\frac{\beta_{jk}}{\sqrt{2}}(u_j \otimes u_k+u_k\otimes u_j),\qquad \mathbf{j}=(j,k)\in \boldsymbol{\J}
\end{align*}
with $\alpha_{jk}= \E \eta_j^2\eta_k^2$ and $\beta_{jk}=\operatorname{sgn}(\lambda_k\lambda_j/(\lambda_j-\lambda_k))$. This can be seen by multiplying \eqref{eq:KL:expansion:linear:term:truncated} with $(w^2-1)/\sigma_w$. Here, the uncorrelatedness of the $\tilde\zeta_{\mathbf{j}}$ follows from the $4$-th cumulant uncorrelatedness assumption. Moreover, by Assumptions \ref{ass_moments} and \ref{ass:lower:bound}, we have
\begin{align*}
    c_\eta\leq \alpha_{jk}\leq C_\eta^{2/p}\quad\text{and}\quad  \E \vert \tilde\zeta_{\mathbf{j}}\vert ^p \leq \frac{C_\eta}{c_\eta^{p}} =:\tilde C_{\eta},\qquad  \mathbf{j}=(j,k)\in \boldsymbol{\J}.
\end{align*}
We now write
\begin{align*}
    &\operatorname{tr}^2(S(\tilde Y_{\I}\otimes \tilde Y_{\I}-\E \tilde Y_{\I}\otimes \tilde Y_{\I}))
    = \Big(\sum_{\mathbf{j}_1,\mathbf{j}_2\in\boldsymbol{\J}}\sqrt{\tilde\vartheta_{\mathbf{j}_1}\tilde\vartheta_{\mathbf{j}_2}}\overline{\tilde\zeta_{\mathbf{j}_1}\tilde\zeta_{\mathbf{j}_2}}\langle u_{\mathbf{j}_1},Su_{\mathbf{j}_2}\rangle\Big)^2\\
    &=\sum_{\mathbf{j}_1,\mathbf{j}_2,\mathbf{j}_3,\mathbf{j}_4\in\boldsymbol{\J}}\sqrt{\tilde\vartheta_{\mathbf{j}_1}\tilde\vartheta_{\mathbf{j}_2}}\sqrt{\tilde\vartheta_{\mathbf{j}_3}\tilde\vartheta_{\mathbf{j}_4}}\overline{\tilde\zeta_{\mathbf{j}_1}\tilde\zeta_{\mathbf{j}_2}}\,\overline{\tilde\zeta_{\mathbf{j}_3}\tilde\zeta_{\mathbf{j}_4}}\langle u_{\mathbf{j}_1},Su_{\mathbf{j}_2}\rangle\langle u_{\mathbf{j}_3},Su_{\mathbf{j}_4}\rangle.
\end{align*}
On the one hand, we have $\E \vert \overline{\tilde\zeta_{\mathbf{j}_1}\tilde\zeta_{\mathbf{j}_2}}\,\overline{\tilde\zeta_{\mathbf{j}_3}\tilde\zeta_{\mathbf{j}_4}}\vert \leq \tilde C_{\eta}^{4/p}$ for all indices $\mathbf{j}_1,\mathbf{j}_2,\mathbf{j}_3,\mathbf{j}_4\in \boldsymbol{\J}$. On the other hand, if $\E \overline{\tilde\zeta_{\mathbf{j}_1}\tilde\zeta_{\mathbf{j}_2}}\,\overline{\tilde\zeta_{\mathbf{j}_3}\tilde\zeta_{\mathbf{j}_4}}\neq 0$, then the $8$-th cumulant uncorrelatedness implies that the two sets $\{j_1,j_2,j_3,j_4\}$ and $\{k_1,k_2,k_3,k_4\}$ both contain at most $2$ elements (different indices). Combining these facts with the bound $\vert \langle u_{\mathbf{j}_1},Su_{\mathbf{j}_2}\rangle\langle u_{\mathbf{j}_3},Su_{\mathbf{j}_4}\rangle\vert \leq \|S\|_\infty^2\leq 1$, we conclude that
\begin{align*}
    &\E\operatorname{tr}^2(S(\tilde Y_{\I}\otimes \tilde Y_{\I}-\E \tilde Y_{\I}\otimes \tilde Y_{\I}))\\
    &\leq \tilde C_{\eta}^{4/p}C_\eta^{4/p}\sum_{j_1,j_2\in \J,k_1,k_2\in \J^c\cap \I}\frac{3\lambda_{j_1}\lambda_{j_2}\lambda_{k_1}\lambda_{k_2}}{(\lambda_{j_1}-\lambda_{k_1})^2(\lambda_{j_2}-\lambda_{k_2})^2}\\
    &+\tilde C_{\eta}^{4/p}C_\eta^{4/p}\sum_{j_1,j_2\in \J,k_1,k_2\in \J^c\cap \I}\frac{6\lambda_{j_1}\lambda_{j_2}\lambda_{k_1}\lambda_{k_2}}{\vert \lambda_{j_1}-\lambda_{k_1}\vert \vert \lambda_{j_1}-\lambda_{k_2}\vert \vert \lambda_{j_2}-\lambda_{k_2}\vert \vert \lambda_{j_2}-\lambda_{k_1}\vert }\\
    &\leq 9C_{\eta}^{4/p}C_\eta^{4/p}\Big(\sum_{j\in \J,k\in \J^c\cap \I}\frac{\lambda_{j}\lambda_{k}}{(\lambda_{j}-\lambda_{k})^2}\Big)^2.
\end{align*}
Using again Assumption \ref{ass:lower:bound} and the definition of $\Psi_{\J,\I}$, the claim follows.
\end{proof}





\section*{Acknowledgements}

This research was funded in part by the Austrian Science Fund (FWF), Project I 5485. We would like to thank the reviewers for their careful reading and the thoughtful comments, which we were happy to take into account.


\begin{thebibliography}{53}
\ifx \bisbn   \undefined \def \bisbn  #1{ISBN #1}\fi
\ifx \binits  \undefined \def \binits#1{#1}\fi
\ifx \bauthor  \undefined \def \bauthor#1{#1}\fi
\ifx \batitle  \undefined \def \batitle#1{#1}\fi
\ifx \bjtitle  \undefined \def \bjtitle#1{#1}\fi
\ifx \bvolume  \undefined \def \bvolume#1{\textbf{#1}}\fi
\ifx \byear  \undefined \def \byear#1{#1}\fi
\ifx \bissue  \undefined \def \bissue#1{#1}\fi
\ifx \bfpage  \undefined \def \bfpage#1{#1}\fi
\ifx \blpage  \undefined \def \blpage #1{#1}\fi
\ifx \burl  \undefined \def \burl#1{\textsf{#1}}\fi
\ifx \doiurl  \undefined \def \doiurl#1{\url{https://doi.org/#1}}\fi
\ifx \betal  \undefined \def \betal{\textit{et al.}}\fi
\ifx \binstitute  \undefined \def \binstitute#1{#1}\fi
\ifx \binstitutionaled  \undefined \def \binstitutionaled#1{#1}\fi
\ifx \bctitle  \undefined \def \bctitle#1{#1}\fi
\ifx \beditor  \undefined \def \beditor#1{#1}\fi
\ifx \bpublisher  \undefined \def \bpublisher#1{#1}\fi
\ifx \bbtitle  \undefined \def \bbtitle#1{#1}\fi
\ifx \bedition  \undefined \def \bedition#1{#1}\fi
\ifx \bseriesno  \undefined \def \bseriesno#1{#1}\fi
\ifx \blocation  \undefined \def \blocation#1{#1}\fi
\ifx \bsertitle  \undefined \def \bsertitle#1{#1}\fi
\ifx \bsnm \undefined \def \bsnm#1{#1}\fi
\ifx \bsuffix \undefined \def \bsuffix#1{#1}\fi
\ifx \bparticle \undefined \def \bparticle#1{#1}\fi
\ifx \barticle \undefined \def \barticle#1{#1}\fi
\bibcommenthead
\ifx \bconfdate \undefined \def \bconfdate #1{#1}\fi
\ifx \botherref \undefined \def \botherref #1{#1}\fi
\ifx \url \undefined \def \url#1{\textsf{#1}}\fi
\ifx \bchapter \undefined \def \bchapter#1{#1}\fi
\ifx \bbook \undefined \def \bbook#1{#1}\fi
\ifx \bcomment \undefined \def \bcomment#1{#1}\fi
\ifx \oauthor \undefined \def \oauthor#1{#1}\fi
\ifx \citeauthoryear \undefined \def \citeauthoryear#1{#1}\fi
\ifx \endbibitem  \undefined \def \endbibitem {}\fi
\ifx \bconflocation  \undefined \def \bconflocation#1{#1}\fi
\ifx \arxivurl  \undefined \def \arxivurl#1{\textsf{#1}}\fi
\csname PreBibitemsHook\endcsname

\bibitem[\protect\citeauthoryear{Abdalla and Zhivotovskiy}{2023+}]{ZA}
\begin{botherref}
\oauthor{\bsnm{Abdalla}, \binits{P.}},
\oauthor{\bsnm{Zhivotovskiy}, \binits{N.}}:
Covariance Estimation: Optimal Dimension-free Guarantees for Adversarial
  Corruption and Heavy Tails.
to appear: Journal of the European Mathematical Society
(2023)
\end{botherref}
\endbibitem

\bibitem[\protect\citeauthoryear{Bai}{2003}]{bai:2003:econometrica}
\begin{barticle}
\bauthor{\bsnm{Bai}, \binits{J.}}:
\batitle{Inferential theory for factor models of large dimensions}.
\bjtitle{Econometrica}
\bvolume{71}(\bissue{1}),
\bfpage{135}--\blpage{171}
(\byear{2003})
\doiurl{10.1111/1468-0262.00392}
\end{barticle}
\endbibitem

\bibitem[\protect\citeauthoryear{Bentkus}{1984}]{bentkus1984asymptotic}
\begin{botherref}
\oauthor{\bsnm{Bentkus}, \binits{V.}}:
Asymptotic expansions in the central limit theorem in hilbert space.
Lithuanian Mathematical Journal
(1984)
\end{botherref}
\endbibitem

\bibitem[\protect\citeauthoryear{Barbour and Hall}{1984}]{hall_lower}
\begin{barticle}
\bauthor{\bsnm{Barbour}, \binits{A.D.}},
\bauthor{\bsnm{Hall}, \binits{P.}}:
\batitle{Reversing the berry-esséen inequality}.
\bjtitle{Proceedings of the American Mathematical Society}
\bvolume{90}(\bissue{1}),
\bfpage{107}--\blpage{110}
(\byear{1984})
\end{barticle}
\endbibitem

\bibitem[\protect\citeauthoryear{Birnbaum et~al.}{2013}]{birnbaum_aos_2013}
\begin{barticle}
\bauthor{\bsnm{Birnbaum}, \binits{A.}},
\bauthor{\bsnm{Johnstone}, \binits{I.M.}},
\bauthor{\bsnm{Nadler}, \binits{B.}},
\bauthor{\bsnm{Paul}, \binits{D.}}:
\batitle{{Minimax bounds for sparse {PCA} with noisy high-dimensional data}}.
\bjtitle{The Annals of Statistics}
\bvolume{41}(\bissue{3}),
\bfpage{1055}--\blpage{1084}
(\byear{2013})
\doiurl{10.1214/12-AOS1014}
\end{barticle}
\endbibitem

\bibitem[\protect\citeauthoryear{Bartlett et~al.}{2020}]{bartlett:benign:2020}
\begin{barticle}
\bauthor{\bsnm{Bartlett}, \binits{P.L.}},
\bauthor{\bsnm{Long}, \binits{P.M.}},
\bauthor{\bsnm{Lugosi}, \binits{G.}},
\bauthor{\bsnm{Tsigler}, \binits{A.}}:
\batitle{Benign overfitting in linear regression}.
\bjtitle{Proc. Natl. Acad. Sci. USA}
\bvolume{117}(\bissue{48}),
\bfpage{30063}--\blpage{30070}
(\byear{2020})
\end{barticle}
\endbibitem

\bibitem[\protect\citeauthoryear{Cai et~al.}{2020}]{cai:pan:2020}
\begin{barticle}
\bauthor{\bsnm{Cai}, \binits{T.T.}},
\bauthor{\bsnm{Han}, \binits{X.}},
\bauthor{\bsnm{Pan}, \binits{G.}}:
\batitle{{Limiting laws for divergent spiked eigenvalues and largest nonspiked
  eigenvalue of sample covariance matrices}}.
\bjtitle{The Annals of Statistics}
\bvolume{48}(\bissue{3}),
\bfpage{1255}--\blpage{1280}
(\byear{2020})
\doiurl{10.1214/18-AOS1798}
\end{barticle}
\endbibitem

\bibitem[\protect\citeauthoryear{Cai et~al.}{2013}]{cai_aos_2013}
\begin{barticle}
\bauthor{\bsnm{Cai}, \binits{T.T.}},
\bauthor{\bsnm{Ma}, \binits{Z.}},
\bauthor{\bsnm{Wu}, \binits{Y.}}:
\batitle{{Sparse {PCA}: Optimal rates and adaptive estimation}}.
\bjtitle{The Annals of Statistics}
\bvolume{41}(\bissue{6}),
\bfpage{3074}--\blpage{3110}
(\byear{2013})
\doiurl{10.1214/13-AOS1178}
\end{barticle}
\endbibitem

\bibitem[\protect\citeauthoryear{Cai and Zhang}{2018}]{cai:wedin:aos:2018}
\begin{barticle}
\bauthor{\bsnm{Cai}, \binits{T.T.}},
\bauthor{\bsnm{Zhang}, \binits{A.}}:
\batitle{{Rate-optimal perturbation bounds for singular subspaces with
  applications to high-dimensional statistics}}.
\bjtitle{The Annals of Statistics}
\bvolume{46}(\bissue{1}),
\bfpage{60}--\blpage{89}
(\byear{2018})
\doiurl{10.1214/17-AOS1541}
\end{barticle}
\endbibitem

\bibitem[\protect\citeauthoryear{Dauxois et~al.}{1982}]{Dauxois_1982}
\begin{barticle}
\bauthor{\bsnm{Dauxois}, \binits{J.}},
\bauthor{\bsnm{Pousse}, \binits{A.}},
\bauthor{\bsnm{Romain}, \binits{Y.}}:
\batitle{Asymptotic theory for the principal component analysis of a vector
  random function: some applications to statistical inference}.
\bjtitle{J. Multivariate Anal.}
\bvolume{12}(\bissue{1}),
\bfpage{136}--\blpage{154}
(\byear{1982})
\doiurl{10.1016/0047-259X(82)90088-4}
\end{barticle}
\endbibitem

\bibitem[\protect\citeauthoryear{Einmahl and Li}{2008}]{einmahl:tams:2008}
\begin{barticle}
\bauthor{\bsnm{Einmahl}, \binits{U.}},
\bauthor{\bsnm{Li}, \binits{D.}}:
\batitle{Characterization of {LIL} behavior in {B}anach space}.
\bjtitle{Trans. Amer. Math. Soc.}
\bvolume{360},
\bfpage{6677}--\blpage{6693}
(\byear{2008})
\end{barticle}
\endbibitem

\bibitem[\protect\citeauthoryear{Feller}{1971}]{fellervolume2}
\begin{bbook}
\bauthor{\bsnm{Feller}, \binits{W.}}:
\bbtitle{An Introduction to Probability Theory and Its Applications. {V}ol.
  {II}.}
\bsertitle{Second edition},
p. \bfpage{669}.
\bpublisher{John Wiley \& Sons Inc.},
\blocation{New York}
(\byear{1971})
\end{bbook}
\endbibitem

\bibitem[\protect\citeauthoryear{Fan et~al.}{2021}]{fan:survey:2021}
\begin{barticle}
\bauthor{\bsnm{Fan}, \binits{J.}},
\bauthor{\bsnm{Li}, \binits{K.}},
\bauthor{\bsnm{Liao}, \binits{Y.}}:
\batitle{Recent developments in factor models and applications in econometric
  learning}.
\bjtitle{Annual Review of Financial Economics}
\bvolume{13}(\bissue{1}),
\bfpage{401}--\blpage{430}
(\byear{2021})
\doiurl{10.1146/annurev-financial-091420-011735}
\end{barticle}
\endbibitem

\bibitem[\protect\citeauthoryear{Fischer and
  Steinwart}{2020}]{steinwart:jml:2020}
\begin{barticle}
\bauthor{\bsnm{Fischer}, \binits{S.}},
\bauthor{\bsnm{Steinwart}, \binits{I.}}:
\batitle{Sobolev norm learning rates for regularized least-squares algorithms}.
\bjtitle{J. Mach. Learn. Res.}
\bvolume{21},
\bfpage{205}--\blpage{38}
(\byear{2020})
\end{barticle}
\endbibitem

\bibitem[\protect\citeauthoryear{Hall}{1992}]{hall:bootstrap:book}
\begin{bbook}
\bauthor{\bsnm{Hall}, \binits{P.}}:
\bbtitle{The Bootstrap and {E}dgeworth Expansion}.
\bpublisher{Springer},
\blocation{New York}
(\byear{1992}).
\doiurl{10.1007/978-1-4612-4384-7}
\end{bbook}
\endbibitem

\bibitem[\protect\citeauthoryear{Hsing and
  Eubank}{2015}]{hsing:eubank:book:2015}
\begin{bbook}
\bauthor{\bsnm{Hsing}, \binits{T.}},
\bauthor{\bsnm{Eubank}, \binits{R.}}:
\bbtitle{Theoretical Foundations of Functional Data Analysis, with an
  Introduction to Linear Operators}.
\bsertitle{Wiley Series in Probability and Statistics},
p. \bfpage{334}.
\bpublisher{John Wiley \& Sons, Ltd.},
\blocation{Chichester}
(\byear{2015}).
\doiurl{10.1002/9781118762547} .
\burl{https://doi.org/10.1002/9781118762547}
\end{bbook}
\endbibitem

\bibitem[\protect\citeauthoryear{Hall and Horowitz}{2007}]{hall2007}
\begin{barticle}
\bauthor{\bsnm{Hall}, \binits{P.}},
\bauthor{\bsnm{Horowitz}, \binits{J.L.}}:
\batitle{Methodology and convergence rates for functional linear regression}.
\bjtitle{The Annals of Statistics}
\bvolume{35}(\bissue{1}),
\bfpage{70}--\blpage{91}
(\byear{2007})
\end{barticle}
\endbibitem

\bibitem[\protect\citeauthoryear{Horv\'{a}th and
  Kokoszka}{2012}]{horvath:kokoszka:book:2012}
\begin{bbook}
\bauthor{\bsnm{Horv\'{a}th}, \binits{L.}},
\bauthor{\bsnm{Kokoszka}, \binits{P.}}:
\bbtitle{Inference for Functional Data with Applications}.
\bpublisher{Springer},
\blocation{New York}
(\byear{2012}).
\doiurl{10.1007/978-1-4614-3655-3}
\end{bbook}
\endbibitem

\bibitem[\protect\citeauthoryear{Hörmann et~al.}{2015}]{siegi:2015}
\begin{barticle}
\bauthor{\bsnm{Hörmann}, \binits{S.}},
\bauthor{\bsnm{Kidzinski}, \binits{L.}},
\bauthor{\bsnm{Hallin}, \binits{M.}}:
\batitle{Dynamic functional principal components}.
\bjtitle{Journal of the Royal Statistical Society. Series B (Statistical
  Methodology)}
\bvolume{77}(\bissue{2}),
\bfpage{319}--\blpage{348}
(\byear{2015})
\end{barticle}
\endbibitem

\bibitem[\protect\citeauthoryear{Jiang and Bai}{2021}]{bai:BJ:2021}
\begin{barticle}
\bauthor{\bsnm{Jiang}, \binits{D.}},
\bauthor{\bsnm{Bai}, \binits{Z.}}:
\batitle{{Generalized four moment theorem and an application to CLT for spiked
  eigenvalues of high-dimensional covariance matrices}}.
\bjtitle{Bernoulli}
\bvolume{27}(\bissue{1}),
\bfpage{274}--\blpage{294}
(\byear{2021})
\doiurl{10.3150/20-BEJ1237}
\end{barticle}
\endbibitem

\bibitem[\protect\citeauthoryear{Jirak}{2016}]{jirak_eigen_expansions_2016}
\begin{barticle}
\bauthor{\bsnm{Jirak}, \binits{M.}}:
\batitle{Optimal eigen expansions and uniform bounds}.
\bjtitle{Probab. Theory Related Fields}
\bvolume{166}(\bissue{3-4}),
\bfpage{753}--\blpage{799}
(\byear{2016})
\doiurl{10.1007/s00440-015-0671-3}
\end{barticle}
\endbibitem

\bibitem[\protect\citeauthoryear{Johnstone}{2001}]{johnstone:spiked:2001}
\begin{barticle}
\bauthor{\bsnm{Johnstone}, \binits{I.M.}}:
\batitle{{On the distribution of the largest eigenvalue in principal components
  analysis}}.
\bjtitle{The Annals of Statistics}
\bvolume{29}(\bissue{2}),
\bfpage{295}--\blpage{327}
(\byear{2001})
\doiurl{10.1214/aos/1009210544}
\end{barticle}
\endbibitem

\bibitem[\protect\citeauthoryear{Jolliffe}{2022}]{jolliffe:jmva:2022}
\begin{barticle}
\bauthor{\bsnm{Jolliffe}, \binits{I.}}:
\batitle{A 50-year personal journey through time with principal component
  analysis}.
\bjtitle{J. Multivariate Anal.}
\bvolume{188},
\bfpage{104820}--\blpage{7}
(\byear{2022})
\doiurl{10.1016/j.jmva.2021.104820}
\end{barticle}
\endbibitem

\bibitem[\protect\citeauthoryear{Johnstone. and
  Paul}{2018}]{johnstone:paul:2018}
\begin{barticle}
\bauthor{\bsnm{Johnstone.}, \binits{I.M.}},
\bauthor{\bsnm{Paul}, \binits{D.}}:
\batitle{{PCA} in high dimensions: An orientation}.
\bjtitle{Proceedings of the IEEE}
\bvolume{106}(\bissue{8}),
\bfpage{1277}--\blpage{1292}
(\byear{2018})
\doiurl{10.1109/JPROC.2018.2846730}
\end{barticle}
\endbibitem

\bibitem[\protect\citeauthoryear{Jirak and Wahl}{2020}]{jirak:wahl:pams:2020}
\begin{barticle}
\bauthor{\bsnm{Jirak}, \binits{M.}},
\bauthor{\bsnm{Wahl}, \binits{M.}}:
\batitle{Perturbation bounds for eigenspaces under a relative gap condition}.
\bjtitle{Proc. Amer. Math. Soc.}
\bvolume{148}(\bissue{2}),
\bfpage{479}--\blpage{494}
(\byear{2020})
\doiurl{10.1090/proc/14714}
\end{barticle}
\endbibitem

\bibitem[\protect\citeauthoryear{Jirak and Wahl}{2023}]{JW18}
\begin{barticle}
\bauthor{\bsnm{Jirak}, \binits{M.}},
\bauthor{\bsnm{Wahl}, \binits{M.}}:
\batitle{Relative perturbation bounds with applications to empirical covariance
  operators}.
\bjtitle{Adv. Math.}
\bvolume{412},
\bfpage{108808}--\blpage{59}
(\byear{2023})
\doiurl{10.1016/j.aim.2022.108808}
\end{barticle}
\endbibitem

\bibitem[\protect\citeauthoryear{Koltchinskii and
  Lounici}{2016}]{koltchinskii:lounici:AHIP:2016}
\begin{barticle}
\bauthor{\bsnm{Koltchinskii}, \binits{V.}},
\bauthor{\bsnm{Lounici}, \binits{K.}}:
\batitle{Asymptotics and concentration bounds for bilinear forms of spectral
  projectors of sample covariance}.
\bjtitle{Ann. Inst. Henri Poincar\'{e} Probab. Stat.}
\bvolume{52}(\bissue{4}),
\bfpage{1976}--\blpage{2013}
(\byear{2016})
\doiurl{10.1214/15-AIHP705}
\end{barticle}
\endbibitem

\bibitem[\protect\citeauthoryear{Koltchinskii and
  Lounici}{2017a}]{koltchinskii:lounici:BJ:2017}
\begin{barticle}
\bauthor{\bsnm{Koltchinskii}, \binits{V.}},
\bauthor{\bsnm{Lounici}, \binits{K.}}:
\batitle{Concentration inequalities and moment bounds for sample covariance
  operators}.
\bjtitle{Bernoulli}
\bvolume{23}(\bissue{1}),
\bfpage{110}--\blpage{133}
(\byear{2017})
\doiurl{10.3150/15-BEJ730}
\end{barticle}
\endbibitem

\bibitem[\protect\citeauthoryear{Koltchinskii and
  Lounici}{2017b}]{koltchinskii2017}
\begin{barticle}
\bauthor{\bsnm{Koltchinskii}, \binits{V.}},
\bauthor{\bsnm{Lounici}, \binits{K.}}:
\batitle{Normal approximation and concentration of spectral projectors of
  sample covariance}.
\bjtitle{Ann. Statist.}
\bvolume{45}(\bissue{1}),
\bfpage{121}--\blpage{157}
(\byear{2017})
\doiurl{10.1214/16-AOS1437}
\end{barticle}
\endbibitem

\bibitem[\protect\citeauthoryear{Koltchinskii
  et~al.}{2020}]{koltchinski:nickl:loeffler:2020}
\begin{barticle}
\bauthor{\bsnm{Koltchinskii}, \binits{V.}},
\bauthor{\bsnm{L\"offler}, \binits{M.}},
\bauthor{\bsnm{Nickl}, \binits{R.}}:
\batitle{{Efficient estimation of linear functionals of principal components}}.
\bjtitle{The Annals of Statistics}
\bvolume{48}(\bissue{1}),
\bfpage{464}--\blpage{490}
(\byear{2020})
\end{barticle}
\endbibitem

\bibitem[\protect\citeauthoryear{Koltchinskii}{2021}]{koltchinskii:jems:2021}
\begin{barticle}
\bauthor{\bsnm{Koltchinskii}, \binits{V.}}:
\batitle{Asymptotically efficient estimation of smooth functionals of
  covariance operators}.
\bjtitle{J. Eur. Math. Soc. (JEMS)}
\bvolume{23}(\bissue{3}),
\bfpage{765}--\blpage{843}
(\byear{2021})
\doiurl{10.4171/jems/1023}
\end{barticle}
\endbibitem

\bibitem[\protect\citeauthoryear{L\"offler}{2019}]{loeffler:2019}
\begin{barticle}
\bauthor{\bsnm{L\"offler}, \binits{M.}}:
\batitle{Wald statistics in high-dimensional {PCA}}.
\bjtitle{ESAIM: PS}
\bvolume{23},
\bfpage{662}--\blpage{671}
(\byear{2019})
\doiurl{10.1051/ps/2019002}
\end{barticle}
\endbibitem

\bibitem[\protect\citeauthoryear{Lopes et~al.}{2019}]{miles:blandiono:aue:2019}
\begin{barticle}
\bauthor{\bsnm{Lopes}, \binits{M.E.}},
\bauthor{\bsnm{Blandino}, \binits{A.}},
\bauthor{\bsnm{Aue}, \binits{A.}}:
\batitle{{Bootstrapping spectral statistics in high dimensions}}.
\bjtitle{Biometrika}
\bvolume{106}(\bissue{4}),
\bfpage{781}--\blpage{801}
(\byear{2019})
\doiurl{10.1093/biomet/asz040}
\end{barticle}
\endbibitem

\bibitem[\protect\citeauthoryear{Liu et~al.}{2023}]{bai:LSS:arxiv}
\begin{barticle}
\bauthor{\bsnm{Liu}, \binits{Z.}},
\bauthor{\bsnm{Hu}, \binits{J.}},
\bauthor{\bsnm{Bai}, \binits{Z.}},
\bauthor{\bsnm{Song}, \binits{H.}}:
\batitle{A {CLT} for the {LSS} of large-dimensional sample covariance matrices
  with diverging spikes}.
\bjtitle{Ann. Statist.}
\bvolume{51}(\bissue{5}),
\bfpage{2246}--\blpage{2271}
(\byear{2023})
\doiurl{10.1214/23-aos2333}
\end{barticle}
\endbibitem

\bibitem[\protect\citeauthoryear{Lopes}{2023+}]{Lopes}
\begin{botherref}
\oauthor{\bsnm{Lopes}, \binits{M.E.}}:
Improved rates of bootstrap approximation for the operator norm: A
  coordinate-free approach.
to appear: Annales de l'Institut Henri Poincaré (B) Probabilités et
  Statistiques, Available at https://arxiv.org/abs/2208.03050
(2023)
\end{botherref}
\endbibitem

\bibitem[\protect\citeauthoryear{Minsker}{2018}]{minsker17:aos}
\begin{barticle}
\bauthor{\bsnm{Minsker}, \binits{S.}}:
\batitle{{Sub-Gaussian estimators of the mean of a random matrix with
  heavy-tailed entries}}.
\bjtitle{The Annals of Statistics}
\bvolume{46}(\bissue{6A}),
\bfpage{2871}--\blpage{2903}
(\byear{2018})
\doiurl{10.1214/17-AOS1642}
\end{barticle}
\endbibitem

\bibitem[\protect\citeauthoryear{Mendelson and
  Zhivotovskiy}{2020}]{mendelson:nikita:aos:2020}
\begin{barticle}
\bauthor{\bsnm{Mendelson}, \binits{S.}},
\bauthor{\bsnm{Zhivotovskiy}, \binits{N.}}:
\batitle{{Robust covariance estimation under $L_{4}-L_{2}$ norm equivalence}}.
\bjtitle{The Annals of Statistics}
\bvolume{48}(\bissue{3}),
\bfpage{1648}--\blpage{1664}
(\byear{2020})
\doiurl{10.1214/19-AOS1862}
\end{barticle}
\endbibitem

\bibitem[\protect\citeauthoryear{Nagaev}{1979}]{nagaev}
\begin{botherref}
\oauthor{\bsnm{Nagaev}, \binits{S.V.}}:
Large deviations of sums of independent random variables.
Ann. Probab.
(1979)
\end{botherref}
\endbibitem

\bibitem[\protect\citeauthoryear{Naumov
  et~al.}{2019}]{naumov_spokoiny_ptrf2019}
\begin{barticle}
\bauthor{\bsnm{Naumov}, \binits{A.}},
\bauthor{\bsnm{Spokoiny}, \binits{V.}},
\bauthor{\bsnm{Ulyanov}, \binits{V.}}:
\batitle{Bootstrap confidence sets for spectral projectors of sample
  covariance}.
\bjtitle{Probab. Theory Related Fields}
\bvolume{174}(\bissue{3-4}),
\bfpage{1091}--\blpage{1132}
(\byear{2019})
\doiurl{10.1007/s00440-018-0877-2}
\end{barticle}
\endbibitem

\bibitem[\protect\citeauthoryear{Petrov}{1995}]{petrov_book_1995}
\begin{bbook}
\bauthor{\bsnm{Petrov}, \binits{V.V.}}:
\bbtitle{Limit Theorems of Probability Theory}.
\bsertitle{Oxford Studies in Probability},
vol. \bseriesno{4},
p. \bfpage{292}.
\bpublisher{The Clarendon Press Oxford University Press},
\blocation{New York}
(\byear{1995}).
\bcomment{Sequences of independent random variables, Oxford Science
  Publications}
\end{bbook}
\endbibitem

\bibitem[\protect\citeauthoryear{Pinelis}{1994}]{pinelis:1994}
\begin{barticle}
\bauthor{\bsnm{Pinelis}, \binits{I.}}:
\batitle{{Optimum Bounds for the Distributions of Martingales in Banach
  Spaces}}.
\bjtitle{The Annals of Probability}
\bvolume{22}(\bissue{4}),
\bfpage{1679}--\blpage{1706}
(\byear{1994})
\doiurl{10.1214/aop/1176988477}
\end{barticle}
\endbibitem

\bibitem[\protect\citeauthoryear{Panaretos and
  Tavakoli}{2013}]{PANARETOS20132779}
\begin{barticle}
\bauthor{\bsnm{Panaretos}, \binits{V.M.}},
\bauthor{\bsnm{Tavakoli}, \binits{S.}}:
\batitle{Cramér–{K}arhunen–{L}oève representation and harmonic principal
  component analysis of functional time series}.
\bjtitle{Stochastic Processes and their Applications}
\bvolume{123}(\bissue{7}),
\bfpage{2779}--\blpage{2807}
(\byear{2013})
\doiurl{10.1016/j.spa.2013.03.015} .
\bcomment{A Special Issue on the Occasion of the 2013 International Year of
  Statistics}
\end{barticle}
\endbibitem

\bibitem[\protect\citeauthoryear{Reiss and Wahl}{2020}]{reiss:wahl:aos:2020}
\begin{barticle}
\bauthor{\bsnm{Reiss}, \binits{M.}},
\bauthor{\bsnm{Wahl}, \binits{M.}}:
\batitle{Nonasymptotic upper bounds for the reconstruction error of {PCA}}.
\bjtitle{Ann. Statist.}
\bvolume{48}(\bissue{2}),
\bfpage{1098}--\blpage{1123}
(\byear{2020})
\doiurl{10.1214/19-AOS1839}
\end{barticle}
\endbibitem

\bibitem[\protect\citeauthoryear{Senatov}{1998}]{senatov:book}
\begin{bbook}
\bauthor{\bsnm{Senatov}, \binits{V.V.}}:
\bbtitle{Normal Approximation: New Results, Methods and Problems}.
\bpublisher{VSP},
\blocation{Utrecht}
(\byear{1998}).
\doiurl{10.1515/9783110933666.363} .
\bcomment{Translated from the Russian manuscript by A. V. Kolchin}
\end{bbook}
\endbibitem

\bibitem[\protect\citeauthoryear{Scholkopf and Smola}{2001}]{scholkopf}
\begin{bbook}
\bauthor{\bsnm{Scholkopf}, \binits{B.}},
\bauthor{\bsnm{Smola}, \binits{A.J.}}:
\bbtitle{Learning with Kernels: Support Vector Machines, Regularization,
  Optimization, and Beyond}.
\bpublisher{MIT Press},
\blocation{Cambridge, MA, USA}
(\byear{2001})
\end{bbook}
\endbibitem

\bibitem[\protect\citeauthoryear{Sazonov et~al.}{1989}]{SAZONOV1989304}
\begin{barticle}
\bauthor{\bsnm{Sazonov}, \binits{V.V.}},
\bauthor{\bsnm{Ulyanov}, \binits{V.V.}},
\bauthor{\bsnm{Zalesskii}, \binits{B.A.}}:
\batitle{Asymptotically precise estimate of the accuracy of gaussian
  approximation in hilbert space}.
\bjtitle{Journal of Multivariate Analysis}
\bvolume{28}(\bissue{2}),
\bfpage{304}--\blpage{330}
(\byear{1989})
\end{barticle}
\endbibitem

\bibitem[\protect\citeauthoryear{Stock and Watson}{2002}]{stock:watson:2002}
\begin{barticle}
\bauthor{\bsnm{Stock}, \binits{J.H.}},
\bauthor{\bsnm{Watson}, \binits{M.W.}}:
\batitle{Forecasting using principal components from a large number of
  predictors}.
\bjtitle{J. Amer. Statist. Assoc.}
\bvolume{97}(\bissue{460}),
\bfpage{1167}--\blpage{1179}
(\byear{2002})
\doiurl{10.1198/016214502388618960}
\end{barticle}
\endbibitem

\bibitem[\protect\citeauthoryear{Ul’yanov}{1987}]{ulyanov1987asymptotic}
\begin{botherref}
\oauthor{\bsnm{Ul’yanov}, \binits{V.V.}}:
Asymptotic expansions for distributions of sums of independent random variables
  in $h$.
Theory of Probability and Its Applications
(1987)
\end{botherref}
\endbibitem

\bibitem[\protect\citeauthoryear{Wahl}{2019}]{Wahl}
\begin{botherref}
\oauthor{\bsnm{Wahl}, \binits{M.}}:
On the perturbation series for eigenvalues and eigenprojections.
Available at https://arxiv.org/abs/1910.08460
(2019)
\end{botherref}
\endbibitem

\bibitem[\protect\citeauthoryear{Wei and Minsker}{2017}]{minsker:nips:2017}
\begin{bchapter}
\bauthor{\bsnm{Wei}, \binits{X.}},
\bauthor{\bsnm{Minsker}, \binits{S.}}:
\bctitle{Estimation of the covariance structure of heavy-tailed distributions}.
In: \beditor{\bsnm{Guyon}, \binits{I.}},
\beditor{\bsnm{Luxburg}, \binits{U.V.}},
\beditor{\bsnm{Bengio}, \binits{S.}},
\beditor{\bsnm{Wallach}, \binits{H.}},
\beditor{\bsnm{Fergus}, \binits{R.}},
\beditor{\bsnm{Vishwanathan}, \binits{S.}},
\beditor{\bsnm{Garnett}, \binits{R.}} (eds.)
\bbtitle{Advances in Neural Information Processing Systems},
vol. \bseriesno{30}.
\bpublisher{Curran Associates},
\blocation{Inc.}
(\byear{2017})
\end{bchapter}
\endbibitem

\bibitem[\protect\citeauthoryear{Yao and Lopes}{2023}]{miles:2022+}
\begin{barticle}
\bauthor{\bsnm{Yao}, \binits{J.}},
\bauthor{\bsnm{Lopes}, \binits{M.E.}}:
\batitle{Rates of bootstrap approximation for eigenvalues in high-dimensional
  {PCA}}.
\bjtitle{Statist. Sinica}
\bvolume{33}(\bissue{Special online issue}),
\bfpage{1461}--\blpage{1481}
(\byear{2023})
\doiurl{10.5705/ss.202020.0457}
\end{barticle}
\endbibitem

\bibitem[\protect\citeauthoryear{Yu et~al.}{2015}]{samworth:pca:2015}
\begin{barticle}
\bauthor{\bsnm{Yu}, \binits{Y.}},
\bauthor{\bsnm{Wang}, \binits{T.}},
\bauthor{\bsnm{Samworth}, \binits{R.J.}}:
\batitle{A useful variant of the {D}avis—{K}ahan theorem for statisticians}.
\bjtitle{Biometrika}
\bvolume{102}(\bissue{2}),
\bfpage{315}--\blpage{323}
(\byear{2015})
\end{barticle}
\endbibitem

\bibitem[\protect\citeauthoryear{Zhivotovskiy}{2024}]{Z}
\begin{barticle}
\bauthor{\bsnm{Zhivotovskiy}, \binits{N.}}:
\batitle{Dimension-free bounds for sums of independent matrices and simple
  tensors via the variational principle}.
\bjtitle{Electron. J. Probab.}
\bvolume{29},
\bfpage{1}
(\byear{2024})
\doiurl{10.1214/23-ejp1021}
\end{barticle}
\endbibitem

\end{thebibliography}


\end{document}